\documentclass[reqno,11pt]{amsart}
\usepackage[top=2.0cm,bottom=2.0cm,left=3cm,right=3cm]{geometry}
\usepackage{amsthm,amsmath,amssymb,dsfont}
\usepackage{mathrsfs,amsfonts,functan,extarrows,mathtools}

\usepackage{xcolor}
\usepackage{cite}
\usepackage{indentfirst, latexsym, amssymb, enumerate,amsmath,graphicx}
\usepackage{float}
\usepackage{relsize}
\usepackage{marginnote}
\usepackage{stmaryrd}
\usepackage{esint}
\usepackage{graphicx}
\usepackage{bm}

\usepackage{cite}
\allowdisplaybreaks[4]
\theoremstyle{plain}
\newtheorem{thm}{Theorem}[section]

\newtheorem{lem}[thm]{Lemma}
\newtheorem{prop}[thm]{Proposition}
\newtheorem{rem}{Remark}[section]

\numberwithin{equation}{section}

\begin{document}

\title[The compressible NSF-$P_1$ approximation model]
{Global well-posedness and optimal decay rates of classical solutions
to the compressible Navier-Stokes-Fourier-$P_1$ approximation model in
radiation hydrodynamics}

\author[P. Jiang]{Peng Jiang}
\address{School of Mathematics, Hohai University, Nanjing
 210098, P. R. China}
\email{syepmathjp@163.com}

\author[F.-C. Li]{fucai Li}
\address{Department of Mathematics, Nanjing University, Nanjing
 210093, P. R. China}
\email{fli@nju.edu.cn}

\author[J.-K. Ni]{Jinkai Ni$^*$} \thanks{$^*$\! Corresponding author}
\address{Department of Mathematics, Nanjing University, Nanjing
 210093, P. R. China}
\email{602023210006@smail.nju.edu.cn}

\begin{abstract}
In this paper, the compressible Navier-Stokes-Fourier-$P_1$ (NSF-$P_1$) approximation
model in radiation hydrodynamics is investigated in the whole space $\mathbb{R}^3$. 
This model consists of the compressible NSF equations of ﬂuid coupled with the transport equations of the radiation field propagation.
Assuming that the initial data are a small perturbation near
the equilibrium state,
we establish the global well-posedness of classical solutions for this model
by performing the Fourier analysis 
techniques and employing
the delicate energy estimates in frequency spaces. 
Here, we develop a new method to overcome a series of difficulties arising from the linear terms $n_1$ in \eqref{I-3}$_2$ and $n_0$ in \eqref{I-3}$_3$ related to the radiation intensity.
Furthermore, if the $L^1$-norm of the initial data is bounded, we obtain the optimal time decay rates of the classical solution
at $L^p$-norm $(2\leq p\leq \infty)$.
To the best of our knowledge, this is the first result on the global well-posedness of the NSF-$P_1$ approximation model. 
\end{abstract}



\maketitle

\setcounter{equation}{0}
 \indent \allowdisplaybreaks

\section{Introduction and main result}\label{Sec:intro-resul}

 \subsection{Previous literature and our model}
The purpose of radiation hydrodynamics is to establish hydrodynamical framework from the macroscopic analysis, which takes account of the effects of radiation (see, for example, 
\cite{Cs-1960,Pgc-2005}). It is widely used in many fields, such as stellar evolution \cite{KW-book-1994},  
laser fusion \cite{LLH-CPC-2021} and vehicle reentry \cite{PO-book-1968}.
Therefore, it is very important to study radiation hydrodynamics from the perspective of mathematical theory for further understanding of radiation hydrodynamics and related physical phenomena and their applications.

The interaction between the radiation field and the fluid can be described by the absorption, scattering and emission of photons, and the general radiation hydrodynamics equations can be derived by considering all the interactions mentioned above. This is a nonlinear differential-integral system coupled by the compressible Navier-Stokes-Fourier (NSF) equations and the photon transport equation, and the three-dimensional case can be written as  (see \cite{M-M})
\begin{equation}\label{NJK-1}
\left\{
\begin{aligned}
& \partial_t \varrho+{\rm div}(\varrho u)=0, \\
&\partial_t\Big(\varrho u+\frac{1}{c^2}F_r\Big)+{\rm div}(\varrho u\otimes u+P\mathbb{I}_3+P_r)={\rm div}\Psi(u), \\
& \partial_t (\varrho E+E_r)+{\rm div}\big(\varrho u(E+P)+F_r\big)={\rm div}\big(\Psi(u)\cdot u  \big)+\mathrm{div}(\kappa\nabla\Theta), \\
& \frac{1}{c}\partial_t I+\omega\cdot\nabla I=Q,
\end{aligned}\right.
\end{equation}
where $\varrho=\varrho(t,x)>0$,\ $\Theta=\Theta(t,x)>0$, and $u=u(t,x)=(u_1,u_2,u_3)\in\mathbb{R}^3$ for $t\geq 0$, $x\in\mathbb{R}^{3}$
denote the density, temperature, and velocity field of the fluid
respectively. And $P\geq 0$ is the pressure, and $E=e+\frac{|u|^2}{2}$ is the total energy where
$e$ denotes the internal energy.  For the prefect gas case, $P=R\varrho\Theta$ with $e=c_\textsl{v}\Theta$, 
  and $R>0$ is the perfect gas constant and $c_\textsl{v}>0$ is the heat capacity at constant volume. $\Psi(u)$ is the viscous stress tensor satisfying
\begin{align*}
 \Psi(u):=2\mu D +\lambda {\rm div}u \mathbb{I}_3, \quad \text{with} \quad D=D(u):=\frac{ \nabla u+(\nabla u)^{\top} }{2},
\end{align*}
where $D$ is the deformation tensor,  $\mathbb{I}_3$ denotes the $3\times 3$  
identity matrix, and $\lambda$ and $\mu$ are the viscosity coefficients 
  satisfying $\mu>0$ and $3\lambda+2\mu\geq 0$.  And $\kappa>0$ is the heat conductivity coefficient. For simplicity, in this paper we assume that $c_\textsl{v}$, $\mu$, $\lambda$, $\kappa$ are constants.

The radiation effect terms in the  NSF equations  \eqref{NJK-1}$_{1}$--\eqref{NJK-1}$_{3}$ are radiation flux $F_{r}$, radiation pressure $P_{r}$, and radiation energy $E_{r}$, respectively, which are defined as  
\begin{equation}\label{I8}
\left\{
\begin{aligned}
F_r=&\,\int_0^{\infty} \mathrm{d}\nu \int_{\mathbb{S}^{2}}\omega I(\nu,\omega) \mathrm{d}  \omega, \\
P_r=&\,\frac{1}{c}\int_0^{\infty} \mathrm{d}\nu \int_{\mathbb{S}^{2}}\omega\otimes\omega I(\nu,\omega) \mathrm{d}  \omega,\\
E_r=&\,\frac{1}{c}\int_0^{\infty} \mathrm{d}\nu \int_{\mathbb{S}^{2}}I(\nu,\omega) \mathrm{d}  \omega.
\end{aligned}\right.
\end{equation}
Here, $I=I(t,x,\nu,\omega)$ represents the radiation intensity, depending on the direction vector $\omega\in \mathbb{S}^{2}$ and the frequency $\nu\geq 0$, and  $c$ is the light speed. $I$ satisfy the photon transport equation $\eqref{NJK-1}_{4}$, which describes the propagation rules of the radiation field. The source term $Q$ is given by
\begin{equation}\label{I9}
Q=S(\nu)-\sigma_{a}(\nu)I+\int_{0}^{\infty}\mathrm{d}\nu^{\prime}
\int_{\mathbb{S}^{2}}\left[\frac{\nu}{\nu^{\prime}}\sigma_{s}(\nu^{\prime}\rightarrow\nu)I(\nu^{\prime}, \omega^{\prime})-\sigma_{s}(\nu\rightarrow\nu^{\prime})I(\nu, \omega)\right]\mathrm{d}\omega^{\prime},
\end{equation}
where $S(\nu)=S(t,x,\nu,\Omega,\varrho,\Theta)$
is the rate of energy emission due to spontaneous processes,
$\sigma_a (\nu)=\sigma_a (t,x,\nu,\Omega,\varrho,\Theta)$ is the
absorption coefficient, and $\sigma_{s}(\cdot)=\sigma_s(t,x,\cdot,\Omega'\cdot\Omega,\varrho,\Theta)$ is the scattering coefficient.

Due its complexity, there are few mathematical results on the general radiation hydrodynamics equations \eqref{NJK-1}--\eqref{I9}. The global weak solution of \eqref{NJK-1}--\eqref{I9} in bounded domain was established in \cite{B-F-N}, and the singular limits for the weak solution were presented in \cite{B-N1, B-N2}. In the case of inviscid flow for \eqref{NJK-1}--\eqref{I9}, the existence of local smooth solutions to the Cauchy problem and initial boundary value problem and ﬁnite-time formation of singularities were considered in \cite{J-W1, J-W2, J-Z1, L-Z}.

Therefore, in the study of radiation hydrodynamics, it is an important method to simplify or approximate the equations of general form under the premise of guaranteeing the physical significance. 
In this paper, we work with the ``gray" approximation and spherical harmonic expansion, whose purpose is to reduce the radiation intensity $I$ in \eqref{NJK-1} to a function dependent only on $x$ and $t$. That means the absorption and scattering coefficients $\sigma_a$, $\sigma_s$ are independent of the frequency $\nu$. And for the sake of simplicity, we can further assume that the two coefficients are positive constants. Meanwhile, the energy emission term $S(\nu)$ can be taken as the well-known Planck function, deﬁned by 
\begin{equation}\label{I10}
 S(\nu)=\frac{2\hbar\nu^3}{c^{2}}(e^{\frac{\hbar\nu}{\varkappa\Theta}}-1)^{-1},
\end{equation}
 where $\hbar $ and $\varkappa$ are the  Planck and Boltzmann constants,  respectively.
Then, integrating $\eqref{NJK-1}_{4}$ with respect to $\nu$ yields
\begin{equation}\label{I11}
\frac{1}{c}\partial_t I+\omega\cdot\nabla I=\tilde{C}\Theta^{4}-\sigma_a I+4\pi\sigma_s\left(\frac{1}{4\pi}\int_{\mathbb{S}^{2}}I(t,x,\omega)\mathrm{d}\omega-I\right),
\end{equation}
where the integration of $I$ with respect to $\nu$ is still represented as $I$ and $\tilde{C}$ is a positive constant from the result of integration of \eqref{I10}.

Next, the radiation intensity $I(t,x,\omega)$ in \eqref{I11} can be further represented
by the ﬁrst two terms in a spherical harmonic expansion as  
\begin{equation}\label{I12}
I(t,x,\omega)=\frac{1}{4\pi}I_{0}(t,x)+\frac{3}{4\pi}\omega\cdot I_{1}(t,x).
\end{equation}
Integrating   \eqref{I12} over all solid angle gives
\begin{equation*}
I_{0}(t,x):=\int_{\mathbb{S}^{2}}I(t,x,\omega)\mathrm{d}\omega,
\end{equation*}
and multiplying   \eqref{I12} by $\omega$  and then taking  a similar  integration yield
\begin{equation*}
I_{1}(t,x):=\int_{\mathbb{S}^{2}}\omega I(t,x,\omega)\mathrm{d}\omega.
\end{equation*}
In \eqref{I12}, $\omega\cdot I_{1}$ is regarded as a correction term of the main term $I_{0}$. If the propagation of the radiation field is almost isotropic, the role of the correction term can be ignored. Then $|I_{1}|\ll I_{0}$, and it satisfies the following Fick's Law:
\begin{equation}\label{I13}
I_{1}(t,x,\nu)=-\nabla I_{0}(t,x,\nu).
\end{equation}

Now, integrating \eqref{I11} with respect to $\omega$ over $\mathbb{S}^{2}$ and using \eqref{I13}, then ignoring the effect of the positive constants  in the resulting equation after the integration, we can deduce the following diﬀusion approximation or so called $P_{0}$ approximation model via \eqref{NJK-1}--\eqref{I9}:   
\begin{equation}\label{NJK-2}
\left\{
\begin{aligned}
& \partial_t \varrho+{\rm div}(\varrho u)=0, \\
& \varrho(\partial_t u+u\cdot\nabla u)+\nabla P=\mu \Delta u+(\lambda+\mu)\nabla{\rm div}u-\nabla I_0, \\
& c_{\nu}\varrho(\partial_t \Theta+u\cdot\nabla \Theta)+P{\rm div}u=\kappa \Delta \Theta+\lambda({\rm div}u)^2+2\mu D\cdot D-\Theta^4+I_0, \\
& \partial_t I_0-\Delta I_0=\Theta^4-I_0,
\end{aligned}\right.
\end{equation}
where
\begin{align*}
D\cdot D:=\sum_{i,j=1}^3 D_{ij}^2 \quad    \text{and}    \quad D_{ij}:=\frac{1}{2}\Big(\frac{\partial u_i}{\partial x_j}+\frac{\partial u_j}{\partial x_i}\Big).
\end{align*}

Let us brieﬂy recall some mathematical related works about this model. For 1-D initial-boundary value problem, the global existence of smooth solutions was established in \cite{J, J-X-Z} with large initial data.
The global existence and large time behavior of smooth solutions for 3-D Cauchy problem near the equilibrium state for ideal polyatomic gas and general gas are shown in \cite{Jp-DCDS-2017} and \cite{K-H-K} respectively.
Recently, 
Wang-Xie-Yang \cite{WXY-JDE-2023} used the Littlewood-Paley decomposition technique and frequency decomposition method combining with delicate energy estimates to obtain the global existence and time decay rates of classical solution   of \eqref{NJK-2} in $\mathbb{R}^{3}$ with small initial data in $H^{4}(\mathbb{R}^3)$.
In \cite{jln-2024}, we extend the results  in \cite{WXY-JDE-2023} by  diminishing  the regularity of the initial data from $H^4(\mathbb{R}^3)$ to $H^2(\mathbb{R}^3)$ and  deriving  the optimal decay rates (including highest-order derivatives) of strong solutions.  
For the inviscid flow case of this model, the local existence of smooth solutions was shown in \cite{J-Z} for 3-D Cauchy problem.

On the one hand, when the propagation of the radiation field is weakly anisotropic, the effect of the correction term $\omega\cdot I_{1}$ in \eqref{I12} cannot be ignored. Then integrating both \eqref{I11}
and  $\eqref{I11}\cdot\omega$ with respect to $\omega$ over $\mathbb{S}^{2}$ and also ignoring the effect
of the positive constants in the resulting equation after the integration, we obtain the so-called NSF-$P_1$ approximation model from \eqref{NJK-1}--\eqref{I9} as  
\begin{equation}\label{I-1}
\left\{
\begin{aligned}
& \partial_t \varrho+{\rm div}(\varrho u)=0, \\
& \varrho(\partial_t u+u\cdot\nabla u)+\nabla P=\mu \Delta u+(\lambda+\mu)\nabla{\rm div}u+I_1, \\
& c_{\nu}\varrho(\partial_t \Theta+u\cdot\nabla \Theta)+P{\rm div}u=\kappa \Delta \Theta+\lambda({\rm div}u)^2+2\mu D\cdot D-\Theta^4+I_0, \\
& \partial_t I_0+{\rm div}I_1=\Theta^4-I_0, \\
& \partial_t I_1+\nabla I_0=-I_1.
\end{aligned}\right.
\end{equation}
It should be noted that from the numerical simulation results \cite{BD-JMP-2006, S-J-X}, the approximate accuracy of $P_1$ model \eqref{I-1} is higher than that of $P_{0}$ model \eqref{NJK-2}. Meanwhile, the $P_1$ model is more complex in structure, which adds additional difficulties to the mathematical analysis of this model. The more detailed derivation of these two models can be found in \cite{Pgc-1973}.  

Currently, research on the $P_1$ approximation model is primarily focused on the isentropic case (see  \cite{Fm-bima-2007}), 
i.e. 
\begin{equation}\label{NJKI-2}
\left\{
\begin{aligned}
& \partial_t \varrho+{\rm div}(\varrho u)=0, \\ 
& \varrho(\partial_t u+u\cdot\nabla u)+\nabla P=\mu \Delta u+(\lambda+\mu)\nabla{\rm div}u+I_1, \\
& \partial_t I_0+{\rm div}I_1=-I_0, \\
& \partial_t I_1+\nabla I_0=-I_1.
\end{aligned}\right.
\end{equation}
Danchin-Ducomet \cite {DD-JEE-2014} first proved the global existence of strong solutions for \eqref{NJKI-2} in critical Besov spaces $\dot{B}_{2,1}^\frac{1}{2}(\mathbb{R}^{3})\bigcap\dot{B}_{2,1}^\frac{3}{2}(\mathbb{R}^{3})$
for Cauchy problem when the initial data is sufficient small. Meanwhile, Danchin-Ducomet \cite{DD-SIAM-2016} investigated the low Mach number limit of this model in the framework of strong solutions.
Afterwards,  Xu et al. \cite{XCWWL-MM-2020} established time-decay rates of the strong solution.
Later, Wang-Yang-Xie \cite{WXY-JDE-2018} obtained the global-in-time classical solution 
for small initial data   in $H^4(\mathbb{R}^3)$. In addition, they
also achieved the large time behaviour of smooth solutions when the $L^1$-norm  of   initial data is bounded. 
Then, Wang \cite{Wwj-siam-2021} further obtained the global strong
solutions while the initial data have $H^2$ regularity. Moreover, the highest
order spatial derivative of the solution has been achieved by
the standard frequency decomposition method. Besides the results in global well-posedness, Fan-Hu-Nakamura \cite{FHN-BKMS-2020} showed the local existence
of strong solutions including vacuum on a 3-D bounded domain. 

Since the importance of thermal radiation in physical problems increases as the temperature is raised, the NSF-$P_1$ model \eqref{I-1} has a more obvious physical significance compared to the isentropic one. At present, there are only a few mathematical results on NSF-$P_1$ model. The   Non-relativistic and low Mach number limits for the local smooth solution of this model in $\mathbb{R}^{3}$ are presented
in \cite{FLN-CMS-2016, JLX-SIAM-2015}. The Non-equilibrium-diffusion limit result is shown in \cite{J-J-L}. In addition, the numerical simulation results of
this model is presented  in  \cite{DST-JCO-2007, TS-AMM-2008, TSGKS-JCP-2006}. However, to our knowledge, the global well-posedness of solutions for the NSF-$P_1$ model \eqref{I-1} has not yet been resolved.
The main reason is that there is an additional linear term $I_{0}$ in $\eqref{I-1}_{3}$, which would bring fundamental difficulties for the uniform a priori estimates, and cause the energy method developed  in \cite{WXY-JDE-2023, WXY-JDE-2018}   failed when  dealing  with the NSF-$P_1$ model.

\subsection{Our results}
In this paper, we investigate the global existence 
and optimal time decay rates of 
classical solutions to system \eqref{I-1} with the initial data
\begin{equation*}
(\varrho,u,\Theta,I_0,I_1)(x,t)|_{t=0}=(\varrho_0(x),u_0(x),\Theta_0(x),I_0^0(x),I_1^0(x)), \quad x\in\mathbb{R}^3. 
\end{equation*}
Under the assumption that  the  $H^4$-norm  of the initial data   
is sufficiently small, we will develop a new way to obtain the global existence 
of classical solutions. Moreover, the optimal time-decay rates of 
classical solutions can be achieved if we further assume
that the $L^1$-norm  of initial data   is bounded. 
By employing Sobolev's interpolation inequality,
we eventually get the desired decay rates of solutions at $L^p$-norm with $p \in [2,+\infty]$.

Noticing that $(\varrho, u, \Theta,I_0, I_1)\equiv(1,0,1,1,0)$ is an equilibrium
state of \eqref{I-1}, we shall set the standard perturbation $\varrho=1+\rho,\ \Theta=1+\theta, \ I_0=1+n_0, \ I_1=n_1$. Without loss of generality, we  take $c_\textsl{v}=R=\mu=\lambda=\kappa=1$ for presentation 
simplicity. Then the system \eqref{I-1}
can be written as follows:
\begin{equation}\label{I-2}
\left\{
\begin{aligned}
& \partial_t \rho+(1+\rho){\rm div}u+\nabla \rho\cdot u=0, \\
& \partial_t u+u\cdot\nabla u+\frac{1+\theta}{1+\rho}\nabla\rho+\nabla\theta=\frac{\Delta u}{1+ \rho}+\frac{2\nabla {\rm div}u}{1+\rho}+\frac{n_1}{1+\rho}, \\
& \partial_t \theta+u\cdot\nabla\theta+(1+\theta){\rm div}u=\frac{\Delta \theta}{1+ \rho}+\frac{({\rm div}u)^2}{1+\rho}+\frac{2D\cdot D}{1+\rho}-\frac{(1+\theta)^4}{1+\rho}+\frac{1+n_0}{1+\rho},\\
& \partial_t n_0+{\rm div}n_1=(1+\theta)^4-(1+n_0), \\
& \partial_t n_1+\nabla n_0=-n_1,
\end{aligned}\right.
\end{equation}
with the initial data
\begin{align}\label{I--2}
(\rho,u,\theta,n_0,n_1)(x,t)|_{t=0}=\,&( \rho_{0}(x),u_{0}(x),\theta_{0}(x),n_0^0(x),n_1^0(x))\nonumber\\
:=\,&(\varrho_{0}(x)-1,u_{0}(x),\Theta_{0}(x)-1, I_0^0(x)-1, I_1^0(x) ).
\end{align}

\begin{rem}
Because of the appearing of the linear term $n_1$ in \eqref{I-2}$_2$ and linear term $n_0$ in \eqref{I-2}$_3$,  the energy method developed in \cite{WXY-JDE-2023, WXY-JDE-2018} doesn't work here. This is
the main difficulty in the proof of the global well-posedness of classical solutions.
\end{rem}

The main results of this paper read 
\begin{thm}\label{T2.1}
Suppose that $\|(\rho_0,u_0,\theta_0,n_0^0, n_1^0)\|_{H^4}$ is small enough.
Then,
the Cauchy problem \eqref{I-2}--\eqref{I--2} admits a unique global classical solution $(\rho,u,\theta,n_0,n_1)$
satisfying 
\begin{gather*}
\rho,u,\theta,n_0,n_1\in C([0,\infty);H^{4}),
\end{gather*}
and
\begin{align}\label{G1.6}
&\|(\rho,u,\theta,n_{0},n_1)(\tau)\|_{H^{4}}^{2}+\int_{0}^{t}
\Big(\|\nabla(\rho,n_0)(\tau)\|_{H^{3}}^{2}+\|\nabla(u,\theta)(\tau)\|_{H^{4}}^{2}\Big)
\mathrm{d}\tau\nonumber\\
&\quad \,+\int_0^t\Big(\|n_1(\tau)\|_{H^{4}}^{2}+\|(4\theta-n_0)(\tau)\|_{H^4}^2\Big)\mathrm{d}\tau 
\leq  \,C\|(\rho,u,\theta,n_{0},n_1)(0)\|_{H^{4}}^{2}.
\end{align}
\end{thm}

\begin{thm}\label{T1.2}
Under the supposition of Theorem \ref{T2.1},  if we further assume that $\|(\rho_0,u_0,\theta_0, \linebreak n_0^0, n_1^0)\|_{L^1}$ is bounded, then we have  
\begin{align}\label{G1.9}
\|\nabla^m (\rho,u,\theta,n_0,n_1)\|_{L^2}\leq&\, C(1+t)^{-\frac{3}{4}-\frac{m}{2}},
\end{align}  
for all $t\geq 0$ and $m=0,1,2$; and
\begin{align}\label{G1.11}
\|\nabla^m (\rho,u,\theta,n_0,n_1)\|_{L^2}\leq&\, C(1+t)^{-\frac{7}{4}},
\end{align}  
for all $t\geq 0$ and $m=3,4$; and
\begin{align}\label{G1.10}
\|(\rho,u,\theta,n_0,n_1)(t)\|_{L^p}&\leq C(1+t)^{-\frac{3}{2}(1-\frac{1}{p})},
\end{align} 
for all $t\geq 0$ and $2\leq p\leq{\infty}$; and 
\begin{align}\label{GB1.10}
\|\nabla(\rho,u,\theta,n_0,n_1)(t)\|_{L^p}&\leq C(1+t)^{-\frac{3}{2}(\frac{4}{3}-\frac{1}{p})},
\end{align} 
for all $t\geq 0$ and $2\leq p\leq 6$.
Moreover, \begin{align}
\|\partial_t \rho(t)\|_{L^2}\leq C(1+t)^{-\frac{5}{4}}, \\
\|\partial_t (u,\theta,n_0,n_1)(t)\|_{L^2}\leq C(1+t)^{-\frac{3}{4}}.
\end{align}
\end{thm}

\begin{rem}
  In Theorem \ref{T1.2}, we have obtained the optimal decay rates up to the second-order  derivatives of the classical solutions
 to the NSF-$P_1$ model \eqref{I-2}  in the sense of that they are  consistent with that of 
 the compressible Navier-Stokes equations \cite{MN79}. 
\end{rem}

\subsection{Strategies of the proofs  in our results}
In Theorem \ref{T2.1}, the third term on the left-hand side of  \eqref{G1.6} describes the dissipative terms arise from
the interaction between radiation field and fluid. 
Due to the fact that the classical energy method (see, for example the arguments  in  \cite{MN-jmku-1980, WXY-JDE-2023, WXY-JDE-2018})  is failed to get the 
$L_x^1$-norm estimate of $n_1u$ and $\theta n_0$,
as a result, the main difficulty in
the proof of Theorem \ref{T2.1} is how to establish the uniform a priori estimates of
classical solutions to Cauchy problem \eqref{I-2}--\eqref{I--2}.
Fortunately, we surmount this difficulty through dividing the solution 
$(\rho,u,\theta,n_0,n_1)$ into the lower, medium, and higher frequency parts.
For the  low frequency part, we overcome the enormous difficulties caused by the lack of dissipative structure in the system \eqref{G3.8} by constructing a change of variable (see \eqref{G3.44}--\eqref{G3.45}).
This is the key to get the estimate of solutions in low frequency regimes.
Besides, noting that the $L^{\frac{6}{5}}$-norm estimates of $\mathcal{N}_3$ 
and $\mathcal{N}_4$  are not closed, we work out it through
the refined energy estimates to get the $L^{\frac{6}{5}}$-norm  
of $\mathcal{N}_3+\mathcal{N}_4$ and the $L^2$-norm  
of $\mathcal{N}_3$ and $\mathcal{N}_4$  (see \eqref{G3.59}).
For the  medium frequency part, the estimate of solutions lacks
the dissipative term, which leads to additional nonlinear terms
in \eqref{GB3.106}. To handle it, we perform spectral analysis by
the tedious computation in Lemma \ref{L3.3}, and hence get the $L^2$-norm estimate of medium frequency part.
We mainly use utilize Lemma \ref{L3.7} to deal with these nonlinear terms.
Meanwhile, the estimate of high-frequency part can be also obtained by delicate energy method.
Finally,  by making use of the dissipative terms and the structure of
system \eqref{I-2}, we arrive at sufficient a priori estimates and   establish  the global existence
of classical solutions in $\mathbb{R}^3$.

In Theorem \ref{T1.2}, we get the optimal decay rates of lower order derivatives
of solutions in $L^2$-norm, which are the same as that to the heat equation. 
Furthermore,  the optimal decay rates of classical solutions in $L^p(2\leq p\leq \infty)$ and their gradient
in $L^p(2\leq p \leq 6)$ are also obtained by
Sobolev's interpolation inequality. 
More precisely,  using the result established in high frequency regimes  in 
Section 3, we first  give the estimates of solutions at $L^2$-norm.
Compared to \cite{WXY-JDE-2023}, here we develop a new method to combine the low-frequency part with the medium-frequency part (see Lemma \ref{L4.3}), which is based on the spectral
analysis on the linearized system of \eqref{I-3}--\eqref{I--3} and the result
achieved in Section 3.
Then, with the aid of the estimates in different regimes, large time decay
rates of classical solutions can be acquired by employing Duhamel principle.
At last, we   achieve the decay rates of $\partial_t(\rho,u,\theta,n_0,n_1)$
and the optimal decay rates of solutions at $L^p$-norm. 

The rest of this paper is stated as follows. In 
Section 2, we give 
some preliminaries containing notations and  useful lemmas. In Section 3,
we establish a priori estimates concerning different frequency part and
consequently obtain the global existence of classical solutions to Cauchy problem $\eqref{I-2}$--$\eqref{I--2}$. 
In Section 4, we provide time-decay estimates of classical solutions for high frequency part and low-medium frequency parts, and thus derive the optimal time decay rates of classical solutions.

\section{Preliminaries}
In this section, we shall introduce some notations and the useful lemmas, which are
frequently used throughout this paper. 
\subsection{Notations}
For $L^p(\mathbb{R}^3)$ ($1\leq p \leq \infty$) space, the corresponding norm is 
denoted by $\|\cdot\|_{L^p}$; for the Sobolev space $H^k(\mathbb{R}^3)$, its norm is
denoted by $\|\cdot\|_{H^k}$, with $k\geq 0$. The letter $C$ represents a generic 
positive constant,  independent of time $t$, and may change from line to line.
$A\sim B$ means that $C^{-1}A\leq B\leq CA$ for some constant $C>0$.
For all integer $m\geq 0$, $\nabla^m$ represents the m-th order derivatives.
We also use $\langle\cdot, \cdot\rangle$ to represent the standard $L^2$ inner product in $\mathbb{R}^3$, i.e.
\begin{align*}
\langle f,g \rangle=\int_{\mathbb{R}^3} f(x)g(x) \mathrm{d}x,
\end{align*}
for any $f(x), g(x) \in L^2(\mathbb{R}^3)$. 
If function $h:\mathbb{R}^3\rightarrow \mathbb{R}$ is integrable, it has the following Fourier transform:
\begin{align*}
\mathcal{F}(h)(\xi)=\widehat{h}(\xi)=\int_{\mathbb{R}^{3}}e^{-ix\cdot\xi}h(x)dx,    
\end{align*}
where $i=\sqrt{-1}\in\mathbb{C}$ and $x\cdot\xi=\sum_{j=1}^{3}x_{j}\xi_{j}$ for 
any $\xi\in\mathbb{R}^{3}$.

Next, we recall the classical Littlewood-Paley decomposition. Setting
\begin{align*}
A_{i}=\{\xi\in\mathbb{R}^{3}|2^{i-1}\leq|\xi|\leq2^{i+1}\},\quad i\in\mathbb{Z},   
\end{align*}
and assuming $\{\phi_i\}_{i\in\mathbb{Z}} \subset \mathcal{S}$ ($\mathcal{S}$ is Schwartz class) satisfies
\begin{align*}
{\rm supp}\widehat{\phi}_{i}\subset A_{i},\quad \widehat{\phi}_{i}(\xi)=\widehat{\phi}_{0}(2^{-i}\xi)\quad\text{or}\quad\phi_{i}(x)=2^{3i}\phi_{0}(2^{i}x),    
\end{align*}
and
\begin{align*}
\sum\limits_{i=-\infty}^{+\infty}\widehat{\phi}_{i}(\xi)=\left\{\begin{array}{ll}1,
&\xi\in\mathbb{R}^{3}\setminus\{0\},\\0,&\xi=0,\end{array}\right.    
\end{align*}
then we have
\begin{align*}
\dot{\Delta}_{i}u:=\widehat{\phi}_{i}(\mathfrak{D})u=2^{3i}\int_{\mathbb{R}^{3}}\phi_{0}(2^{i}y)u(x-y)\mathrm{d}y,\quad i\in\mathbb{Z},
\end{align*}
where 
$(\dot{\Delta}_i)_{i\in\mathbb{Z}}$ over $\mathbb{R}^3$ is the homogeneous Littlewood-Paley decomposition and $\mathfrak{D}=\mathfrak{D}_x=\frac{1}{\sqrt{-1}}(\partial_{x_1},\partial_{x_2},\partial_{x_3})$.

Now, we introduce the homogeneous Besov space.
For any $s\in\mathbb{R}$ and $p,q\geq 1$, the homogeneous Besov space 
$\dot{B}_{p,q}^{s}(\mathbb{R}^{3})$ consists of $g\in{\mathscr{S}}_h^{\prime}=\mathscr{S}^{\prime}/\mathscr{P}$ satisfying
$$\|g\|_{\dot{B}_{p,q}^{s}(\mathbb{R}^{3})}:
=\left\|2^{sk}\|\dot{\Delta}_{k}g\|_{L^{p}(\mathbb{R}^{3})}\right\|_{l^{q}(\mathbb{Z})}<\infty,$$
where $\mathscr{S}^\prime$ is the dual of $\mathscr{S}$ and $\mathscr{P}$ is the space of polynomials.

Notice that $g\in\mathscr{S}_h^{\prime}$ can be expressed by
\begin{align*}
g=\sum_{k\in\mathbb{Z}}\dot{\Delta}_{k}g,    
\end{align*}
and  its long wave part, medium wave part and short wave part are defined as  follows:
\begin{align*}
g^{L}:=\sum_{k< k_{0}}\dot{\Delta}_{k}g,\quad g^{M}:=\sum_{k_0\leq k\leq k_{1}}\dot{\Delta}_{k}g,
\quad g^{S}:=\sum_{k>k_{1}}\dot{\Delta}_{k}g, \quad g^{L+M}:=g^L+g^M,  
\end{align*}
where the fixed positive integers $k_0$ and $ k_1$ are defined in \eqref{G3.23} and
\eqref{G3.53} respectively. 
 
We also use the following notations
\begin{align}\label{G2.7}
\|g^{L}\|_{\dot{B}_{p,q}^{s}}:&=\bigg(\sum_{k< k_{0}}2^{qsk}\|\dot{\Delta}_{k}g\|_{L^{p}}^{q}\bigg)^{\frac{1}{q}},\nonumber\\
\|g^{M}\|_{\dot{B}_{p,q}^{s}}:&=\bigg(\sum_{k_0\leq k\leq k_{1}}2^{qsk}\|\dot{\Delta}_{k}g\|_{L^{p}}^{q}\bigg)^{\frac{1}{q}},\nonumber\\
\|g^{S}\|_{\dot{B}_{p,q}^{s}}:&=\bigg(\sum_{k>k_{1}}2^{qsk}\|\dot{\Delta}_{k}g\|_{L^{p}}^{q}\bigg)^{\frac{1}{q}}.
\end{align}

For brevity, we set $\|(a,b)\|_{X}:=\|a\|_{X}+\|b\|_{X}$ for Banach
space $X$, where $a=a(x)$ and $b=b(x)$ belong to $X$.
Besides, we denote $\partial_i=\partial_{x_i}$ for $i=1,2,3$, 
$\partial^{\alpha}_x:=\partial^{\alpha_1}_{x_1}\partial^{\alpha_2}_{x_2}\partial^{\alpha_3}_{x_3}$ for the multi-index $\alpha=(\alpha_1,\alpha_2,\alpha_3)$,
and the corresponding norm
\begin{align*}
\|h\|_{H^k}:=\sum_{|\alpha|\leq k}\|\partial^{\alpha} h\|_{L^2}.    
\end{align*}
$(f|g):=\int_{\mathbb{R}^3}f  \bar g \mathrm{d}x $ represents the dot product of $f$ 
with the complex conjugate of $g$. Finally,
let $\Lambda^s$ be the pseudo-differential operator defined as 
\begin{align*}
\Lambda^s g=\mathcal{F}^{-1}\big(|\xi|^s \mathcal{F}(g)\big),    
\end{align*}
for some $s\in \mathbb{R}$. 

With the above preparations, we   state some useful results which will be used frequently.  

\subsection{Some useful lemmas}

In this subsection, we present some   useful lemmas which will be used  in the rest of this paper frequently. 

\begin{lem}[see\cite{J,AF-Pa-2003}] \label{L2.1}
There exists a constant $C>0$ such that for any $f, g \in H^4\left(\mathbb{R}^3\right)$ and any multi-index $\alpha$ with $1 \leq|\alpha| \leq 4$, it holds that 
\begin{align*}
\|f\|_{L^{\infty}\left(\mathbb{R}^3\right)} & \leq C\left\|\nabla f\right\|_{L^2\left(\mathbb{R}^3\right)}^{1 / 2}\left\|\nabla^2 f\right\|_{L^2\left(\mathbb{R}^3\right)}^{1 / 2}\leq C\| f\|_{H^2\left(\mathbb{R}^3\right)}, \\
\|f g\|_{H^3\left(\mathbb{R}^3\right)} & \leq C\|f\|_{H^3\left(\mathbb{R}^3\right)}\left\|\nabla g\right\|_{H^3\left(\mathbb{R}^3\right)}, \\
\left\|\partial^\alpha(f g)\right\|_{L^2\left(\mathbb{R}^3\right)} & \leq C\left\|\nabla f\right\|_{H^3\left(\mathbb{R}^3\right)}\left\|\nabla g\right\|_{H^3\left(\mathbb{R}^3\right)},\\
\|f \|_{L^r\left(\mathbb{R}^3\right)} & \leq C\|f\|_{H^1\left(\mathbb{R}^3\right)} \quad \text{for} \quad  2\leq r \leq 6. 
\end{align*}  
\end{lem}

\begin{lem}[see \cite{LZ-Cpaa-2020}]\label{L2.2}
Let $h$ and $g$ be Schwarz functions. Then, for $k\geq 0$, one has
\begin{align*}
\|\nabla^{k}(gh) \|_{L^r\left(\mathbb{R}^3\right)} & \leq C\|g\|_{L^{r_1}\left(\mathbb{R}^3\right)}\|\nabla^{k}h\|_{L^{r_2}\left(\mathbb{R}^3\right)}+C\|h\|_{L^{r_3}\left(\mathbb{R}^3\right)}\|\nabla^{k}g\|_{L^{r_4}\left(\mathbb{R}^3\right)},
\end{align*}
with $1<r,r_2,r_4<\infty$ and $r_i(1\leq i\leq 4)$ satisfy the following relation
\begin{align*}
\frac{1}{r_1}+\frac{1}{r_2}=\frac{1}{r_3}+\frac{1}{r_4}=\frac{1}{r}.   
\end{align*}
\end{lem}

\begin{lem}[see\cite{J}]\label{L2.3}
Given any $0<\beta_1\neq1$
and $\beta_2>1$, it holds that
\begin{align*}
\int_0^t(1+t-s)^{-\beta_1}(1+s)^{-\beta_2}\mathrm{d}s\leq C(1+t)^{-\min\{\beta_1,\beta_2\}} 
\end{align*}
for all $t\geq 0$.
\end{lem}

\begin{lem}[see\cite{AF-Pa-2003}]\label{L2.4}
Suppose that $1\leq r\leq s\leq t\leq \infty$, and
\begin{align*}
\frac{1}{s}=\frac{\zeta}{r}+\frac{1-\zeta}{t},
\end{align*}
where $0\leq \zeta\leq 1$.
Assume that $f\in L^r(\mathbb{R}^3)\cap L^t(\mathbb{R}^3)$, then $f\in L^s(\mathbb{R}^3)$, and
\begin{align*}
\|f\|_{L^s}\leq \|f\|_{L^r}^\zeta \|f\|_{L^t}^{1-\zeta}.    
\end{align*}
\end{lem}

\begin{lem}[see the Appendix in\cite{WXY-JDE-2023}]\label{L2.5}
For any s $\in\mathbb{R}$,
\begin{align*}
\dot{H}^{s}\sim\dot{B}_{2,2}^{s}.    
\end{align*}
For any $s\in \mathbb{R}$ and $1< q< \infty$, it holds that
\begin{align*}
\dot{B}_{q,\min\{q,2\}}^{s}\hookrightarrow&\dot{W}^{s,q}\hookrightarrow\dot{B}_{q,\max\{q,2\}}^{s},\nonumber\\
\dot{B}_{q,\min\{q,2\}}^{0}\hookrightarrow& L^{q}\hookrightarrow\dot{B}_{q,\max\{q,2\}}^{0}.
\end{align*}
\end{lem}
\begin{lem}\label{L2.6}
For any $s\geq 0$ and $m_1\geq m_2\geq 0$, it holds that
\begin{align*}
\|\Lambda^{m_{2}}f^{S}\|_{\dot{B}_{2,2}^{s}}&\,\leq C\|\Lambda^{m_{1}}f^{S}\|_{\dot{B}_{2,2}^{s}},\quad \ \, \|\Lambda^{m_{1}}f^{L}\|_{\dot{B}_{2,2}^{s}}\leq C\|\Lambda^{m_{2}}f^{L}\|_{\dot{B}_{2,2}^{s}},\nonumber\\
\|\Lambda^{m_{2}}f^{M}\|_{\dot{B}_{2,2}^{s}}&\,\leq C\|\Lambda^{m_{1}}f^{M}\|_{\dot{B}_{2,2}^{s}},\quad \|\Lambda^{m_{1}}f^{M}\|_{\dot{B}_{2,2}^{s}}\leq C\|\Lambda^{m_{2}}f^{M}\|_{\dot{B}_{2,2}^{s}}.
\end{align*}
and
\begin{align*}
\|f\|_{\dot{B}_{2,2}^{s}}\sim\|f^{L}\|_{\dot{B}_{2,2}^{s}}+\|f^{M}\|_{\dot{B}_{2,2}^{s}}+\|f^{S}\|_{\dot{B}_{2,2}^{s}}.   
\end{align*}
\end{lem}
\begin{proof}
The proof can be easily completed by Plancherel's theorem, 
Bernstein's inequality and \eqref{G2.7}. For brevity, we omit the detail here.
\end{proof}
\begin{lem}[see Appendix in\cite{WXY-JDE-2023}]\label{L2.7}
For the commutator
\begin{align*}
[\dot{\Delta}_{k},v\cdot\nabla]f:=\dot{\Delta}_{k}(v\cdot\nabla f)-v\cdot\nabla\dot{\Delta}_{k}f,    
\end{align*}
it holds that
\begin{align*}
\int_{\mathbb{R}^{3}}[\dot{\Delta}_{k},v\cdot\nabla]f\cdot\dot{\Delta}_{k}g\mathrm{d}x\leq
&\,C\|\nabla v\|_{L^{\infty}}\|\dot{\Delta}_{k}f\|_{L^{2}}\|\dot{\Delta}_{k}g\|_{L^{2}}+C\|\nabla f\|_{L^{\infty}}\|\dot{\Delta}_{k}v\|_{L^{2}}\|\dot{\Delta}_{k}g\|_{L^{2}}\nonumber\\
&+C\|\nabla v\|_{L^{\infty}}\|\dot{\Delta}_{k}g\|_{L^{2}}\sum_{l\geq k-1}2^{k-l}\|\dot{\Delta}_{l}f\|_{L^{2}},    
\end{align*}
where $v$ is a vector field over $\mathbb{R}^3$.
\end{lem}
\begin{lem}[see Appendix in\cite{Wyj-2012-jde}]\label{L2.8}
Let $0<l<3, 1<q<p<\infty, \frac{1}{p}+\frac{l}{3}=\frac{1}{q}$. Then,  it holds
\begin{align*}
\|\Lambda^{-l}f\|_{L^p}\leq C \|f\|_{L^q}.
\end{align*}
\end{lem}

\section{Global existence of smooth solutions to the Cauchy problem \eqref{I-2}--\eqref{I--2} }

In this section, we shall establish the global existence of 
classical solutions to the Cauchy problem \eqref{I-2}-\eqref{I--2}
in the whole space $\mathbb{R}^3$.

\subsection{A priori assumption and the reformulation of the problem   \eqref{I-2}--\eqref{I--2}}
First, in order to achieve the uniform-in-time
a priori estimate of $(\rho,u,\theta,n_0,n_1)$, we make the following assumptions
\begin{align}\label{G3.1}
\sup_{t\geq0}\left\|(\rho,u,\theta,n_0,n_1)\right\|_{H^{4}}\leq\sigma,
\end{align}
where the generic constant $0<\sigma<1$  is small enough, and $(\rho,u,\theta,n_0,n_1)$
is classical solution to the Cauchy problem \eqref{I-2}--\eqref{I--2} on
$0\leq t<T$ for some $T>0$.

Below we perform some processing on the system \eqref{I-2}, 
we rewrite it  as the equivalent  form:
\begin{equation}\label{I-3}
\left\{
\begin{aligned}
& \partial_t \rho+{\rm div}u=\mathcal{N}_1, \\
& \partial_t u+\nabla\rho+\nabla\theta-\Delta u-2\nabla {\rm div}u-n_1=\mathcal{N}_2, \\
& \partial_t \theta+{{\rm div}u}-\Delta \theta+4\theta-n_0=\mathcal{N}_3,\\
& \partial_t n_0+{{\rm div}n_1}-4\theta+n_0=\mathcal{N}_4,\\
& \partial_t n_1+\nabla n_0+n_1=0,
\end{aligned}\right.
\end{equation}
with the initial data
\begin{align}\label{I--3}
(\rho,u,\theta,n_0,n_1)(x,t)|_{t=0}&=( \rho_{0}(x),u_{0}(x),\theta_{0}(x),n_0^0(x),n_1^0(x))\nonumber\\
&=(\varrho_{0}(x)-1,u_{0}(x),\Theta_{0}(x)-1, I_0^0(x)-1, I_1^0(x) ),
\end{align}
where the nonlinear terms $\mathcal{N}_i(1\leq i\leq 4)$ are defined as
\begin{equation}\label{I-4}
\left\{
\begin{aligned}
\mathcal{N}_1:=&\,-\rho{\rm div}u-\nabla\rho\cdot u, \\
\mathcal{N}_2:=&\,-u\cdot\nabla u-\big(g(\rho)+h(\rho)\theta\big)\nabla\rho+g(\rho)\big(\Delta u+2\nabla{\rm div} u+n_1\big), \\
\mathcal{N}_3:=&\,g(\rho)\big(\Delta \theta+n_0-4\theta\big)-\theta{\rm div}u-u\cdot\nabla\theta\\
& \,+h(\rho)\big(({\rm div}u)^2+2D\cdot D-\theta^4-4\theta^3-6\theta^2\big),\\
\mathcal{N}_4:=&\,\theta^4+4\theta^3+6\theta^2, 
\end{aligned}\right.
\end{equation}
and the nonlinear functions $g(\rho)$ and   $h(\rho)$  read 
\begin{equation}\label{I--4}
 g(\rho)= \,\frac{1}{1+\rho}-1, \quad 
h(\rho)= \,\frac{1}{1+\rho}.
 \end{equation}

Motivated by \cite{CMZ-CPAM-2010}, we decompose the $u$ and $n_1$ by setting $d:=\Lambda^{-1} {\rm div}u$,
$M:=\Lambda^{-1}{\rm div}n_1$.
Utilizing the identity $\Delta=\nabla{\rm div}-{\rm curl}({\rm curl})$, we have
\begin{align}\label{G3.6}
u=&-\Lambda^{-1}\nabla d-\Lambda^{-1}{\rm curl}(\Lambda^{-1}{\rm curl}\ u),\\\label{G3.7}
n_1=&-\Lambda^{-1}\nabla M-\Lambda^{-1}{\rm curl}(\Lambda^{-1}{\rm curl}\ n_1),
\end{align}
where ({curl} $u$)$_{ij}=\partial_{x_j}u^i-\partial_{x_i}u^j$. Then, the system \eqref{I-3} can be rewritten as 
\begin{equation}\label{G3.8}
\left\{
\begin{aligned}
& \partial_t \rho+\Lambda d=\mathcal{N}_1, \\
& \partial_t d-\Lambda\rho-\Lambda\theta-3\Delta d-M=\mathcal{D}_2, \\
& \partial_t \theta+\Lambda d-\Delta \theta+4\theta-n_0=\mathcal{N}_3,\\
& \partial_t n_0+\Lambda M-4\theta+n_0=\mathcal{N}_4,\\
& \partial_t M-\Lambda n_0+M=0,
\end{aligned}\right.
\end{equation}
with the initial data
\begin{align}\label{G3.9}
( \rho,d,\theta,n_0,M)(x,t)|_{t=0}=&\,(\rho_{0}(x),d_{0}(x),\theta_{0}(x),n_0^0(x),M_0(x))\nonumber\\
=&\,(\varrho_{0}(x)-1,{\Lambda^{-1}{\rm div}u}_{0}(x),\Theta_{0}(x)-1,I_0^0(x)-1, \Lambda^{-1}{\rm div}I_1^0(x)),
\end{align}
where $\mathcal{D}_2=\Lambda^{-1}{\rm div}\mathcal{N}_2$.
In addition, $\mathcal{P}=\Lambda^{-1}{\rm curl}$ is
 the projection operator satisfies
\begin{equation}\label{G3.10}
\left\{
\begin{aligned}
& \partial_t \mathcal{P}u-\Delta\mathcal{P}u-\mathcal{P}n_1=\mathcal{P}\mathcal{N}_2, \\
& \partial_t \mathcal{P}n_1+\mathcal{P}n_1=0,
\end{aligned}\right.
\end{equation}
with the initial data
\begin{align}\label{G3.11}
\mathcal{P}u(x,t)|_{t=0}=\mathcal{P} u_0(x),\quad 
\text{and}\quad  \mathcal{P}n_1(x,t)|_{t=0}=\mathcal{P} n_{1}^0(x). 
\end{align}
It is obviously that the estimate of $u$ and $n_1$ can be achieved through
the estimate of $d, \mathcal{P}u$ and $M, \mathcal{P}n_1$ separately.
Applying the homogeneous frequency localized operator $\dot{\Delta}_k$
to \eqref{G3.8}--\eqref{G3.11}, we arrive at
\begin{equation}\label{G3.12}
\left\{
\begin{aligned}
& \partial_t \rho_k+\Lambda d_k=\mathcal{N}_1^k, \\
& \partial_t d_k-\Lambda\rho_k-\Lambda\theta_k-3\Delta d_k-M_k=\mathcal{D}_2^k, \\
& \partial_t \theta_k+\Lambda d_k-\Delta \theta_k+4\theta_k-n_{0,k}=\mathcal{N}_3^k,\\
& \partial_t n_{0,k}+\Lambda M_k-4\theta_k+n_{0,k}=\mathcal{N}_4^k,\\
& \partial_t M_k-\Lambda n_{0,k}+M_k=0,
\end{aligned}\right.
\end{equation}
with the initial data
\begin{align}\label{G3.13}
&(\rho_k,d_k,\theta_k,n_{0,k},M_k)(x,t)|_{t=0}\nonumber\\
=&\,(\rho_{0}^k(x),d_{0}^k(x),\theta_{0}^k(x),n_0^{0,k}(x),M_0^k(x))\nonumber\\
=&\,(\varrho_{0}^k(x)-1,{\Lambda^{-1}{\rm div}u}_{0}^k(x),\Theta_{0}^k(x)-1,I_0^{0,k}(x)-1, \Lambda^{-1}{\rm div}I_1^{0,k}(x)),
\end{align}
as well as
\begin{equation}\label{G3.14}
\left\{
\begin{aligned}
& \partial_t (\mathcal{P}u)_k-\Delta(\mathcal{P}u)_k-(\mathcal{P}n_1)_k=(\mathcal{P}\mathcal{N}_2)_k, \\
& \partial_t (\mathcal{P}n_1)_k+(\mathcal{P}n_1)_k=0,
\end{aligned}\right.
\end{equation}
with the initial data
\begin{align}\label{G3.15}
(\mathcal{P}u)_k(x,t)|_{t=0}=(\mathcal{P} u_0)_k(x),\quad 
\text{and}\quad  (\mathcal{P}n_1)_k(x,t)|_{t=0}=(\mathcal{P} n_{1}^0)_k(x). 
\end{align}
Here, $\rho_k:=\dot{\Delta}_k\rho$, $d_k:=\dot{\Delta}_kd$ and so on.
In what follows, we shall use the homogeneous Littlewood-Paley decomposition $(\dot{\Delta}_k)_{k\in\mathbb{Z}}$ to obtain the estimates through
the corresponding frequency part.
\subsection{Estimates of the high-frequency part}
In this subsection, we will establish the estimates of the high-frequency part of 
smooth solutions to the problem \eqref{G3.12}--\eqref{G3.15}. It follows from
\eqref{G3.12} that
\begin{align}\label{G3.16}
&\frac{1}{2}\frac{\mathrm{d}}{\mathrm{d}t}\int_{\mathbb{R}^{3}}\Big(\rho_{k}^{2}+d_{k}^{2}+\theta_{k}^{2}+n_{0,k}^{2}+M_k^2\Big)\mathrm{d}x\nonumber\\
&+\int_{\mathbb{R}^{3}}\Big(3|\Lambda d_{k}|^{2}+|\Lambda\theta_{k}|^{2}+4|\theta_{k}|^{2}+| n_{0,k}|^{2}+|M_{k}|^{2}\Big)\mathrm{d}x\nonumber\\
&\qquad = 5\int_{\mathbb{R}^{3}}n_{0,k}\theta_{k}\mathrm{d}x+\int_{\mathbb{R}^{3}}M_{k}d_{k}\mathrm{d}x 
  +\int_{\mathbb{R}^3}    \Big(\rho_{k}\mathcal{N}_1^{k}+d_{k}\mathcal{D}_2^{k}+\theta_{k}\mathcal{N}_3^{k}+n_{0,k}\mathcal{N}_4^{k}\Big)       \mathrm{d}x    
\end{align}
From the equations \eqref{G3.12}$_1$--\eqref{G3.12}$_2$, we also have
\begin{align}\label{G3.17}
&\frac{\mathrm{d}}{\mathrm{d}t}\int_{\mathbb{R}^{3}}\Big(\frac{3}{2}|\nabla\rho_{k}|^{2}-\Lambda\rho_{k}d_{k}\Big)\mathrm{d}x\nonumber\\
&\qquad = \int_{\mathbb{R}^{3}}\Big(|\Lambda d_{k}|^{2}-|\Lambda\rho_{k}|^{2}-\Lambda\rho_{k}\Lambda\theta_{k}-M_{k}\Lambda \rho_k\Big)\mathrm{d}x
\nonumber\\
&\qquad \quad -\int_{\mathbb{R}^{3}}\left(d_{k}\Lambda\, \mathcal{N}_{1}^{k}+\Lambda\rho_{k}\mathcal{D}_2^{k}\right)\mathrm{d}x
+3\int_{\mathbb{R}^{3}}\nabla\rho_{k}\cdot\nabla \mathcal{N}_{1}^{k}\mathrm{d}x.    
\end{align}
Making use of the equations \eqref{G3.12}$_4$--\eqref{G3.12}$_5$, we deduce that
\begin{align}\label{G3.18}
&\frac{\mathrm{d}}{\mathrm{d}t}\int_{\mathbb{R}^{3}}\Big(|\Lambda M_{k}|^2+|\Lambda n_{0,k}|^2\Big)\mathrm{d}x+\int_{\mathbb{R}^3}\Big(|\Lambda M_{k}|^2+|\Lambda n_{0,k}|^2\Big)\mathrm{d}x\nonumber\\
&\,\qquad =\int_{\mathbb{R}^{3}}\Big(4\Lambda \theta_{k}\Lambda n_{0,k}+\Lambda\,\mathcal{N}_{4}^k\Lambda n_{0,k}\Big)\mathrm{d}x.    
\end{align}   

For fixed constant $\beta_1$ and $\beta_2$ ($0<\beta_1,\beta_2<1$), we sum
up \eqref{G3.16}, $\beta_1\times$\eqref{G3.17} and $\beta_2\times$\eqref{G3.18}
to get
\begin{align}\label{G3.19}
&\frac{1}{2}\frac{\mathrm{d}}{\mathrm{d}t}\mathcal{L}_{k}(t)+\beta_{1}\|\Lambda\rho_{k}\|_{L^{2}}^{2}+(3-\beta_{1})\|\Lambda d_{k}\|_{L^{2}}^{2}+\|\Lambda\theta_{k}\|_{L^{2}}^{2}+4\|\theta_{k}\|_{L^{2}}^{2}+\| n_{0,k}\|_{L^{2}}^{2}\nonumber\\
&+\|M_{k}\|_{L^{2}}^{2}+\beta_2\|\Lambda n_{0,k}\|+\beta_2\|\Lambda M_k\|\nonumber\\
 &\,\qquad =\int_{\mathbb{R}^{3}}\Big( 5n_{0,k}\theta_{k}+M_k d_k\Big)\mathrm{d}x-\beta_{1}\int_{\mathbb{R}^{3}}\Big(\Lambda\rho_{k}\Lambda\theta_{k}+M_{k}\Lambda \rho_{k}\Big)\mathrm{d}x\nonumber
\\ &\qquad\quad+\int_{\mathbb{R}^{3}}\Big(\rho_{k}\mathcal{N}_{1}^{k}+d_{k}\mathcal{D}_2^{k}+\theta_{k}\mathcal{N}_{3}^{k}+n_{0,k}\mathcal{N}_{4}^{k}-\beta_{1}d_{k}
\Lambda\,\mathcal{N}_{1}^{k}-\beta_{1}\Lambda\rho_{k}\mathcal{D}_2^{k}\Big)\mathrm{d}x\nonumber\\
&\qquad\quad+3\beta_{1}\int_{\mathbb{R}^{3}}\nabla\rho_{k}\cdot\mathcal{N}_{1}^{k}\mathrm{d}x+
\beta_{2}\int_{\mathbb{R}^{3}}\Big(4\Lambda\theta_{k}\Lambda n_{0,k}+\Lambda n_{0,k}\Lambda\,\mathcal{N}_4^k\Big)\mathrm{d}x,
\end{align}
where the functional $\mathcal{L}_k(t)$ has the following form
\begin{align*}
\mathcal{L}_k(t):=\int_{\mathbb{R}^{3}}\Big(\rho_{k}^{2}+d_{k}^{2}+\theta_{k}^{2}+n_{0,k}^{2}+M_k^2+\beta_{1}\big(\frac{3}{2}|\nabla\rho_{k}|^{2}-\Lambda\rho_{k}d_{k}\big)+\beta_2\big(|\Lambda M_k|^2+|\Lambda n_{0,k}|^2\big)\Big)\mathrm{d}x.    
\end{align*}
Utilizing Young's inequality, we derive  that 
\begin{align}
5\int_{\mathbb{R}^3}n_{0,k}\theta_k \mathrm{d}x \leq&\, \frac{1}{4}\|n_{0,k}\|_{L^2}^2+25\|\theta_{k}\|_{L^2}^2,\label{G3.211} \\
\int_{\mathbb{R}^3}M_kd_k \mathrm{d}x \leq&\, \frac{1}{4}\|M_{k}\|_{L^2}^2+\|d_{k}\|_{L^2}^2,\label{G3.212}\\
-\beta_1\int_{\mathbb{R}^3}\Lambda \rho_k\Lambda\theta_k \mathrm{d}x \leq&\, \frac{1}{4}\beta_1\|\Lambda\rho_{k}\|_{L^2}^2+\beta_1\|\Lambda\theta_{k}\|_{L^2}^2,\label{G3.213}\\
-\beta_1\int_{\mathbb{R}^3}\Lambda \rho_k M_k \mathrm{d}x \leq&\, \frac{1}{4}\beta_1\|\Lambda\rho_{k}\|_{L^2}^2+\beta_1\|M_{k}\|_{L^2}^2,\label{G3.214}\\
4\beta_2\int_{\mathbb{R}^3}\Lambda \theta_k \Lambda n_{0,k} \mathrm{d}x \leq&\, \frac{1}{4}\beta_2\|\Lambda n_{0,k}\|_{L^2}^2+16\beta_2\|\Lambda\theta_{k}\|_{L^2}^2.\label{G3.215}
\end{align}
Inserting \eqref{G3.211}--\eqref{G3.215} into \eqref{G3.19}, we discover that
\begin{align}\label{G3.22}
 &\frac{1}{2}\frac{\mathrm{d}}{\mathrm{d}t}\mathcal{L}_{k}(t)+\frac{\beta_{1}}{2}\|\Lambda\rho_{k}\|_{L^{2}}^{2}+(3-\beta_{1})\|\Lambda d_{k}\|_{L^{2}}^{2}-\|d_{k}\|_{L^{2}}^{2}+(1-\beta_{1}-16\beta_{2})\|\Lambda\theta_{k}\|_{L^{2}}^{2}\nonumber\\
 &+(4-25)\|\theta_{k}\|_{L^{2}}^{2}+\frac{3}{4}\| n_{0,k}\|_{L^{2}}^{2}+
 \Big(\frac{3}{4}-\beta_1\Big)\|M_k\|_{L^2}^2+\frac{3}{4}\beta_2\|\Lambda n_{0,k}\|_{L^2}^2+\beta_2\|\Lambda M_k\|_{L^2}^2\nonumber\\
&\,\qquad \leq\int_{\mathbb{R}^{3}}\Big(\rho_{k}\mathcal{N}_{1}^{k}+d_{k}\mathcal{D}_2^{k}+\theta_{k}\mathcal{N}_{3}^{k}+n_{0,k}\mathcal{N}_{4}^{k}\Big)\mathrm{d}x-\beta_1\int_{\mathbb{R}^3}\Big(d_{k}\Lambda\,\mathcal{N}_{1}^{k}+\Lambda\rho_{k}\mathcal{D}_2^{k}\Big)          \mathrm{d}x\nonumber\\
&\qquad \quad +3\beta_{1}\int_{\mathbb{R}^{3}}\nabla\rho_{k}\cdot\nabla \mathcal{N}_{1}^{k}\mathrm{d}x+\beta_{2}\int_{\mathbb{R}^{3}}\Lambda n_{0,k}\Lambda\,\mathcal{N}_{4}^{k}\mathrm{d}x.   
\end{align}
Therefore, by choosing the positive constants $\beta_1<$ and  $\beta_2$ ($\beta_2<\beta_1$) 
  satisfying $\beta_1<\frac{1}{4}, \beta_2<\frac{1}{64}$,
and a positive integer $k_1$ with $2^{2k_1-4}>25$, 
 it holds that, for $k>k_1>0$,
\begin{align}\label{G3.23}
(3-\beta_1)\|\Lambda d_{k}\|_{L^{2}}^{2}-\|d_{k}\|_{L^{2}}^{2}\geq&\,\|\Lambda d_{k}\|_{L^{2}}^{2},\nonumber\\
(1-\beta_1-16\beta_2)\|\Lambda \theta_{k}\|_{L^{2}}^{2}-25\|\theta_{k}\|_{L^{2}}^{2}\geq&\,\frac{1}{4}\|\Lambda \theta_{k}\|_{L^{2}}^{2}.    
\end{align}
Here the following facts have been used:
\begin{align*}
\|(\theta_k,d_k)\|_{L^2}&\leq 2^{k_1-1}\|(\theta_k,d_k)\|_{L^2}\leq \|\Lambda(\theta_k,d_k)\|_{L^2},\nonumber\\
\|\nabla \rho_k\|_{L^2}& \backsim \|\Lambda \rho_k\|_{L^2}.
\end{align*}
Combining the estimate \eqref{G3.23}  and the inequality \eqref{G3.22} gives
\begin{align}\label{G3.25}
&\frac{1}{2}\frac{\mathrm{d}}{\mathrm{d}t}\mathcal{L}_{k}(t)+\frac{\beta_{1}}{2}\|\Lambda\rho_{k}\|_{L^{2}}^{2}+\|\Lambda d_{k}\|_{L^{2}}^{2}+\frac{1}{4}\|\Lambda\theta_{k}\|_{L^{2}}^{2}+\|\theta_{k}\|_{L^{2}}^{2}+\frac{3}{4}\| n_{0,k}\|_{L^{2}}^{2}\nonumber\\
&+\frac{1}{2}\|M_{k}\|_{L^{2}}^{2}+\frac{3}{4}\beta_2\|\Lambda n_{0,k}\|_{L^2}^2+\beta_2\|\Lambda M_k\|_{L^2}^2\nonumber\\
&\quad \leq\int_{\mathbb{R}^{3}}\Big(\rho_{k}\mathcal{N}_{1}^{k}+d_{k}\mathcal{D}_2^{k}+\theta_{k}\mathcal{N}_{3}^{k}+n_{0,k}\mathcal{N}_{4}^{k}\Big)\mathrm{d}x-\beta_1\int_{\mathbb{R}^3}\Big(d_{k}\Lambda\,\mathcal{N}_{1}^{k}+\Lambda\rho_{k}\mathcal{D}_2^{k}\Big)          \mathrm{d}x\nonumber\\
&\,\quad \quad +3\beta_{1}\int_{\mathbb{R}^{3}}\nabla\rho_{k}\cdot\nabla \mathcal{N}_{1_1}^{k}\mathrm{d}x+3\beta_{1}\int_{\mathbb{R}^{3}}\nabla\rho_{k}\cdot\nabla \mathcal{N}_{1_2}^{k}\mathrm{d}x+\beta_{2}\int_{\mathbb{R}^{3}}\Lambda n_{0,k}\Lambda\,\mathcal{N}_{4}^{k}\mathrm{d}x.  
\end{align}
Here, we decompose  $\mathcal{N}_1=-\mathcal{N}_{1_1}-\mathcal{N}_{1_2}$,
where $\mathcal{N}_{1_1}=u\cdot\nabla\rho$ and $\mathcal{N}_{1_2}=\rho{\rm div}u$.
By Young's inequality and the fact $\|(\rho_k, d_k,\theta_k,n_{0,k})\|_{L^2}\leq \|\Lambda(\rho_k, d_k,\theta_k,n_{0,k})\|_{L^2}$, we compute
\begin{align}\label{G3.26}
&\int_{\mathbb{R}^{3}}\Big(\rho_{k}\mathcal{N}_{1}^{k}+d_{k}\mathcal{D}_2^{k}
+\theta_{k}\mathcal{N}_{3}^{k}+n_{0,k}\mathcal{N}_{4}^{k}\Big)\mathrm{d}x
-\beta_1\int_{\mathbb{R}^3}\Big(d_{k}\Lambda\,\mathcal{N}_{1}^{k}+\Lambda\rho_{k}D_2^{k}\Big)          \mathrm{d}x\nonumber\\
&+3\beta_{1}\int_{\mathbb{R}^{3}}\nabla\rho_{k}\cdot\nabla \mathcal{N}_{1_2}^{k}\mathrm{d}x+\beta_{2}\int_{\mathbb{R}^{3}}\Lambda n_{0,k}\Lambda\,\mathcal{N}_{4}^{k}\mathrm{d}x\nonumber\\
&\quad \leq \frac{1}{8}\beta_1\|\Lambda \rho_k\|_{L^2}^2+\frac{1}{2}\|\Lambda d_k\|_{L^2}^2
+\frac{1}{4}\beta_2\|\Lambda n_{0,k}\|_{L^2}^2+\frac{1}{8}\|\Lambda \theta_k\|_{L^2}^2\nonumber\\
&\,\qquad +C\|(\mathcal{N}_1^k,\mathcal{D}_2^k,\mathcal{N}_3^k,\mathcal{N}_4^k,\nabla\mathcal{N}_{1_2}^k,\nabla\mathcal{N}_4^k)\|_{L^2}^2,
\end{align}
for all integer $k>k_1>0$.
Now, we estimate the remaining term on the right-hand side of \eqref{G3.19}:
\begin{align}\label{G3.27}
3\beta_{1}\int_{\mathbb{R}^{3}}\nabla\rho_{k}\cdot\nabla \mathcal{N}_{1_1}^{k}\mathrm{d}x
\leq&\,3\beta_{1}\big|\big((\nabla u)^{T}\cdot\nabla\rho)_{k}|\nabla\rho_{k}\big)\big|+3\beta_{1}\big|\big([\dot{\Delta}_{k},u\cdot\nabla]\nabla\rho|\nabla\rho_{k}\big)\big|\nonumber\\
&+3\beta_{1}\big|\big((u\cdot\nabla)\nabla\rho_{k}|\nabla\rho_{k}\big)\big|\nonumber\\
\leq&\,\frac{\beta_{1}}{16}\|\nabla\rho_{k}\|_{L^{2}}^{2}+C\|((\nabla u)^{T}\cdot\nabla\rho)_{k}\|_{L^{2}}^{2}\nonumber\\
&+C\big|\big([\dot{\Delta}_{k},u\cdot\nabla\big]\nabla\rho|\nabla\rho_{k}\big)\big|+C\|\nabla u\|_{L^{\infty}}\|\nabla\rho_{k}\|_{L^{2}}^{2}.    
\end{align}
Using Lemma \ref{L2.7} and Young’s inequality, we have
\begin{align}\label{G3.28}
\big|\big([\dot{\Delta}_{k},u\cdot\nabla]\nabla\rho|\nabla\rho_{k}\big)\big|&\leq C\|\nabla u\|_{L^{\infty}}\|\nabla\rho_{k}\|_{L^{2}}^{2}+C\|\nabla^{2}\rho\|_{L^{\infty}}\|\dot{\Delta}_{k}u\|_{L^{2}}\|\nabla\rho_{k}\|_{L^{2}}\nonumber\\
&\quad+C\|\nabla u\|_{L^{\infty}}\|\nabla\rho_{k}\|_{L^{2}}\sum_{l\geq k-1}2^{k-l}\|\nabla\rho_{l}\|_{L^{2}}\nonumber\\
&\leq\frac{\beta_{1}}{16}\|\nabla\rho_{k}\|_{L^{2}}^{2}+C\|\nabla u\|_{H^{2}}\|\nabla\rho_{k}\|_{L^{2}}^{2}+C\|\nabla^{2}\rho\|_{H^{2}}^{2}\|u_{k}\|_{L^{2}}^{2}\nonumber\\
&\quad+C\|\nabla u\|_{H^{2}}^{2}\bigg(\sum_{l\geq k-1}2^{k-l}\|\nabla\rho_{l}\|_{L^{2}}\bigg)^{2}.
\end{align}
Plugging \eqref{G3.26}--\eqref{G3.28} into \eqref{G3.25} yields
\begin{align}\label{G3.29}
&\frac{1}{2}\frac{\mathrm{d}}{\mathrm{d}t}\mathcal{L}_{k}(t)
+\frac{{1}}{4}\beta_1\|\Lambda\rho_{k}\|_{L^2}^{2}+\frac{1}{2}\|\Lambda d_{k}\|_{L^2}^{2}+\frac{1}{8}\|\Lambda\theta_{k}\|_{L^2}^{2}+\|\theta_{k}\|_{L^2}^{2}\nonumber\\
&+\frac{3}{4}\|n_{0,k}\|_{L^2}^{2}+\frac{1}{2}\| M_{k}\|_{L^2}^{2}+\frac{1}{2}\beta_2\|\Lambda n_{0,k}\|_{L^2}^{2}+\beta_2\|\Lambda M_{k}\|_{L^2}^{2}\nonumber\\
&\quad \leq C\|(\mathcal{N}_1^k,\mathcal{D}_2^k,\mathcal{N}_3^k,\mathcal{N}_4^k,
\nabla\mathcal{N}_{1_2}^k,\nabla\mathcal{N}_4^k)\|_{L^2}^2+C\|((\nabla u)^{T}\cdot\nabla\rho)_{k}\|_{L^{2}}^{2}\nonumber\\
&\,\qquad +C\|\nabla u\|_{H^{2}}\|\nabla\rho_{k}\|_{L^{2}}^{2}+C\|\nabla^{2}\rho\|_{H^{2}}^{2}\|u_{k}\|_{L^{2}}^{2}+C\|\nabla u\|_{H^{2}}^{2}\bigg(\sum_{l\geq k-1}2^{k-l}\|\nabla\rho_{l}\|_{L^{2}}\bigg)^{2}.    
\end{align}
Moreover, we shall give the estimates of the first-order derivative of $d_k,\theta_k$. It follows from \eqref{G3.12}$_2$--\eqref{G3.12}$_3$ that
\begin{align}\label{G3.30}
&\frac{1}{2}\frac{\mathrm{d}}{\mathrm{d}t}\|(\Lambda d_{k},\Lambda\theta_{k})\|_{L^{2}}^{2}+3\|\Lambda^{2}d_{k}\|_{L^{2}}^{2}+\|\Lambda^{2}\theta_{k}\|_{L^{2}}^{2}+4\|\Lambda\theta_{k}\|_{L^{2}}^{2}\nonumber\\
&\,\qquad =(\Lambda n_{0,k}|\Lambda\theta_{k})+(\Lambda M_{k}|\Lambda d_{k})-(\Delta\rho_{k}|\Lambda d_{k})+(\Lambda \mathcal{D}_{2}^{k}|\Lambda d_{k})+(\Lambda\,\mathcal{N}_{3}^{k}|\Lambda \theta_{k}).    
\end{align}
Then, using Cauchy-Schwarz's inequality, we have
\begin{align}\label{G3.31}
|(\Lambda n_{0,k}|\Lambda\theta_{k})|\leq&\,2\|\Lambda \theta_{k}\|_{L^{2}}^{2}+\frac{1}{8}\|\Lambda n_{0,k}\|_{L^{2}}^{2},\nonumber\\
|(\Lambda M_{k}|\Lambda d_{k})|\leq&\,2\|\Lambda^2 d_{k}\|_{L^{2}}^{2}+\frac{1}{8}\|M_{k}\|_{L^{2}}^{2},\nonumber
\\|(\Delta\rho_{k}|\Lambda d_{k})|\leq&\,\frac{1}{4}\|\Lambda^2 d_{k}\|_{L^{2}}+\|\Lambda \rho_{k}\|_{L^{2}}^{2},\nonumber\\
|(\Lambda \mathcal{D}_2^{k}|\Lambda d_{k})|\leq&\,\|\mathcal{D}_2^{k}\|_{L^{2}}^{2}+\frac{1}{4}\|\Lambda^{2}d_{k}\|_{L^{2}}^{2},\nonumber\\
|(\Lambda\,\mathcal{N}_{3}^{k}|\Lambda\theta_{k})|\leq&\,\frac{1}{2}\|\mathcal{N}_{3}^{k}\|_{L^{2}}^{2}+\frac{1}{2}\|\Lambda^{2}\theta_{k}\|_{L^{2}}^{2}.  
\end{align}
Putting \eqref{G3.31} into \eqref{G3.30} yields
\begin{align}\label{G3.32}
&\frac{1}{2}\frac{\mathrm{d}}{\mathrm{d}t}\|(\Lambda d_{k},\Lambda\theta_{k})\|_{L^{2}}^{2}+\frac{1}{2}\|\Lambda^{2}d_{k}\|_{L^{2}}^{2}+\frac{1}{2}\|\Lambda^{2}\theta_{k}\|_{L^{2}}^{2}+2\|\Lambda\theta_{k}\|_{L^{2}}^{2}\nonumber\\
&\,\,\qquad\leq\|\Lambda \rho_k\|_{L^2}^2+\frac{1}{8}\|M_k\|_{L^2}^2+\frac{1}{8}\|\Lambda n_{0,k}\|_{L^2}^2+\|\mathcal{D}_2^k\|_{L^2}^2+\|\mathcal{N}_3^k\|_{L^2}^2.        
\end{align}
For the estimate of $(\mathcal{P}u)_k$, $(\mathcal{P}n_1)_k$, it
follows from \eqref{G3.10} and Young's inequality that
\begin{align*}
&\frac{1}{2}\frac{\mathrm{d}}{\mathrm{d}t}\big\|\big((\mathcal{P}u)_{k}, (\mathcal{P}n_1)_{k},\Lambda(\mathcal{P}u)_{k}, \Lambda(\mathcal{P}n_1)_{k}\big)\big\|_{L^{2}}^{2}+\frac{2}{3}\|\Lambda (\mathcal{P}u)_{k}\|_{L^{2}}^{2}\nonumber\\
&+\frac{1}{2}\|\Lambda^2(\mathcal{P}u)_{k}\|_{L^{2}}^{2}+\frac{1}{4}\| (\mathcal{P}n_1)_{k}\|_{L^{2}}^{2}+\frac{1}{4}\|\Lambda (\mathcal{P}n_1)_{k}\|_{L^{2}}^{2}\nonumber\\
&\,\,\quad\leq\frac{5}{12}\|(\mathcal{P} u)_{k}\|_{L^{2}}^{2}+C\|(\mathcal{P}\mathcal{N}_2)_k\|_{L^2}^2,
\end{align*}
for all integer $k>k_1>0$.
It is easy to see that
\begin{align*}
\|(\mathcal{P}u)_{k}\|_{L^{2}}^{2}\leq\|\Lambda(\mathcal{P}u)_{k}\|_{L^{2}}^{2},   
\end{align*}
which implies that 
\begin{align}\label{G3.35}
&\frac{1}{2}\frac{\mathrm{d}}{\mathrm{d}t}\big\|\big((\mathcal{P}u)_{k}, (\mathcal{P}n_1)_{k},\Lambda(\mathcal{P}u)_{k}, \Lambda(\mathcal{P}n_1)_{k}\big)\big\|_{L^{2}}^{2}+\frac{1}{4}\|\Lambda (\mathcal{P}u)_{k}\|_{L^{2}}^{2}\nonumber\\
&\qquad +\frac{1}{2}\|\Lambda^2(\mathcal{P}u)_{k}\|_{L^{2}}^{2}+\frac{1}{4}\| (\mathcal{P}n_1)_{k}\|_{L^{2}}^{2}+\frac{1}{4}\|\Lambda (\mathcal{P}n_1)_{k}\|_{L^{2}}^{2} \leq C\|(\mathcal{P}\mathcal{N}_2)_k\|_{L^2}^2.    
\end{align}
Then, the sum of inequalities \eqref{G3.29}, $\beta_3\times$\eqref{G3.32} 
and \eqref{G3.35} gives
\begin{align*}
&\frac{1}{2}\frac{\mathrm{d}}{\mathrm{d}t}\Big(\mathcal{L}_{k}(t)+\beta_{3}\|(\Lambda d_{k},\Lambda\theta_{k})(t)\|_{L^{2}}^{2}+\big\|\big((\mathcal{P}u)_{k}, (\mathcal{P}n_1)_{k},\Lambda(\mathcal{P}u)_{k}, \Lambda(\mathcal{P}n_1)_{k}\big)(t)\big\|_{L^{2}}^{2}\Big)\nonumber\\
&+\Big(\frac{1}{4}\beta_1-\beta_3\Big)\|\Lambda\rho_{k}(t)\|_{L^{2}}^{2}+\frac{1}{2}\|\Lambda d_{k}(t)\|_{L^{2}}^{2}+\frac{1}{2}\beta_{3}\|\Lambda^{2}d_{k}(t)\|_{L^{2}}^{2}+\| \theta_k\|_{L^2}^2+\frac{1}{8}\|\Lambda \theta_k\|_{L^2}^2\nonumber\\
&+\frac{1}{2}\beta_3\|\Lambda^2 \theta_k\|_{L^2}^2+\frac{3}{4}\|n_{0,k}(t)\|_{L^{2}}^{2}+\Big(\frac{1}{2}\beta_2-\frac{1}{8}\beta_3   \Big)\|\Lambda n_{0,k}(t)\|_{L^{2}}^{2}+\Big(\frac{1}{2}-\frac{1}{8}\beta_3   \Big)\|M_{k}(t)\|_{L^{2}}^{2}\nonumber\\
&+\beta_{2}\|\Lambda M_{k}(t)\|_{L^{2}}^{2}+\frac{1}{4}\big\|\big(\Lambda(\mathcal{P}u)_{k}, (\mathcal{P}n_1)_{k},\Lambda^2(\mathcal{P}u)_{k}, \Lambda(\mathcal{P}n_1)_{k}\big)(t)\big\|_{L^{2}}^{2}\nonumber\\
&\,\quad \leq C\|(\mathcal{N}_1^k,\mathcal{N}_2^k,\mathcal{N}_3^k,\mathcal{N}_4^k,\nabla\mathcal{N}_{1_2}^k,\nabla\mathcal{N}_4^k)(t)\|_{L^2}^2+C\|((\nabla u)^{T}\cdot\nabla\rho)_{k}(t)\|_{L^{2}}^{2}\nonumber\\
&\,\qquad \, +C\|\nabla u\|_{H^{2}}\|\nabla\rho_{k}\|_{L^{2}}^{2}+C\|\nabla^{2}\rho\|_{H^{2}}^{2}\|u_{k}\|_{L^{2}}^{2}+C\|\nabla u\|_{H^{2}}^{2}\bigg(\sum_{l\geq k-1}2^{k-l}\|\nabla\rho_{l}\|_{L^{2}}\bigg)^{2},  
\end{align*}
where the constant $\beta_3$ satisfying $0<\beta_3<\beta_2$.
Here, we have used the inequalities
\begin{align*}
\|\mathcal{D}_2^k\|_{L^2}\leq C\|\mathcal{N}_2^k\|_{L^2} \quad\text{and}\quad    \|(\mathcal{P} \mathcal{N}_2)_k\|_{L^2}\leq C\|\mathcal{N}_2^k\|_{L^2}.
\end{align*}
For simplicity, we set
\begin{align*}
\mathcal{H}_{k}(t)=\mathcal{L}_{k}(t)+\beta_{2}\|(\Lambda d_{k},\Lambda\theta_{k})(t)\|_{L^{2}}^{2}+\big\|\big((\mathcal{P}u)_{k}, (\mathcal{P}n_1)_{k},\Lambda(\mathcal{P}u)_{k}, \Lambda(\mathcal{P}n_1)_{k}\big)(t)\big\|_{L^{2}}^{2}.   
\end{align*}
From identities \eqref{G3.6}--\eqref{G3.7}, Young's inequality, and Bernstein's inequality, we have
\begin{align*}
\mathcal{H}_{k}(t)\backsim \|(\rho_{k},u_{k},\theta_{k},n_{0,k},n_{1,k})(t)\|_{L^{2}}^{2}+2^{2k}\|(\rho_{k},u_{k},\theta_{k},n_{0,k},n_{1,k})(t)\|_{L^{2}}^{2}.   
\end{align*}
Thus, there exists a constant $\lambda_1>0$ such that for any $k>k_1>0$,
\begin{align}\label{G3.40}
&\frac{\mathrm{d}}{\mathrm{d}t}\mathcal{H}_{k}(t)+\lambda_1\Big(\|(\theta_{k},n_{0,k},n_{1,k})(t)\|_{L^{2}}^{2}+2^{2k}
\|(\rho_{k},u_{k},\theta_{k},n_{0,k},n_{1,k})(t)\|_{L^{2}}^{2}+2^{4k}\|(u_{k},\theta_{k})(t)\|_{L^{2}}^{2}\Big)\nonumber\\
&\,\quad \leq C\|(\mathcal{N}_1^k,\mathcal{N}_2^k,\mathcal{N}_3^k,\mathcal{N}_4^k,
\nabla\mathcal{N}_{1_2}^k,\nabla\mathcal{N}_4^k)(t)\|_{L^2}^2+C\|((\nabla u)^{T}\cdot\nabla\rho)_{k}(t)\|_{L^{2}}^{2}\nonumber\\
&\qquad \, +C\|\nabla u\|_{H^{2}}\|\nabla\rho_{k}\|_{L^{2}}^{2}+C\|\nabla^{2}\rho\|_{H^{2}}^{2}\|u_{k}\|_{L^{2}}^{2}+C\|\nabla u\|_{H^{2}}^{2}\bigg(\sum_{l\geq k-1}2^{k-l}\|\nabla\rho_{l}\|_{L^{2}}\bigg)^{2}.   
\end{align}
Integrating \eqref{G3.40} over $\mathbb{R}^3$, we consequently deduce 
\begin{align}\label{G3.41}
&\|(\rho_{k},u_{k},\theta_{k},n_{0,k},n_{1,k})(t)\|_{L^{2}}^{2}+2^{2k}\|(\rho_{k},u_{k},\theta_{k},n_{0,k},n_{1,k})(t)\|_{L^{2}}^{2}\nonumber\\
&+\int_{0}^{t}\Big(\|(\theta_{k},n_{0,k},n_{1,k})(\tau)\|_{L^{2}}^{2}+2^{2k}
\|(\rho_{k},u_{k},\theta_{k},n_{0,k},n_{1,k})(\tau)\|_{L^{2}}^{2}+2^{4k}\|(u_{k},\theta_{k})(\tau)\|_{L^{2}}^{2}\Big)\mathrm{d}\tau
\nonumber\\
&\,\quad \leq C\Big(\|(\rho_{k},u_{k},\theta_{k},n_{0,k},n_{1,k})(0)\|_{L^{2}}^{2}+2^{2k}\|(\rho_{k},u_{k},\theta_{k},n_{0,k},n_{1,k})(0)\|_{L^{2}}^{2}\Big)\nonumber\\
&\qquad \, +C\int_{0}^{t}\Big(\|(\mathcal{N}_{1}^{k},\mathcal{N}_{2}^{k},\mathcal{N}_{3}^{k},\mathcal{N}_{4}^{k},\nabla \mathcal{N}_{1_2}^{k},\nabla \mathcal{N}_{4}^{k})(\tau)\|_{L^{2}}^{2}+\|((\nabla u)^{T}\cdot\nabla\rho)_{k}(\tau)\|_{L^{2}}^{2}\Big)\mathrm{d}\tau\nonumber\\
&\qquad \,+C\int_{0}^{t}\bigg(\|\nabla u\|_{H^2}\|\nabla\rho_{k}\|_{L^{2}}^{2}+\|\nabla^{2}\rho\|_{H^{2}}^{2}\|u_{k}\|_{L^{2}}^{2}+\|\nabla u\|_{L^{2}}^{2}\bigg(\sum_{l\geq k-1}2^{k-l}\|\nabla\rho_{l}\|_{L^{2}}\bigg)^{2}\bigg)\mathrm{d}\tau,   
\end{align}
for any integer $k>k_1>0$.
\begin{prop}\label{P3.1}
Let $s\geq 1$ be a real number. Then,  for the smooth solutions to system \eqref{I-3}--\eqref{I--3},  it holds that 
\begin{align}\label{G3.42}
&\|(\rho,u,\theta,n_{0},n_{1})^{S}(t)\|_{\dot{B}_{2,2}^{s-1}}^{2}+\|(\rho,u,\theta,n_{0},n_1)^{S}(t)\|_{\dot{B}_{2,2}^{s}}^{2}\nonumber\\
&+\int_{0}^{t}\Big(\|(\theta,n_0,n_1)^{S}(\tau)\|_{\dot{B}_{2,2}^{s-1}}^{2}+\|(\rho,u,\theta,n_{0},n_1)^{S}(\tau)\|_{\dot{B}_{2,2}^{s}}^{2}+\|(u,\theta)^{S}(\tau)\|_{\dot{B}_{2,2}^{s+1}}^{2}\Big)\mathrm{d}\tau\nonumber\\
&\,\quad \leq C\|(\rho,u,\theta,n_{0},n_1)^{S}(0)\|_{\dot{B}_{2,2}^{s-1}}^{2}+C\|(\rho,u,\theta,n_{0},n_1)^{S}(0)\|_{\dot{B}_{2,2}^{s}}^{2}\nonumber\\
&\qquad \, +C\int_{0}^{t}\Big(\|(\mathcal{N}_{1_2},\mathcal{N}_4)^{S}(\tau)\|_{\dot{B}_{2,2}^{s}}^{2}+\big\|\big(\mathcal{N}_1,\mathcal{N}_2,\mathcal{N}_3,\mathcal{N}_4,(\nabla u)^{\top}\cdot\nabla\rho\big)^S(\tau)\|_{\dot{B}_{2,2}^{s-1}}^{2}\Big)\mathrm{d}\tau\nonumber\\
&\qquad \, +C\int_{0}^{t}\Big(\|\nabla^{2}\rho\|_{H^{2}}^{2}\|u^{S}\|_{\dot{B}_{2,2}^{s-1}}^{2}+\|\nabla u(\tau)\|_{H^{2}}\|\nabla\rho^{S}(\tau)\|_{\dot{B}_{2,2}^{s-1}}^{2}\Big)\mathrm{d}\tau\nonumber\\
&\qquad \, +C\int_{0}^{t}\|\nabla u\|_{H^{2}}^{2}\sum_{k> k_{1}}2^{2k(s-1)}\bigg(\sum_{l\geq k-1}2^{k-l}\|\nabla\rho_{l}\|_{L^{2}}\bigg)^{2}\mathrm{d}\tau,    
\end{align}
where $(\cdot)^S$ represents the high-frequency part of $(\cdot)$.
\end{prop}
\begin{proof}
Thanks to the definition of \eqref{G2.7} and the weighted $l^2$ summation over $k$ with $k\geq k_1+1$ in \eqref{G3.41}, we get desired \eqref{G3.42}. In view of Fourier transform, by denoting $R_0:=2^{k_1+1}$, we notice that the high-frequency part means $|\xi|\geq R_0$.
\end{proof}
\subsection{Estimates of the low-frequency part}
In this  subsection, we  investigate  the estimate of classical solutions to
problem \eqref{G3.8}--\eqref{G3.11}. To begin with, we apply 
Fourier transform to \eqref{I-4} to obtain
\begin{equation}\label{G3.43}
\left\{
\begin{aligned}
& \partial_t \widehat{\rho}+|\xi|\widehat{d}=\widehat{\mathcal{N}}_1, \\
& \partial_t \widehat{d}-|\xi|\widehat{\rho}-|\xi|\widehat{\theta}+3|\xi|^2 \widehat{d}-\widehat{M}=\widehat{\mathcal{D}_2}, \\
& \partial_t \widehat{\theta}+|\xi|\widehat{d}+|\xi|^2\widehat{\theta}+4\widehat{\theta}-\widehat{n_{0}}=\widehat{\mathcal{N}_3},\\
& \partial_t \widehat{n_{0}}+|\xi|\widehat{M}-4\widehat{\theta}+\widehat{n_{0}}=\widehat{\mathcal{N}_4},\\
& \partial_t \widehat{M}-|\xi|\widehat{n_{0}} +\widehat{M}=0.
\end{aligned}\right.
\end{equation}
Similar to the method developed in \cite{DD-JEE-2014}, we use a change of variables as follows:
\begin{align}\label{G3.44} \left(
\begin{matrix}
\widehat{G}  \\
\widehat{F}   \\
\widehat{J}  \\
\widehat{K}  \\
\widehat{M}  \\
\end{matrix}\right):=\mathbb{T}\left(
\begin{matrix}
\widehat{\rho}  \\
\widehat{d}   \\
\widehat{\theta}  \\
\widehat{n_0}  \\
\widehat{M}  \\
\end{matrix}\right),
\end{align}
where
\begin{align}\label{G3.45}
\mathbb{T}=\left(
\begin{matrix}
1 & 0 & 0 & 0 & |\xi|  \\
0 & 1 & 0 & 0 & 1   \\
0 & 0 & 4 & -1 & 0  \\
0 & 0 & 1 & 1 & 0  \\
0 & 0 & 0 & 0 & 1  \\
\end{matrix}\right).
\end{align}
Thus, \eqref{G3.43} can be rewritten as
\begin{equation}\label{G3.46}
\left\{
\begin{aligned}
& \partial_t \widehat{G}+|\xi|\widehat{F}+\frac{1}{5}|\xi|^2\widehat{J}-\frac{4}{5}|\xi|^2\widehat{K}=\widehat{\mathcal{N}}_1, \\
& \partial_t \widehat{F}-|\xi|\widehat{K}-|\xi|\widehat{G}+3|\xi|^2 \widehat{F}-2|\xi|^2\widehat{M}=\widehat{\mathcal{D}_2}, \\
& \partial_t \widehat{J}+4|\xi|\widehat{F}-5|\xi|\widehat{M}+5\widehat{J}+\frac{4}{5}|\xi|^2\widehat{J}+\frac{4}{5}|\xi|^2\widehat{K}=4\widehat{\mathcal{N}_3}-\widehat{\mathcal{N}_4},\\
& \partial_t \widehat{K}+|\xi|\widehat{F}+\frac{1}{5}|\xi|^2\widehat{J}+\frac{1}{5}|\xi|^2\widehat{K}=\widehat{\mathcal{N}_3}+\widehat{\mathcal{N}_4},\\
& \partial_t \widehat{M}-\frac{4}{5}|\xi|\widehat{K}+\frac{1}{5}|\xi|\widehat{M} +\widehat{M}=0.
\end{aligned}\right.
\end{equation}
By using energy method, we easily get
\begin{align}\label{G3.47}
&\frac{1}{2}\frac{\mathrm{d}}{\mathrm{d}t}\Big(|\widehat{G}|^{2}+|\widehat{F}|^{2}+|\widehat{K}|^{2}\Big)+3|\xi|^{2}|\widehat{F}|^{2}+\frac{1}{5}|\xi|^{2}|\widehat{K}|^{2}\nonumber\\
&\,\quad =-\frac{1}{5}|\xi|^{2}\mathbf{Re}(\widehat{J}|{\widehat{G}})+\frac{4}{5}|\xi|^{2}\mathbf{Re}(\widehat{K}|{\widehat{G}})-\frac{1}{5}|\xi|^{2}\mathbf{Re}(\widehat{J}|{\widehat{K}})+2|\xi|^{2}\mathbf{Re}(\widehat{M}|{\widehat{F}})\nonumber\\
&\,\qquad +\mathbf{Re}(\widehat{\mathcal{N}}_{1}|{\widehat{G}})+\mathbf{Re}(\widehat{\mathcal{D}}_{2}|{\widehat{F}})+\mathbf{Re}(\widehat{\mathcal{N}}_{3}|{\widehat{K}})+\mathbf{Re}(\widehat{\mathcal{N}}_{4}|{\widehat{K}}),   
\end{align}
and
\begin{align}\label{G3.48}
&\frac{1}{2}\frac{\mathrm{d}}{\mathrm{d}t}\Big(|\widehat{M}|^{2}+\frac{1}{25}|\widehat{J}|^{2}\Big)+|\widehat{M}|^{2}+\frac{1}{5}|\widehat{J}|^{2}+\frac{4}{125}|\xi|^2|\widehat{J}|^2\nonumber\\
&\,\quad =-\frac{4}{25}|\xi|\mathbf{Re}(\widehat{F}|{\widehat{J}})+\frac{4}{5}|\xi|\mathbf{Re}(\widehat{K}|{\widehat{M}})-\frac{4}{125}|\xi|^{2}
\mathbf{Re}(\widehat{K}|{\widehat{J}})
\nonumber\\
& \qquad\, +\frac{4}{25}\mathbf{Re}
(\widehat{\mathcal{N}_3}|{\widehat{J}})-\frac{1}{25}\mathbf{Re}(\widehat{\mathcal{N}_4}|{\widehat{J}}).
\end{align}
Similar to \eqref{G3.47}--\eqref{G3.48}, it follows from \eqref{G3.46}$_1$--\eqref{G3.46}$_2$
that
\begin{align}\label{G3.49}
&-\frac{\mathrm{d}}{\mathrm{d}t}\big(|\xi|\mathbf{Re}(\widehat{G}|{\widehat{F}})\big)+|\xi|^2|\widehat{G}|^{2}-|\xi|^2|\widehat{F}|^{2}    \nonumber\\
&\,\quad =-|\xi|^{2}\mathbf{Re}(\widehat{K}|{\widehat{G}})+3|\xi|^3\mathbf{Re}(\widehat{G}
|{\widehat{F}})-2|\xi|^3\mathbf{Re}(\widehat{M}|{\widehat{G}})-\frac{4}
{5}|\xi|^3\mathbf{Re}(\widehat{K}|{\widehat{F}})\nonumber\\
&\,\qquad +\frac{1}
{5}|\xi|^3\mathbf{Re}(\widehat{J}|{\widehat{F}})-|\xi|\mathbf{Re}(\widehat{N_{1}}|{\widehat{F}})-|\xi|\mathbf{Re}(\widehat{D_{2}}|{\widehat{G}}).   
\end{align}
Adding up \eqref{G3.47} and $\frac{4}{5}\times$\eqref{G3.49} gives
\begin{align}\label{G3.50}
&\frac{1}{2}\frac{\mathrm{d}}{\mathrm{d}t}\Big(|\widehat{G}|^{2}+|\widehat{F}|^{2}+|\widehat{K}|^{2}-\frac{8}{5}|\xi|
\mathbf{Re}(\widehat{G}|{\widehat{F}})\Big)+\frac{4}{5}|\xi|^{2}
|\widehat{G}|^{2}+\frac{11}{5}|\xi|^{2}|\widehat{F}|^{2}+\frac{1}{5}|\xi|^{2}|\widehat{K}|^{2}
\nonumber\\
=&\,-\frac{1}{5}|\xi|^{2}\mathbf{Re}(\widehat{J}|{\widehat{G}})+2|\xi|^{2}\mathbf{Re}(\widehat{M}|{\widehat{F}})-\frac{1}{5}|\xi|^{2}\mathbf{Re}(\widehat{J}|{\widehat{K}})+\frac{12}{5}|\xi|^{3}\mathbf{Re}(\widehat{G}|{\widehat{F}})\nonumber\\
&-\frac{16}{25}|\xi|^{3}\mathbf{Re}(\widehat{K}|{\widehat{F}})+\frac{4}{25}|\xi|^3\mathbf{Re}(\widehat{J}|{\widehat{F}})-\frac{8}{5}|\xi|^3\mathbf{Re}(\widehat{M}|{\widehat{G}})-\frac{4}{5}|\xi|\mathbf{Re}(\widehat{\mathcal{N}_1}
|{\widehat{F}})
    \nonumber\\
&\,-\frac{4}{5}|\xi|\mathbf{Re}(\widehat{\mathcal{D}_2}|{\widehat{G}})
+\mathbf{Re}(\widehat{\mathcal{N}}_{1}|{\widehat{G}})+\mathbf{Re}(\widehat{\mathcal{D}}_{2}|{{\widehat{F}}})+\mathbf{Re}(\widehat{\mathcal{N}}_{3}|{\widehat{K}})+\mathbf{Re}(\widehat{\mathcal{N}}_{4}|{\widehat{K}})\nonumber\\
\leq&\,\frac{1}{5}|\xi|^2|\widehat{G}|^2+\frac{1}{20}|\xi|^2|\widehat{J}|^2
+|\xi|^2|\widehat{F}|^2+|\xi|^2|\widehat{M}|^2+\frac{1}{20}|\xi|^2|\widehat{K}|^2
+\frac{1}{5}|\xi|^2|\widehat{J}|^2+\frac{6}{5}|\xi|^3|\widehat{F}|^2\nonumber\\
&\,+\frac{6}{5}|\xi|^3|\widehat{G}|^2+\frac{8}{25}|\xi|^3|\widehat{K}|^2+\frac{8}{25}|\xi|^3|\widehat{F}|^2+\frac{2}{25}|\xi|^3|\widehat{F}|^2+\frac{2}{25}|\xi|^3|\widehat{J}|^2+\frac{4}{5}|\xi|^3|\widehat{M}|^2+\frac{4}{5}|\xi|^3|\widehat{G}|^2\nonumber\\
&\,-\frac{4}{5}|\xi|\mathbf{Re}(\widehat{\mathcal{N}_1}|{\widehat{F}})-\frac{4}{5}|\xi|\mathbf{Re}(\widehat{\mathcal{D}_2}|{\widehat{G}})+\mathbf{Re}(\widehat{\mathcal{N}}_{1}|{\widehat{G}})
+\mathbf{Re}(\widehat{\mathcal{D}}_{2}|{{\widehat{F}}})+\mathbf{Re}(\widehat{\mathcal{N}}_{3}
|{\widehat{K}})+\mathbf{Re}(\widehat{\mathcal{N}}_{4}|{\widehat{K}}).
\end{align}
By an argument similar to \eqref{G3.48}, we arrive at
\begin{align}\label{G3.51}
&\frac{1}{2}\frac{\mathrm{d}}{\mathrm{d}t}\Big(\frac{1}{2}|\widehat{M}|^{2}+\frac{1}{50}|\widehat{J}|^{2}\Big)+\frac{1}{2}|\widehat{M}|^{2}+\frac{1}{10}|\widehat{J}|^{2}+\frac{2}{125}|\xi|^2|\widehat{J}|^2\nonumber\\
&\,\quad \leq \frac{2}{25}|\xi|^2|\widehat{F}|^2+\frac{1}{50}|\widehat{J}|^2
+\frac{1}{10}|\xi|^2|\widehat{K}|^2+\frac{2}{5}|\widehat{M}|^2+\frac{2}{125}|\xi|^2|\widehat{J}|^2+\frac{1}{250}|\xi|^2|\widehat{K}|^2\nonumber\\
&\,\qquad +\frac{2}{25}\mathbf{Re}(\widehat{\mathcal{N}_3}|{\widehat{J}})-\frac{1}{50}\mathbf{Re}(\widehat{\mathcal{N}_4}|{\widehat{J}}).
\end{align}
By combining \eqref{G3.50} with \eqref{G3.51}, it is easy to infer that
\begin{align*}
&\frac{1}{2}\frac{\mathrm{d}}{\mathrm{d}t}\Big(|\widehat{G}|^{2}+|\widehat{F}|^{2}+|\widehat{K}|^{2}+\frac{1}{2}|\widehat{M}|^2+\frac{1}{50}|\widehat{J}|^2-\frac{8}{5}|\xi|\mathbf{Re}(\widehat{G}|{\widehat{F}})\Big)
\nonumber\\
&\,+\Big(\frac{2}{25}-\frac{1}{4}|\xi|^2-\frac{2}{25}|\xi|^3         \Big)|\widehat{J}|^2+\Big(\frac{1}{10}-|\xi|^2-\frac{4}{5}|\xi|^3         \Big)|\widehat{M}|^2+\Big(\frac{3}{5}-2|\xi|         \Big)|\xi|^2|\widehat{G}|^2
\nonumber\\
&\,+\Big(\frac{6}{5}-\frac{42}{25}|\xi|         \Big)|\xi|^2|\widehat{F}|^2
+\Big(\frac{1}{50}-\frac{2}{5}|\xi|         \Big)|\xi|^2|\widehat{K}|^2\nonumber\\
&\,\quad \leq-\frac{4}{5}|\xi|\mathbf{Re}(\widehat{\mathcal{N}_1}|{\widehat{F}})-\frac{4}{5}|\xi|\mathbf{Re}(\widehat{\mathcal{D}_2}|{\widehat{G}})+\mathbf{Re}(\widehat{\mathcal{N}}_{1}
|{\widehat{G}})+\mathbf{Re}(\widehat{\mathcal{D}}_{2}|{{\widehat{F}}})\nonumber\\
&\,\qquad +\mathbf{Re}(\widehat{\mathcal{N}}_{3}|{\widehat{K}})+\mathbf{Re}(\widehat{\mathcal{N}}_{4}|{\widehat{K}})+\frac{2}{25}\mathbf{Re}(\widehat{\mathcal{N}_3}|{\widehat{J}})
-\frac{1}{50}\mathbf{Re}(\widehat{\mathcal{N}_4}|{\widehat{J}}).    
\end{align*}
Now, we denote $r_0=\frac{1}{40}$ and take $k_0:=[{\rm log}_2 r_0]-1$.
For $|\xi|\leq r_0$, there exists a constant $\lambda_2>0$ such that 
\begin{align}\label{G3.53}
 &\frac{1}{2}\frac{\mathrm{d}}{\mathrm{d}t}\mathcal{L}_{l}(t,\xi)+\lambda_2\Big(|\widehat{J}|^2+|\widehat{M}|^2+|\xi|^2|\widehat{G}|^2+|\xi|^2|\widehat{F}|^2+|\xi|^2|\widehat{K}|^2\Big)
\nonumber\\
&\,\quad \leq-\frac{4}{5}|\xi|\mathbf{Re}(\widehat{\mathcal{N}_1}|{\widehat{F}})-\frac{4}{5}|\xi|\mathbf{Re}(\widehat{\mathcal{D}_2}|{\widehat{G}})+\mathbf{Re}(\widehat{\mathcal{N}}_{1}
|{\widehat{G}})+\mathbf{Re}(\widehat{\mathcal{D}}_{2}|{{\widehat{F}}})\nonumber\\
&\,\qquad +\mathbf{Re}(\widehat{\mathcal{N}}_{3}|{\widehat{K}})+\mathbf{Re}(\widehat{\mathcal{N}}_{4}|{\widehat{K}})+\frac{2}{25}\mathbf{Re}(\widehat{\mathcal{N}_3}|{\widehat{J}})-\frac{1}{50}\mathbf{Re}(\widehat{\mathcal{N}_4}|{\widehat{J}}),       
\end{align}
where 
\begin{align*}
\mathcal{L}_{l}(t,\xi):=\Big(|\widehat{G}|^{2}+|\widehat{F}|^{2}+|\widehat{K}|^{2}+\frac{1}{2}|\widehat{M}|^2+\frac{1}{50}|\widehat{J}|^2-\frac{8}{5}|\xi|\mathbf{Re}(\widehat{G}|{\widehat{F}})\Big).
\end{align*}
It is obvious that, for any $|\xi|\leq r_0$, we have
\begin{align}\label{G3.55}
\mathcal{L}_{l}(t,\xi)\sim|\widehat{\rho}|^{2}+|\widehat{d}|^{2}+|\widehat{\theta}|^{2}+|\widehat{n_0}|^{2}+|\widehat{M}|^{2},
\end{align}
which implies that
\begin{align}\label{G3.56}
 &\frac{1}{2}\frac{\mathrm{d}}{\mathrm{d}t}\mathcal{L}_{l}(t,\xi)+\lambda_2|\xi|^{2}\mathcal{L}_{l}(t,\xi)
\nonumber\\
&\,\quad \leq C\Big(-\mathbf{Re}(\widehat{\mathcal{N}_1}|{\widehat{F}})-\mathbf{Re}(\widehat{\mathcal{D}_2}|{\widehat{G}})+\mathbf{Re}(\widehat{\mathcal{N}}_{1}
|{\widehat{G}})+\mathbf{Re}(\widehat{\mathcal{D}}_{2}|{{\widehat{F}}})\Big)\nonumber\\
&\,\qquad +C\Big(\mathbf{Re}(\widehat{\mathcal{N}}_{3}|{\widehat{K}})+\mathbf{Re}(\widehat{\mathcal{N}}_{4}
|{\widehat{K}})+\mathbf{Re}(\widehat{\mathcal{N}_3}|{\widehat{J}})-\mathbf{Re}(\widehat{\mathcal{N}_4}|{\widehat{J}})\Big).       
\end{align}
Then, by using  Plancherel's theorem, we get the estimate 
of $(\rho_k,d_k,\theta_k,n_{0,k}, M_k)$ as follows:
\begin{align}\label{G3.57}
&\|(\rho_{k},d_{k},\theta_{k},n_{0,k},M_k)(t)\|_{L^{2}}^{2}+\int_{0}^{t}2^{2k}\|(\rho_{k},d_{k},\theta_{k},n_{0,k},M_k)(\tau)\|_{L^{2}}^{2}\mathrm{d}\tau\nonumber\\
&\,+\int_{0}^{t}\big\|\big((4\theta_k-n_{0,k}),M_{k}\big)(\tau)\big\|_{L^2}^2\mathrm{d}\tau\nonumber\\
&\,\quad \leq C\|(\rho_{k},d_{k},\theta_{k},n_{0,k},M_k)(0)\|_{L^{2}}^{2}+C\int_{0}^{t}\int_{\mathbb{R}^{3}}\widehat{\phi}_{i}^{2}(\xi)\mathbf{Re}(\widehat{\mathcal{N}_1}|{\widehat{G}})\mathrm{d}\xi\mathrm{d}\tau\nonumber\\
&\,\qquad -C\int_{0}^{t}\int_{\mathbb{R}^{3}}\widehat{\phi}_{i}^{2}(\xi)\Big(\mathbf{Re}(\widehat{\mathcal{D}_2}|{\widehat{G}})+\mathbf{Re}(\widehat{\mathcal{N}_4}|{\widehat{J}})
+\mathbf{Re}(\widehat{\mathcal{N}_1}|{\widehat{F}})\Big)\mathrm{d}\xi\mathrm{d}\tau\nonumber\\
&\,\qquad +C\int_{0}^{t}\int_{\mathbb{R}^{3}}\widehat{\phi}_{i}^{2}(\xi)\Big(\mathbf{Re}(\widehat{\mathcal{N}_3}|{\widehat{K}})+\mathbf{Re}(\widehat{\mathcal{N}_4}|{\widehat{K}})+\mathbf{Re}(\widehat{\mathcal{       D}_2}|{\widehat{F}})+\mathbf{Re}(\widehat{\mathcal{       N}_3}|{\widehat{J}})\Big)\mathrm{d}\xi\mathrm{d}\tau,   
\end{align}
for any integer $k\leq k_0$.
Making use of Young's inequality and  H\"{o}lder’s inequality, we have
\begin{align*}
& \int_{0}^{t}\int_{\mathbb{R}^{3}}\widehat{\phi}_{i}^{2}(\xi)
\mathbf{Re}(\widehat{\mathcal{N}_1}|{\widehat{G}})\mathrm{d}\xi\mathrm{d}\tau\nonumber\\
\leq&\,\int_{0}^{t}\|\widehat{\phi}_{i}(\xi)
\widehat{\mathcal{N}_{1}}\|_{L^{2}}\|\widehat{\phi}_{i}(\xi)\widehat{G}\|_{L^{2}}\mathrm{d}\tau\nonumber\\
\leq&\,\delta_1\int_{0}^{t}2^{2k}\|(\rho_{k},d_k,\theta_k,n_{0,k},M_k)\|_{L^{2}}^{2}\mathrm{d}\tau
+C_{\delta_1}\int_{0}^{t}2^{-2k}\|\mathcal{N}_{1}^{k}\|_{L^{2}}^{2}\mathrm{d}\tau,   \\
&\int_{0}^{t}\int_{\mathbb{R}^{3}}\widehat{\phi}_{i}^{2}(\xi)\Big(\mathbf{Re}(\widehat{\mathcal{N}_3}|{\widehat{K}})+\mathbf{Re}(\widehat{\mathcal{N}_4}|{\widehat{K}})\Big)\mathrm{d}\xi\mathrm{d}\tau\nonumber\\
\leq&\,\int_{0}^{t}\|\widehat{\phi}_{i}(\xi)(\widehat{\mathcal{N}_{3}}+\widehat{\mathcal{N}_{4}})\|_{L^{2}}\|\widehat{\phi}_{i}(\xi)\widehat{K}\|_{L^{2}}\mathrm{d}\tau\nonumber\\
\leq&\,\delta_2\int_{0}^{t}2^{2k}\|(\rho_{k},d_k,\theta_k,n_{0,k},M_k)\|_{L^{2}}^{2}\mathrm{d}\tau
+C_{\delta_2}\int_{0}^{t}2^{-2k}\|(\mathcal{N}_{3}+\mathcal{N}_{4})^{k}\|_{L^{2}}^{2}\mathrm{d}\tau,  \\ 
&\int_{0}^{t}\int_{\mathbb{R}^{3}}\widehat{\phi}_{i}^{2}(\xi)\Big(\mathbf{Re}(\widehat{\mathcal{N}_3}|{\widehat{J}})-\mathbf{Re}(\widehat{\mathcal{N}_4}|{\widehat{J}})\Big)\mathrm{d}\xi\mathrm{d}\tau\nonumber\\
\leq&\,\int_{0}^{t}\|\widehat{\phi}_{i}(\xi)(\widehat{\mathcal{N}_{3}},\widehat{\mathcal{N}_{4}})\|_{L^{2}}\|\widehat{\phi}_{i}(\xi)\widehat{J}\|_{L^{2}}\mathrm{d}\tau\nonumber\\
\leq&\,\delta_3\int_{0}^{t}\|(4\theta_k-n_{0,k})\|_{L^{2}}^{2}\mathrm{d}\tau
+C_{\delta_3}\int_{0}^{t}\|(\mathcal{N}_{3}^k,\mathcal{N}_{4}^k)\|_{L^{2}}^{2}\mathrm{d}\tau,   
\end{align*}
where $\delta_1,\delta_2$ and $\delta_3$ are positive small constants.
The remaining terms in \eqref{G3.57}  can be
estimated by the same way. Consequently,
\begin{align}\label{G3.59}
&\|(\rho_{k},d_{k},\theta_{k},n_{0,k},M_k)(t)\|_{L^{2}}^{2}+\int_{0}^{t}2^{2k}\|(\rho_{k},d_{k},\theta_{k},n_{0,k},M_k)(\tau)\|_{L^{2}}^{2}\mathrm{d}\tau\nonumber\\
&\,+\int_{0}^{t}\big\|\big((4\theta_k-n_{0,k}),M_{k}\big)(\tau)\big\|_{L^2}^2\mathrm{d}\tau\nonumber\\
&\,\quad \leq C\|(\rho_{k},d_{k},\theta_{k},n_{0,k},M_k)(0)\|_{L^{2}}^{2}+C\int_{0}^{t}2^{-2k} \big\|\big(\mathcal{N}_{1}^{k},\mathcal{N}_{2}^{k},(\mathcal{N}_{3}+\mathcal{N}_{4})^{k}\big)(\tau)\big\|_{L^{2}}^{2}\mathrm{d}\tau\nonumber\\
&\,\qquad +C\|(\rho_{k},d_{k},\theta_{k},n_{0,k},M_k)(0)\|_{L^{2}}^{2}+C\int_{0}^{t} \|(\mathcal{N}_{3}^{k},\mathcal{N}_{4}^{k})(\tau)\|_{L^{2}}^{2}\mathrm{d}\tau,  
\end{align}
for any integer $k\leq k_0$.

Next, we give the estimate of $\mathcal{P}u$ and $\mathcal{P}n_{1}$.
It follows from \eqref{G3.10} that
\begin{align}\label{G3.60}
&\frac{\mathrm{d}}{\mathrm{d}t}\Big(\widehat{\mathcal{P}u}+\widehat{\mathcal{P}n_1}                 \Big)+|\xi|^2\widehat{\mathcal{P}u}=\widehat{\mathcal{P}\mathcal{N}_2}.
\end{align}
Multiplying \eqref{G3.60} by $\overline{\widehat{\mathcal{P}u}+\widehat{\mathcal{P}n_1}}$ and utilizing
Young's inequality, we get 
\begin{align}\label{G3.61}
&\frac{1}{2}\frac{\mathrm{d}}{\mathrm{d}t}\Big(|\widehat{\mathcal{P}u}+\widehat{\mathcal{P}n_1}                 |^2\Big)+|\xi|^2|\widehat{\mathcal{P}u}+\widehat{\mathcal{P}n_1}|^2\nonumber\\
=&\,|\xi|^2\widehat{\mathcal{P}n_1}\cdot \overline{\widehat{\mathcal{P}u}+\widehat{\mathcal{P}n_1}}+\widehat{\mathcal{P}\mathcal{N}_2}\cdot\overline{\widehat{\mathcal{P}u}+\widehat{\mathcal{P}n_1}}\nonumber\\
\leq&\,\frac{1}{2}|\xi|^2|\widehat{\mathcal{P}u}+\widehat{\mathcal{P}n_1}|^2+|\xi|^2|\widehat{\mathcal{P}n_1}|^2+\frac{1}{2}|\xi|^{-2}|\widehat{\mathcal{P}\mathcal{N}_2}|^2.
\end{align}
From \eqref{G3.10}$_2$, we also have
\begin{align}\label{G3.62}
&\frac{1}{2}\frac{\mathrm{d}}{\mathrm{d}t}|\widehat{\mathcal{P}n_1}|^2+|\widehat{\mathcal{P}n_1}|^2=0.   
\end{align}
Combining \eqref{G3.61} with \eqref{G3.62} yields
\begin{align}\label{GBBB3.65}
&\frac{1}{2}\frac{\mathrm{d}}{\mathrm{d}t}\Big(|\widehat{\mathcal{P}u}+\widehat{\mathcal{P}n_1}                 |^2+2|\widehat{\mathcal{P}n_1}|^2\Big)+\frac{1}{2}|\xi|^2
\Big(|\widehat{\mathcal{P}u}+\widehat{\mathcal{P}n_1}|^2+2|\widehat{\mathcal{P}n_1}|^2\Big) 
\leq \frac{1}{2}|\xi|^{-2}|\widehat{\mathcal{P}\mathcal{N}_2}|^2,   
\end{align}
which implies that 
\begin{align}\label{G3.64}
&\|\big((\mathcal{P}u)_{k},(\mathcal{P}n_1)_{k}\big)(t)\|_{L^{2}}^{2}+\int_{0}^{t}2^{2k}\|\big((\mathcal{P}u)_{k},(\mathcal{P}n_1)_{k}\big)(\tau)\|_{L^{2}}^{2}\mathrm{d}\tau\nonumber\\
&\,\qquad \leq\|\big((\mathcal{P}u)_{k},(\mathcal{P}n_1)_{k}\big)(0)\|_{L^{2}}^{2}+C\int_{0}^{t}2^{-2k}\|(\mathcal{P}\mathcal{N}_{2})_{k}\|_{L^{2}}^{2}\mathrm{d}\tau,    
\end{align}
and
\begin{align}\label{GBB3.64}
&\|(\mathcal{P}n_1)_{k}(t)\|_{L^{2}}^{2}+\int_{0}^{t}\|(\mathcal{P}n_1)_{k})(\tau)\|_{L^{2}}^{2}\mathrm{d}\tau
\leq 0,    
\end{align}
for any integer $k\geq 0$.

\begin{prop}\label{P3.2}
Let $s^{\prime}\geq 0$ be a real number. Then,  for smooth solutions to system \eqref{I-3}--\eqref{I--3},  it holds that 
\begin{align}\label{G3.65}
&\|(\rho,u,\theta,n_{0},n_{1})^{L}(t)\|_{\dot{B}_{2,2}^{{s}^{\prime}}}^{2}+\int_{0}^{t}\Big(\|(\rho,u,\theta,n_{0},n_1)^{L}(\tau)\|_{\dot{B}_{2,2}^{{s}^{\prime}+1}}^{2}
+\big\|\big((4\theta-n_{0}),n_1\big)^{L}(\tau)\big\|_{\dot{B}_{2,2}^{{s}^{\prime}}}^{2}\Big)\mathrm{d}\tau\nonumber\\
&\,\quad \leq C\|(\rho,u,\theta,n_{0},n_1)^{L}(0)\|_{\dot{B}_{2,2}^{s^{\prime}}}^{2}
+C\int_{0}^{t}\big\|\big(\mathcal{N}_1,\mathcal{N}_2,\mathcal{N}_3+\mathcal{N}_4\big)^L(\tau)\|_{\dot{B}_{2,2}^{s^{\prime}-1}}^{2}\mathrm{d}\tau\nonumber\\
&\,\qquad +C\int_{0}^{t}\big\|\big(\mathcal{N}_3,\mathcal{N}_4\big)^L(\tau)\|_{\dot{B}_{2,2}^{s^{\prime}}}^{2}\mathrm{d}\tau,    
\end{align}
where $(\cdot)^L$ represents the low-frequency part of $(\cdot)$.
\end{prop}
\begin{proof}
According to the definition of \eqref{G2.7}, $\mathcal{N}_i(i=1,2,3,4)$, $\mathcal{D}_2^k$ and $(\mathcal{P}\mathcal{N}_2)_k$,   multiplying by $2^{ks^{\prime}}$ and summing up over $k$ with $k\leq k_0$ for \eqref{G3.59}, \eqref{G3.64} and \eqref{GBB3.64}, we directly get \eqref{G3.65}.
\end{proof}

\subsection{Estimates of the medium-frequency part}
In this subsection, we analyze the \linebreak medium-frequency part of 
classical solutions to the problem \eqref{G3.12}--\eqref{G3.15}. To this end,
we shall calculate the asymptotic expansions with regard to
$y_i(\xi)$ and $e^{-t\mathcal{A}(\xi)}$ in high, medium, and low frequencies, respectively.

Now, we go back to the following linearized system of  \eqref{G3.8}:
\begin{equation}\label{I-5}
\left\{
\begin{aligned}
& \partial_t \rho+\Lambda d=0, \\
& \partial_t d-\Lambda\rho-\Lambda\theta-3\Delta d-M=0, \\
& \partial_t \theta+\Lambda d-\Delta \theta+4\theta-n_0=0,\\
& \partial_t n_0+\Lambda M-4\theta+n_0=0,\\
& \partial_t M-\Lambda n_0+M=0,
\end{aligned}\right.
\end{equation}
with the initial data
\begin{align*}
( \rho,d,\theta,n_0,M)(x,t)|_{t=0}=&\,(\rho_{0}(x),d_{0}(x),\theta_{0}(x),n_0^0(x),M_0(x))\nonumber\\
=&\,(\varrho_{0}(x)-1,{\Lambda^{-1}{\rm div}u}_{0}(x),\Theta_{0}(x)-1,I_0^0(x)-1, \Lambda^{-1}{\rm div}I_1^0(x)).
\end{align*}
Applying the Fourier transform on $x$ to the system \eqref{I-5},
we get
\begin{equation}\label{G3.68}
\left\{
\begin{aligned}
& \partial_t \widehat{\rho}+|\xi|\widehat{d}=0, \\
& \partial_t \widehat{d}-|\xi|\widehat{\rho}-|\xi|\widehat{\theta}+3|\xi|^2 \widehat{d}-\widehat{M}=0, \\
& \partial_t \widehat{\theta}+|\xi|\widehat{d}+|\xi|^2\widehat{\theta}+4\widehat{\theta}-\widehat{n_{0}}=0,\\
& \partial_t \widehat{n_{0}}+|\xi|\widehat{M}-4\widehat{\theta}+\widehat{n_{0}}=0,\\
& \partial_t \widehat{M}-|\xi|\widehat{n_{0}} +\widehat{M}=0.
\end{aligned}\right.
\end{equation}

Set
\begin{align}\label{G3.70}
\widehat{V}:=\left(
\begin{matrix}
\widehat{\rho}  \\
\widehat{d}   \\
\widehat{\theta}  \\
\widehat{n_0}  \\
\widehat{M}  \\
\end{matrix}\right), \qquad 
\mathcal{A}(\xi)=\left(
\begin{matrix}
0 & |\xi| & 0 & 0 & 0  \\
-|\xi| & 3|\xi|^2 & -|\xi| & 0 & -1   \\
0 & |\xi| & |\xi|^2+4 & -1 & 0  \\
0 & 0 & -4 & 1 & |\xi|  \\
0 & 0 & 0 & -|\xi| & 1  \\
\end{matrix}\right).   
\end{align}
We can rewrite \eqref{G3.68} into its compact form:
\begin{align}\label{G3.69} 
\frac{\mathrm{d}}{\mathrm{d}t}\widehat{V} =-\mathcal{A}(\xi)\widehat{V}.
\end{align}
Solving the above ODEs gives 
\begin{align}\label{G3.72}
\widehat{V}(t)=e^{t\mathcal{B}(\xi)}\widehat{V}(0),    
\end{align}
where $\mathcal{B}(\xi):=-\mathcal{A}(\xi)$.
Taking the inverse Fourier transform on \eqref{G3.72}, we directly get
the solution of \eqref{G3.69}:
\begin{align*}
V(t)=B(t)V(0),
\end{align*}
where $B(t)V:=\mathcal{F}^{-1}(e^{t\mathcal{B}(\xi)}\widehat{V}(\xi))$.

Obviously, the eigenvalues of $\mathcal{B}(\xi)$ are determined by
\begin{align}\label{G3.73}
{\rm det}(\mathcal{B}(\xi)-y\mathbb{I}_{5})=&-y^5-(4|\xi|^2+6)y^4-(3|\xi|^4+23|\xi|^2+5)y^3\nonumber\\
&-(11|\xi|^4+28|\xi|^2)y^2-(3|\xi|^6+19|\xi|^4+10|\xi|^2)y-(|\xi|^6+5|\xi|^4)\nonumber\\
=&\,0,
\end{align}
where $\mathbb{I}_{5}$ is the $5\times 5$ identity matrix.

Denote $y_i(\xi)(1\leq i\leq 5)$ be the roots of equation \eqref{G3.73}.
Thanks to the semigroup decomposition theory, we rewrite the semigroup
as follows:
\begin{align}\label{G3.74}
e^{t{\mathcal{B}(\xi)}}=\sum_{i=1}^5e^{y_i t}P_i(\xi),   
\end{align}
where the projectors $P_i(1\leq i\leq5)$ satisfy
\begin{align}\label{G3.75}
P_i=\prod_{j\ne i} \frac{\mathcal{B}(\xi)-y_j\mathbb{I}_{5}}{y_i-y_j}.  
\end{align}
By a direct but tedious computation, we obtain the asymptotic expansions of $y_i(1\leq i\leq 5)$ as follows:
\begin{lem}\label{L3.3}
If $|\xi|<r_0$, then the eigenvalues $y_i(1\leq i\leq 5)$ have the following expansion 
\begin{equation*}
\left\{
\begin{aligned}
& y_1=-5-\frac{11}{20}|\xi|^2+\mathcal{O}(|\xi|^4), \\
& y_2=-1-\frac{1}{4}|\xi|^2+\mathcal{O}(|\xi|^4), \\
& y_3=-\frac{1}{2}|\xi|^2+\mathcal{O}(|\xi|^4), \\
& y_4=-\frac{27}{20}|\xi|^2-i\big(\sqrt{2}|\xi|+\mathcal{O}(|\xi|^3)\big), \\
& y_5=-\frac{27}{20}|\xi|^2+i\big(\sqrt{2}|\xi|+\mathcal{O}(|\xi|^3)\big).
\end{aligned}\right.
\end{equation*}
If $r_0\leq |\xi|\leq R_0$, the eigenvalues $y_i(1\leq i\leq 5)$ have the following spectrum gap property 
\begin{align}\label{G3.77}
\mathbf{Re}(y_j)\leq -c_0,   
\end{align}
for some constant $c_0>0$.
If $|\xi|>R_0$, the eigenvalues $y_i(1\leq i\leq 5)$ have the following expansion 
\begin{equation*}
\left\{
\begin{aligned}
& y_1=-\frac{1}{3}+\mathcal{O}(|\xi|^{-2}), \\
& y_2=-|\xi|^2-\frac{9}{2}+\mathcal{O}(|\xi|^{-2}), \\
& y_3=i|\xi|-1+\mathcal{O}(|\xi|^{-2}), \\
& y_4=-i|\xi|-1+\mathcal{O}(|\xi|^{-2}), \\
& y_5=i|\xi|-1+\mathcal{O}(|\xi|^{-2}).
\end{aligned}\right.
\end{equation*}
\end{lem}
\begin{proof}
We only show the details of \eqref{G3.77}. Motivated by the method in Section 3.3 of \cite{DD-JEE-2014}, through \eqref{G3.73},
we set 
$$h(y):=a_0y^5+a_1y^4+a_2y^3+a_3y^2+a_4y+a_5,$$
where 
\begin{gather*}
   a_0=1,\quad a_1=4|\xi|^2+6,\quad a_2=3|\xi|^4+23|\xi|^2+5, \\
a_3=11|\xi|^4+28|\xi|^2,\quad   a_4=3|\xi|^6+19|\xi|^4+10|\xi|^2, \quad  a_5=|\xi|^6+5|\xi|^4.
\end{gather*}
By Routh-Hurwitz theorem (see \cite{LB-1972}, p.459),
the function $h(y)$ has positive real part if and only if the determinants $A_1$,
$A_2$, $A_3$, $A_4$, $A_5$ below are positive. 
Here,
\begin{gather*}
A_1:= a_1,\quad 
A_2:=\begin{vmatrix}
a_1 & a_0  \\
a_3  & a_2\\
\end{vmatrix},  \quad 
A_3:=\begin{vmatrix}
a_1 & a_0 & 0 \\
a_3  & a_2 & a_1\\
a_5  & a_4  &a_3\\
\end{vmatrix}, \\ 
A_4:=\begin{vmatrix}
a_1 & a_0 & 0 &0\\
a_3  & a_2 & a_1&0\\
a_5  & a_4  &a_3&0\\
0  & 0 &a_5  &a_4\\
\end{vmatrix},\quad 
A_5:=\begin{vmatrix}
a_1 & a_0 & 0 &0&0\\
a_3  & a_2 & a_1&0&0\\
a_5  & a_4  &a_3&0 &0\\
0  & 0 &a_5  &a_4&0\\
0&0&0&0&a_5\\
\end{vmatrix}.  
\end{gather*}
By tedious calculation, we directly have
\begin{align*}
A_1&=\,4|\xi|^2+6>0,\\
A_2&=\,12|\xi|^6+99|\xi|^4+130|\xi|^2+30>0,\\
A_3&=\,84|\xi|^{10}+981|\xi|^8+3048|\xi|^6+2836|\xi|^4+480|\xi|^2>0,\\
A_4&=\,252|\xi|^{16}+4539|\xi|^{14}+28623|\xi|^{12}+76230|\xi|^{10}+85804|\xi|^{8}
+37480|\xi|^{6}+4800|\xi|^4>0.
\end{align*}
In addition, the fact that  $\mathrm{sgn}(A_4)$=$\mathrm{sgn}(A_5)$ gives $A_5>0$.
\end{proof}

\begin{rem}
For $|\xi|<r
_0$, we have used the Taylor series expansion; for $|\xi|>R_0$, we have used the Laurent expansion (see \cite{MN-jmku-1980}); finally, for $r_0\leq |\xi|\leq R_0$, we have utilized Hurwitz's theorem. For brevity, we omit the calculation of details here.
\end{rem}
Taking a similar produce to that in \cite{ DD-JEE-2014},   combining \eqref{G3.74}--\eqref{G3.75} and Lemma \ref{L3.3},
 we directly get
\begin{prop}
There exist two positive constants $c_1$ and $C$, such that    
\begin{align}\label{G3.79}
|e^{-t\mathcal{A}(\xi)}|\leq Ce^{-c_1 t},    
\end{align}
for $r_0\leq |\xi|\leq R_0$ and $t\geq 0$. Here, $c_1$ depends only on $r_0$ and $ R_0$. 
\end{prop}
Now, let's return to the nonlinear form   \eqref{G3.68}:
\begin{equation*}
\left\{
\begin{aligned}
& \partial_t \widehat{\rho}+|\xi|\widehat{d}=\widehat{\mathcal{N}_1}, \\
& \partial_t \widehat{d}-|\xi|\widehat{\rho}-|\xi|\widehat{\theta}+3|\xi|^2 \widehat{d}-\widehat{M}=\widehat{\mathcal{D}_2}, \\
& \partial_t \widehat{\theta}+|\xi|\widehat{d}+|\xi|^2\widehat{\theta}+4\widehat{\theta}-\widehat{n_{0}}=\widehat{\mathcal{N}_3},\\
& \partial_t \widehat{n_{0}}+|\xi|\widehat{M}-4\widehat{\theta}+\widehat{n_{0}}=\widehat{\mathcal{N}_4},\\
& \partial_t \widehat{M}-|\xi|\widehat{n_{0}} +\widehat{M}=0,
\end{aligned}\right.
\end{equation*}
which is equivalent to
\begin{align}\label{G3.81} \frac{\mathrm{d}}{\mathrm{d}t}\left(
\begin{matrix}
\widehat{\rho}  \\
\widehat{d}   \\
\widehat{\theta}  \\
\widehat{n_0}  \\
\widehat{M}  \\
\end{matrix}\right)+\mathcal{A}(\xi)\left(
\begin{matrix}
\widehat{\rho}  \\
\widehat{d}   \\
\widehat{\theta}  \\
\widehat{n_0}  \\
\widehat{M}  \\
\end{matrix}\right)=\left(
\begin{matrix}
\widehat{\mathcal{N}_1}  \\
\widehat{\mathcal{D}_2}   \\
\widehat{\mathcal{N}_3}  \\
\widehat{\mathcal{N}_4}  \\
0 \\
\end{matrix}\right),
\end{align}
where $\mathcal{A}(\xi)$ is defined in \eqref{G3.70}.
From \eqref{G3.79},\eqref{G3.81} and the Duhamel principle, we derive that
\begin{align}\label{G3.82}
|(\widehat{\rho},\widehat{d},\widehat{\theta},\widehat{{n}_{0}},\widehat{M})(t,\xi)|
 \leq\, &  C\mathrm{e}^{-c_1t}|(\widehat{\rho},\widehat{d},\widehat{\theta},\widehat{{n}_{0}},\widehat{M})(0,\xi)| \nonumber\\
 & +C\int_{0}^{t}\mathrm{e}^{-c_1(t-\tau)}|(\widehat{\mathcal{N}}_{1},\widehat{\mathcal{D}_2},\widehat{\mathcal{N}}_{3},\widehat{\mathcal{N}}_{4})(\tau,\xi)|\mathrm{d}\tau,  
\end{align}
for all $r_0\leq|\xi|\leq R_0$.
\begin{prop}\label{P3.5}
For any integer $k$ with $k_0\leq k \leq k_1$, there exists a positive constant $C$ depending only on $k_0$ and $k_1$, such that
\begin{align}\label{G3.83}
\int_{0}^{t}\|(\rho_{k},d_{k},\theta_{k},n_{0,k},M_k)(\tau)\|_{L^{2}}^{2}\mathrm{d}\tau\leq&\,C\|(\rho_{k},d_{k},\theta_{k},n_{0,k},M_k)(0)\|_{L^{2}}^{2}\nonumber\\
&\,+C\int_{0}^{t}\big\|\big(\mathcal{N}_{1}^{k},\mathcal{D}_{2}^{k},\mathcal{N}_{3}^{k},\mathcal{N}_{4}^{k}\big)(\tau)\big\|_{L^{2}}^{2}\mathrm{d}\tau.    
\end{align}
\end{prop}
\begin{proof}
 With \eqref{G3.82} in hand, we get   
\begin{align}\label{G3.84}
\|(\rho_{k},d_{k},\theta_{k},n_{0,k},M_k)(t)\|_{L^{2}}\leq&\, C\mathrm{e}^{-c_1t}\|(\rho_{k},d_{k},\theta_{k},n_{0,k},M_k)(0)\|_{L^{2}}\nonumber\\
&\,+C\int_{0}^{t}\mathrm{e}^{-c_1(t-\tau)}\big\|\big(\mathcal{N}_{1}^k,\mathcal{D}_2^k,\mathcal{N}_{3}^k,\mathcal{N}_{4}^k\big)(\tau)\big\|_{L^2}\mathrm{d}\tau,    
\end{align}
for $k_0\leq k\leq k_1$.
Then, integrating the above inequality over $[0,T]$ gives
\begin{align}\label{G3.85}
\int_{0}^{t}\|(\rho_{k},d_{k},\theta_{k},n_{0,k},M_k)(s)\|_{L^{2}}^{2}\mathrm{d}s\leq &\,C\|(\rho_{k},d_{k},\theta_{k},n_{0,k},M_k)(0)\|_{L^{2}}^{2}\nonumber\\
&\,+C\int_{0}^{t}
\Big(\int_{0}^{s}\mathrm{e}^{-c_1(s-\tau)}
\big\|\big(\mathcal{N}_{1}^k,\mathcal{D}_2^k,\mathcal{N}_{3}^k,\mathcal{N}_{4}^k\big)(\tau)
\big\|_{L^{2}}\Big)^{2}\mathrm{d}s.   
\end{align}
Using H\"{o}lder’s inequality and exchanging the order of integration, we
are able to deal with last term in \eqref{G3.85} as follows:
\begin{align}\label{G3.86}
&\int_{0}^{t}\mathrm{d}s\Big(\int_{0}^{s}\mathrm{e}^{-c_1(s-\tau)}\big\|\big(\mathcal{N}_{1}^k,\mathcal{D}_2^k,\mathcal{N}_{3}^k,\mathcal{N}_{4}^k\big)(\tau)\big\|_{L^{2}}\mathrm{d}\tau\Big)^{2}\nonumber\\
\leq&\,C\int_{0}^{t}\mathrm{d}s\Big(\int_{0}^{s}\mathrm{e}^{-c_1(s-\tau)}\mathrm{d}\tau\Big)\Big(\int_{0}^{s}\mathrm{e}^{-c_1(s-\tau)}\big\|\big(\mathcal{N}_{1}^k,\mathcal{D}_2^k,\mathcal{N}_{3}^k,\mathcal{N}_{4}^k\big)(\tau)\big\|_{L^{2}}^{2}\mathrm{d}\tau\Big)\nonumber\\
\leq&\, C\int_{0}^{t}\mathrm{d}s\int_{0}^{s}\mathrm{e}^{-c_1(s-\tau)}\big\|\big(\mathcal{N}_{1}^k,\mathcal{D}_2^k,\mathcal{N}_{3}^k,\mathcal{N}_{4}^k\big)(\tau)\big\|_{L^{2}}^{2}\mathrm{d}\tau\nonumber\\
\leq&\, C\int_{0}^{t}\mathrm{d}\tau\int_{\tau}^{t}\mathrm{e}^{-c_1(s-\tau)}\big\|\big(\mathcal{N}_{1}^k,\mathcal{D}_2^k,\mathcal{N}_{3}^k,\mathcal{N}_{4}^k\big)(\tau)\big\|_{L^{2}}^{2}\mathrm{d}s\nonumber\\
\leq&\, C\int_{0}^{t}\big\|\big(\mathcal{N}_{1}^k,\mathcal{D}_2^k,\mathcal{N}_{3}^k,\mathcal{N}_{4}^k\big)(\tau)\big\|_{L^{2}}^{2}\mathrm{d}\tau.   
\end{align}
Substituting \eqref{G3.86} into \eqref{G3.85}, we complete the  proof.
\end{proof}

To get the estimates on $u_k$ and $n_{1,k}$, we shall deal with 
$(\mathcal{P}u)_k$ and $(\mathcal{P}n_1)_k$. Then, we continue to consider the following 
linear problem  of  \eqref{G3.10}--\eqref{G3.11}:
\begin{equation}\label{G3.87}
\left\{
\begin{aligned}
& \partial_t \mathcal{P}u-\Delta\mathcal{P}u-\mathcal{P}n_1=0, \\
& \partial_t \mathcal{P}n_1+\mathcal{P}n_1=0,
\end{aligned}\right.
\end{equation}
with the initial data
\begin{align*}
\mathcal{P}u(x,t)|_{t=0}=\mathcal{P} u_0(x),\quad 
\text{and}\quad  \mathcal{P}n_1(x,t)|_{t=0}=\mathcal{P} n_{1}^0(x). 
\end{align*}
Applying the Fourier transform on $x$ to the system \eqref{G3.87},
we get
\begin{equation}\label{G3.89}
\left\{
\begin{aligned}
& \partial_t \widehat{\mathcal{P}u}+|\xi|^2\widehat{\mathcal{P}u}-\widehat{\mathcal{P}n_1}=0, \\
& \partial_t \widehat{\mathcal{P}n_1}+\widehat{\mathcal{P}n_1}=0.
\end{aligned}\right.
\end{equation}
Multiplying \eqref{G3.89}$_1$ by $\overline{\widehat{\mathcal{P}u}}$ yields
\begin{align*}
&\frac{1}{2}\frac{\mathrm{d}}{\mathrm{d}t}|\widehat{\mathcal{P}u}|^2               +|\xi|^2|\widehat{\mathcal{P}u}|^2=\widehat{\mathcal{P}n_1}\cdot\widehat{\mathcal{P}u}.
\end{align*}
For any $|\xi|\geq r_0$, by using Young's inequality, we discover that 
\begin{align*}
&\frac{1}{2}\frac{\mathrm{d}}{\mathrm{d}t}|\widehat{\mathcal{P}u}|^2               +|r_0|^2|\widehat{\mathcal{P}u}|^2\leq \frac{1}{2}|r_0|^2|\widehat{\mathcal{P}u}|^2+\frac{1}{2r_0^2}|\widehat{\mathcal{P}n_1}|^2,
\end{align*}
and hence
\begin{align}\label{G3.92}
&\frac{1}{2}\frac{\mathrm{d}}{\mathrm{d}t}|\widehat{\mathcal{P}u}|^2               +\frac{1}{2}|r_0|^2|\widehat{\mathcal{P}u}|^2\leq \frac{1}{2r_0^2}|\widehat{\mathcal{P}n_1}|^2.
\end{align}
The summation of \eqref{G3.92} and $\frac{1}{r_0^2}\times \eqref{G3.62}$ gives
\begin{align*}
&\frac{1}{2}\frac{\mathrm{d}}{\mathrm{d}t}\Big(|\widehat{\mathcal{P}u}|^2+\frac{1}{r_0^2}|\widehat{\mathcal{P}n_1}|^2\Big)               +\frac{1}{2}|r_0|^2|\widehat{\mathcal{P}u}|^2+ \frac{1}{2r_0^2}|\widehat{\mathcal{P}n_1}|^2\leq 0.
\end{align*}
Therefore, there exists a constant $\lambda_3>0$ depending only on $r_0$,
such that 
\begin{align*}
&\frac{\mathrm{d}}{\mathrm{d}t}|(\widehat{\mathcal{P}u},\widehat{\mathcal{P}n_1})(t)|^2               +\lambda_3 |(\widehat{\mathcal{P}u},\widehat{\mathcal{P}n_1})(t)|^2   \leq 0,
\end{align*}
for any $|\xi|\geq r_0$.
After a straightforward computation, we arrive at 
\begin{align}\label{G3.95}
&|(\widehat{\mathcal{P}u},\widehat{\mathcal{P}n_1})(t)|^2\leq C e^{-\lambda_3 t} |(\widehat{\mathcal{P}u},\widehat{\mathcal{P}n_1})(0)|^2  .
\end{align}
Similar to the estimate \eqref{G3.83}, by utilizing \eqref{G3.95}, we have
\begin{align}\label{G3.96}
\int_{0}^{t}\big\|\big((\mathcal{P}u)_k,(\mathcal{P}n_1)_k\big)(\tau)\big\|_{L^{2}}^{2}\mathrm{d}\tau\leq
&\,C\big\|\big((\mathcal{P}u)_k,(\mathcal{P}n_1)_k\big)(0)\big\|_{L^{2}}^{2}
+C\int_{0}^{t}\|(\mathcal{P}\mathcal{N}_2)_k(\tau)\|_{L^{2}}^{2}\mathrm{d}\tau.    
\end{align}
Combining \eqref{G3.83} with \eqref{G3.96}, then multiplying by $2^k\tilde{s}$ and
summing up over $k$ with $k_0\leq k \leq k_1$, we arrive at the following proposition.
\begin{prop}\label{P3.6}
Let $\tilde{s}\geq 0$ be a real number, for smooth solutions to system \eqref{I-3}--\eqref{I--3},  it holds that 
\begin{align*}
\|(\rho,u,\theta,n_{0},n_{1})^{M}(t)\|_{\dot{B}_{2,2}^{\tilde{s}}}^{2}
\leq\,& C\|(\rho,u,\theta,n_{0},n_1)^{M}(0)\|_{\dot{B}_{2,2}^{\tilde{s}}}^{2}\nonumber\\
&+C\int_{0}^{t}\big\|\big(\mathcal{N}_1,\mathcal{N}_2,\mathcal{N}_3,\mathcal{N}_4\big)^M(\tau)
\|_{\dot{B}_{2,2}^{\tilde{s}}}^{2}\mathrm{d}\tau,    
\end{align*}
where $(\cdot)^M$ represents the medium-frequency part of $(\cdot)$.
\end{prop}
Motivated by the Lemmas 3.10--3.11 in \cite{Wwj-siam-2021}, we develop a new method to handle the estimate on $\|\nabla^m(\rho,u,\theta,n_0,n_1)\|_{L^2}$ for any integer 
$m\geq 0$. What's more, it is a crucial step in the proof of the global existence
of solutions. Denote $\mathbb{U}(t):=(\rho(t),u(t),\theta(t),n_0(t),n_1(t))$, we have the following lemma.
\begin{lem}\label{L3.7}
For any integer $m\geq 0$, it holds that
\begin{align}\label{GB3.98}
\int_0^t \|\nabla^m \mathbb{U}^M(\tau)\|_{L^2}^2\mathrm{d}\tau\leq C\sigma \int_{0}^{t}\Big(\|\nabla(\rho)(\tau)\|_{H^{3}}^{2}+\|\nabla(u,\theta)(\tau)\|_{H^{4}}^{2}\Big)\mathrm{d}\tau+C\|\mathbb{U} (0)\|_{L^2}^2.
\end{align}
\end{lem}
\begin{proof}
For $k_0\leq k \leq k_1$, multiplying \eqref{G3.83} and \eqref{G3.96} by
$2^{mk}$ and then summing up over $k$, we have
\begin{align*}
\|\nabla^m \mathbb{U}^M(t)\|_{L^2}^2\leq&\,Ce^{-c_2t}\| \mathbb{U}(0)\|_{L^2}^2+C\int_0^t e^{-c_2(t-\tau)}\big\|\big(\mathcal{N}_{1},\mathcal{N}_2,\mathcal{N}_{3},\mathcal{N}_{4}\big)(\tau)\big\|_{L^{2}}^{2}\mathrm{d}\tau,
\end{align*}
where $c_2=\min\{c_1,\lambda_3\}>0$.
Here, we only used the property for the finite sum of $k$ and the fact $2^k\leq 2^{k_1}$. Compared to the same conclusion in Lemma 3.10 in \cite{Wwj-siam-2021} which contains Fourier analysis method, our method which contains the homogeneous Littlewood-Paley decomposition is more direct.

Using the same procedure in Proposition \ref{P3.5} again, we have
\begin{align}\label{GB3.100}
\int_0^t\|\nabla^m \mathbb{U}^M(\tau)\|_{L^2}^2\mathrm{d}\tau\leq&\,C\| \mathbb{U}(0)\|_{L^2}^2+C\int_0^t \big\|\big(\mathcal{N}_{1},\mathcal{N}_2,\mathcal{N}_{3},\mathcal{N}_{4}\big)(\tau)\big\|_{L^{2}}^{2}\mathrm{d}\tau.
\end{align}
It follows from Lemmas \ref{L2.1}--\ref{L2.2} and the assumption \eqref{G3.1}
that
\begin{align}\label{GB3.101}
\big\|\big(\mathcal{N}_{1},\mathcal{N}_2,\mathcal{N}_{3},\mathcal{N}_{4}\big)\big\|_{L^{2}}^{2}\leq&\, C\Big(\|(\rho,u,\theta)\|_{L^{\infty}}^2\|\nabla(\rho,u,\theta)\|_{H^{1}}^2+  \|(n_0,n_1)\|_{L^2}^2\|\rho\|_{L^{\infty}}^2                \Big) \nonumber\\
&\,+ C\Big(\|\nabla u\|_{L^{\infty}}^2\|\nabla u\|_{H^{1}}^2+  \|(\rho,\theta)\|_{L^2}^2\|\theta\|_{L^{\infty}}^2                \Big)\nonumber\\
\leq &\, C\sigma^2\Big(\|\nabla\rho\|_{H^{3}}^{2}+\|\nabla(u,\theta)\|_{H^{4}}^{2}\Big).
\end{align}
Plugging \eqref{GB3.101} into \eqref{GB3.100}, we obtain \eqref{GB3.98}
and hence complete the proof of Lemma \eqref{GB3.98}.
\end{proof}

\subsection{Global existence of classical solutions to the problem \eqref{I-3}--\eqref{I--3}}
with Propositions \ref{P3.1}--\ref{P3.2} and Proposition \ref{P3.6} in hand,
we shall establish the global existence and uniqueness of smooth solutions
stated in Theorem \ref{T2.1}.
Under the a priori assumption \eqref{G3.1}, by using Sobolev's imbedding inequality,
we  infer that
\begin{align*}
\frac{1}{2}\leq \rho+1  \leq \frac{3}{2}. 
\end{align*}
Moreover, it is obviously that 
\begin{align*}
|g(\rho)|\leq C|\rho|,\quad |h(\rho)|\leq C, 
\end{align*}
and
\begin{align*}
|g^{(k)}(\rho)|,\quad |h^{(k)}(\rho)|\leq C,\quad \mathrm{for}\  k\geq1,    
\end{align*}
where $g$ and $h$ are defined in  \eqref{I--4}.
To prove Theorem \ref{T2.1}, for $t\in [0,T]$, we define
\begin{align*}
\mathbf{N}(t):=&\,\sup_{0\leq\tau\leq t}\|(\rho,u,\theta,n_{0},n_1)(\tau)\|_{H^{4}}^{2}+\int_{0}^{t}\Big(\|\nabla(\rho,n_0)(\tau)\|_{H^{3}}^{2}\nonumber\\
&\qquad \,+\|\nabla(u,\theta)(\tau)\|_{H^{4}}^{2}+\|n_1(\tau)\|_{H^{4}}^{2}+\|(4\theta-n_0)(\tau)\|_{H^4}^2\Big)\mathrm{d}\tau.   
\end{align*}
\begin{proof}[Proof of Theorem \ref{T2.1}]
Taking $s=s^{\prime}=4$ and $ s=1, s^{\prime}=0$ in Propositions \ref{P3.1}
and \ref{P3.2} respectively, and choosing $\tilde{s}=4$ and $\tilde{s}=0$ in Proposition \ref{P3.6}, we
have
\begin{align}\label{GB3.106}
\mathbf{N}(t)\leq&\, C \mathbf{N}(0)+C\int_{0}^{t}\Big(\|\big(\mathcal{N}_{1_2},\mathcal{N}_{4}\big)(\tau)\|_{H^{4}}^{2}+\big\|\big(\mathcal{N}_1,\mathcal{N}_2,\mathcal{N}_3,\mathcal{N}_4,(\nabla u)^{T}\cdot\nabla\rho\big)(\tau)\big\|_{H^{3}}^{2}\Big)\mathrm{d}\tau \nonumber\\
&\,+C\int_{0}^{t}\Big(\|\nabla^{2}\rho\|_{H^{2}}^{2}\|u^{S}\|_{H^{3}}^{2}+\|\nabla u(\tau)\|_{H^{2}}\|\nabla\rho^{S}(\tau)\|_{H^{3}}^{2}\Big)\mathrm{d}\tau\nonumber\\
&\,+C\int_{0}^{t}\|\nabla u\|_{H^{2}}^{2}\sum_{k>k_{1}}(1+2^{6k})\bigg(\sum_{l\geq k-1}2^{k-l}\|\nabla\rho_{l}\|_{L^{2}}\bigg)^{2}\mathrm{d}\tau\nonumber\\
&\,+C\int_{0}^{t}\big\|\big(\mathcal{N}_{1},\mathcal{N}_{2},\mathcal{N}_{3}+\mathcal{N}_{4}\big)(\tau)\big\|_{L^{\frac{6}{5}}}^{2}+C\int_{0}^{t}\big\|\big(\mathcal{N}_{3},\mathcal{N}_{4}\big)(\tau)\big\|_{L^{2}}^{2}
\mathrm{d}\tau\nonumber\\
&\,+C\int_{0}^{t}\Big(\|\nabla(\rho,n_0)^M(\tau)\|_{H^{3}}^{2}+\|\nabla(u,\theta)^M(\tau)\|_{H^{4}}^{2}+\|n_1^M(\tau)\|_{H^{4}}^{2}\Big)\mathrm{d}\tau\nonumber\\
&\,+C\int_{0}^{t}\|(4\theta-n_0)^M(\tau)\|_{H^{4}}^{2}\mathrm{d}\tau,
\end{align}
where we have used Lemmas \ref{L2.5}--\ref{L2.6}, Lemma \ref{L2.8} (with $l=1$, $p=2$, $q=\frac{6}{5}$) and the following facts:
\begin{align*}
\|f\|_{H^s}\backsim \|f\|_{L^2}+\|f\|_{\dot{H}^s}
\end{align*}
for $s\geq 0$, and
\begin{align*}
\big\|\big(\mathcal{N}_1,\mathcal{N}_2,\mathcal{N}_3+\mathcal{N}_4\big)^{L}\big\|_{\dot{B}_{2,2}^{-1}}\leq&\,C\big\|\Lambda^{-1}\big(\mathcal{N}_1,\mathcal{N}_2,\mathcal{N}_3+\mathcal{N}_4\big)^{L}\big\|_{\dot{B}_{2,2}^{0}}\nonumber\\
\leq&\,C\|\Lambda^{-1}\big(\mathcal{N}_1,\mathcal{N}_2,\mathcal{N}_3+\mathcal{N}_4\big)\|_{L^{2}}\nonumber\\
\leq&\, C\|\big(\mathcal{N}_1,\mathcal{N}_2,\mathcal{N}_3+\mathcal{N}_4\big)\|_{L^{\frac{6}{5}}}.    
\end{align*}
Next, we deal with the terms on the right-hand side of \eqref{GB3.106} one by one.
In view of Lemmas \ref{L2.1}--\ref{L2.2}, we obtain
\begin{align}
\|\mathcal{N}_{1_2}\|_{H^{4}}\leq&\,C\|\rho\|_{L^{\infty}}\|\mathrm{div}u\|_{H^{4}}+C\|\rho\|_{H^{4}}\|\mathrm{div}u\|_{L^{\infty}}\nonumber\\
\leq&\,C\|\rho\|_{H^{2}}\|\nabla u\|_{H^{4}}+C\|\rho\|_{H^{4}}\|\nabla\mathrm{div}u\|_{H^{1}}\nonumber\\
\leq&\,C\|\rho\|_{H^{4}}\|\nabla u\|_{H^{4}},\nonumber\\\label{G3.110}
\|\mathcal{N}_{4}\|_{H^{4}}
\leq&\,C\|\theta\|_{L^{\infty}}\|\theta\|_{H^{4}}
\leq C\|\theta\|_{H^{4}}\|\nabla \theta\|_{H^{4}}.    
\end{align}
Similarly,
\begin{align*}
\|\mathcal{N}_{1}\|_{H^{3}}\leq&\,\|\mathcal{N}_{1_1}\|_{H^{3}}+\|\mathcal{N}_{1_2}\|_{H^{3}}\nonumber\\
\leq&\,C\|u\|_{L^{\infty}}\|\nabla \rho\|_{H^{3}}+C\|u\|_{H^{3}}\|\nabla\rho\|_{L^{\infty}}+\|\mathcal{N}_{12}\|_{H^{3}}\nonumber\\
\leq&\,C\|u\|_{H^{3}}\|\nabla \rho\|_{H^{3}}+C\|\rho\|_{H^{3}}\|\nabla u\|_{H^{3}},\\
\|\mathcal{N}_{2}\|_{H^{3}}\leq&\,
C\|u\|_{L^{\infty}}\|\nabla u\|_{H^{3}}+C\|u\|_{H^{3}}\|\nabla u\|_{L^{\infty}}+C\|g(\rho)\|_{L^{\infty}}\|\nabla\rho\|_{H^{3}}\nonumber\\
&\,+C\|g(\rho)\|_{H^{3}}\|\nabla\rho\|_{L^{\infty}}+C\|h(\rho)\|_{L^{\infty}}
\big(\|\theta\|_{L^{\infty}}\|\nabla\rho\|_{H^{3}}+\|\theta\|_{H^{3}}\|\nabla\rho\|_{L^{\infty}}\big)\nonumber\\
&\,+C\|h(\rho)\|_{H^{3}}\|\theta\|_{L^{\infty}}\|\nabla\rho\|_{L^{\infty}}+C\|g(\rho)\|_{L^{\infty}}\|\nabla^{2}u\|_{H^{3}}\nonumber\\
&\,+C\|g(\rho)\|_{H^{3}}\|\nabla^{2}u\|_{L^{\infty}}+C\|g(\rho)\|_{L^{\infty}}\|n_1\|_{H^{3}}+C\|g(\rho)\|_{H^{3}}\|n_1\|_{L^{\infty}}\nonumber\\
\leq&\,C\|u\|_{H^{3}}\|\nabla u\|_{H^{3}}+C\|\theta\|_{H^{3}}\|\nabla\rho\|_{H^{3}}\nonumber\\
&+C\|\rho\|_{H^{3}}\big(\|\nabla\rho\|_{H^{3}}+\|\nabla u\|_{H^{4}}+\|\nabla\theta\|_{H^{3}}+\|n_1\|_{H^{3}}\big),\\  
\|\mathcal{N}_{3}\|_{H^{3}}\leq&\,
C\|\theta\|_{L^{\infty}}\|\nabla u\|_{H^{3}}+C\|\theta\|_{H^{3}}\|\nabla u\|_{L^{\infty}}+C\|g(\rho)\|_{L^{\infty}}\|\Delta\theta\|_{H^{3}}\nonumber\\
&\,+C\|g(\rho)\|_{H^{3}}\|\Delta\theta\|_{L^{\infty}}+C\|g(\rho)\|_{H^{3}}\|(n_0,\theta)\|_{L^{\infty}}+C\|g(\rho)\|_{L^{\infty}}\|(n_0,\theta)\|_{H^{3}}\nonumber\\
&\,+C\|u\|_{L^{\infty}}\|\nabla \theta\|_{H^{3}}+C\|u\|_{H^{3}}\|\nabla \theta\|_{L^{\infty}}+C\|h(\rho)\|_{H^3}\|\nabla u\|_{L^{\infty}}^2\nonumber\\
&\,+C\|h(\rho)\|_{L^{\infty}}\big(\|\nabla u\|_{L^{\infty}}\|\nabla u\|_{H^{3}}+\|\theta\|_{H^{3}}\big(\|\theta\|_{L^{\infty}}+\|\theta\|_{L^{\infty}}^2+\|\theta\|_{L^{\infty}}^3\big)\big)\nonumber\\
&\,+C\|h(\rho)\|_{H^{3}}\big(\|\theta\|_{L^{\infty}}^2+\|\theta\|_{L^{\infty}}^3+\|\theta\|_{L^{\infty}}^4\big)\nonumber\\
\leq&\,C\|(\rho,u,\theta)\|_{H^{4}}\|\nabla (u,\theta)\|_{H^{4}}+C\|(n_0,\theta)\|_{H^{3}}\|\nabla\rho\|_{H^{4}}\nonumber\\
&+C\|\rho\|_{H^{4}}\big(\|\nabla u\|_{H^{4}}^2+\|n_0\|_{H^{3}}\big).  
\end{align*}
For the term $\mathcal{N}_4$, \eqref{G3.110} gives
\begin{align*}
\|\mathcal{N}_4\|_{H^3}\leq \|\mathcal{N}_4\|_{H^4} \leq C\|\theta\|_{H^{4}}\|\nabla \theta\|_{H^{4}}.  
\end{align*}
From  H\"{o}lder’s inequality and Lemmas \ref{L2.1}--\ref{L2.2}, it holds that
\begin{align*}
\|(\nabla u)^{\top}\cdot\nabla\rho\|_{H^{3}}\leq&\, C\|\nabla\rho\|_{L^{\infty}}\|\nabla u\|_{H^{3}}+C\|\nabla\rho\|_{H^{3}}\|\nabla u\|_{L^{\infty}}\nonumber\\
\leq&\, C\|\rho\|_{H^{4}}\|\nabla u\|_{H^{3}},    
\end{align*}
and
\begin{align*}
\|\mathcal{N}_{1}\|_{L^{\frac{6}{5}}}\leq&\,\|\mathcal{N}_{1_1}\|_{L^{\frac{6}{5}}}+\|\mathcal{N}_{1_2}\|_{L^{\frac{6}{5}}}\nonumber\\
\leq&\,C\|u\|_{L^{2}}\|\nabla\rho\|_{L^{3}}+C\|\rho\|_{L^{2}}\|\nabla u\|_{L^{3}}\nonumber\\
\leq&\,C\|u\|_{L^{2}}\|\nabla\rho\|_{H^{1}}+C\|\rho\|_{L^{2}}\|\nabla u\|_{H^{1}}. 
\end{align*}
Similarly,
\begin{align*}
\big\|\big(\mathcal{N}_{2},\mathcal{N}_{3}+\mathcal{N}_{4}\big)\big\|_{L^{\frac{6}{5}}}
\leq&\,C\|(\rho,u,\theta)\|_{H^1}\|\nabla(\rho,u,\theta)\|_{H^2}+C\|\rho\|_{H^1}\big(\|n_1\|_{L^2}+\|4\theta-n_0\|_{L^2}   \big)\nonumber\\
&\,+\|\rho\|_{H^1}\|\theta\|_{L^2}\big(\|\theta\|_{L^{\infty}}+\|\theta\|_{L^{\infty}}^2\big)+C\|\rho\|_{H^1}\|\nabla u\|_{H^2}^2,
\end{align*}
and
\begin{align*}
\|(\mathcal{N}_3,\mathcal{N}_4)\|_{L^2}\leq \|(\mathcal{N}_3,\mathcal{N}_4)\|_{H^3}\leq&\,C\|(\rho,u,\theta)\|_{H^{4}}\|\nabla (u,\theta)\|_{H^{4}}+C\|(n_0,\theta)\|_{H^{3}}\|\nabla\rho\|_{H^{4}}\nonumber\\
&+C\|\rho\|_{H^{4}}\big(\|\nabla u\|_{H^{4}}^2+\|n_0\|_{H^{3}}\big).   
\end{align*}
By using Lemma \ref{L2.5}, we conclude that
\begin{align*}
\int_{0}^{t}\|\nabla^{2}\rho\|_{H^{2}}^{2}\|u^{S}\|_{H^{3}}^{2}\mathrm{d}\tau\leq C\sup\limits_{0\leq\tau\leq t}\|u(\tau)\|_{H^{3}}^2\int_{0}^{t}\|\nabla^{2}\rho(\tau)\|_{H^{2}}^{2}\mathrm{d}\tau,
\end{align*}
and
\begin{align*}
\int_{0}^{t}\|\nabla u(\tau)\|_{H^{2}}\|\nabla\rho^{S}(\tau)\|_{H^{3}}^{2}\mathrm{d}\tau\leq\sup\limits_{0\leq\tau\leq t}\|\nabla u(\tau)\|_{H^{2}}\int_{0}^{t}\|\nabla\rho(\tau)\|_{H^{3}}^{2}\mathrm{d}\tau.   
\end{align*}
From Lemma \ref{L2.1}, Lemma \ref{L2.7} and Young's inequality 
for series convolution, we obtain
\begin{align*}
&\int_{0}^{t}\|\nabla u\|_{H^{2}}^{2}\sum_{k>k_{1}}(1+2^{6k})\bigg(\sum_{l\geq k-1}2^{k-l}\|\nabla\rho_{l}\|_{L^{2}}\bigg)^{2}\mathrm{d}\tau\nonumber\\
\leq&\,C\sup_{0\leq\tau\leq t}\|\nabla u(\tau)\|_{H^{2}}\int_{0}^{t}\sum_{k>k_{1}}(1+2^{6k})\bigg(\sum_{l\geq k-1}2^{k-l}\|\nabla\rho_{l}\|_{L^{2}}\bigg)^{2}\mathrm{d}\tau\nonumber\\
\leq&\,C\sup_{0\leq\tau\leq t}\|\nabla u(\tau)\|_{H^{2}}^{2}\int_{0}^{t}\sum_{k>k_{1}}\bigg(\sum_{l\geq k-1}2^{k-l}\|\nabla\rho_{l}\|_{L^{2}}\bigg)^{2}\mathrm{d}\tau\nonumber\\
&\,+C\sup_{0\leq\tau\leq t}\|\nabla u(\tau)\|_{H^{2}}^{2}\int_{0}^{t}\sum_{k>k_{1}}\bigg(\sum_{l\geq k-1}2^{4(k-l)}2^{3l}\|\nabla\rho_{l}\|_{L^{2}}\bigg)^{2}\mathrm{d}\tau\nonumber\\
\leq &\,C\sup_{0\leq\tau\leq t}\|\nabla u(\tau)\|_{H^{2}}^{2}\bigg(\int_{0}^{t}\sum_{k\in\mathbb{Z}}\|\nabla\rho_{k}\|_{L^{2}}^{2}\mathrm{d}\tau
+\int_{0}^{t}\sum_{k\in\mathbb{Z}}2^{6k}\|\nabla\rho_{k}\|_{L^{2}}^{2}\mathrm{d}\tau\bigg)\nonumber\\
\leq&\, C\sup_{0\leq\tau\leq t}\|\nabla u(\tau)\|_{H^{2}}^{2}\int_{0}^{t}\|\nabla\rho\|_{H^{3}}^{2}\mathrm{d}\tau.
\end{align*}
For the remaining terms, Lemma \ref{L3.7} provides
\begin{align*}
&\int_{0}^{t}\Big(\|\nabla(\rho,n_0)^M(\tau)\|_{H^{3}}^{2}+\|\nabla(u,\theta)^M(\tau)\|_{H^{4}}^{2}+\|n_1^M(\tau)\|_{H^{4}}^{2}+\|(4\theta-n_0)^M(\tau)\|_{H^{4}}^{2}\Big)\mathrm{d}\tau\nonumber\\
&\,\qquad \leq C\mathbf{N}(0)+C\sigma \mathbf{N}(t).
\end{align*}
Hence, by putting the above estimates into \eqref{GB3.106}, it yields
\begin{align*}
\mathbf{N}(t)\leq C\big(\mathbf{N}(0)+\sigma\mathbf{N}(t)\big),  
\end{align*}
which implies that  $\mathbf{N}(t)\leq C\mathbf{N}(0)$.
 
\end{proof}
Similar to the Section 2.1 in \cite{DD-AMPA-2017}, we can
use the fixed point theorem and the standard iteration arguments
to obtain the local existence and uniqueness of
classical solutions to the system \eqref{I-3}--\eqref{I--3}. For brevity, we omit the details here.
Then, through the standard continuity argument,
we consequently get the global existence and uniqueness of smooth solutions. This completes the proof of Theorem \ref{T2.1}.  

\section{Optimal time-decay rates of classical solutions to the problem \eqref{I-3}--\eqref{I--3}}
In this
section, we devote ourselves to proving the optimal
time-decay rates of classical solutions to the
problem \eqref{I-3}--\eqref{I--3}. Under the assumptions of Theorem \ref{T2.1},
we divide the proof into high and low-medium frequency parts.

\subsection{High frequency part}
Based on the analysis on the  high-frequency part in Proposition \ref{P3.1}, 
we shall show  following proposition.
\begin{prop}\label{P4.1}
It holds that
\begin{align}\label{G4.1}
&\|(\rho,u,\theta,n_{0},n_1)^{S}(t)\|_{B_{2,2}^{3}}^{2}+\|(\rho,u,\theta,n_{0},n_1)^{S}(t)\|_{B_{2,2}^{4}}^{2}\nonumber\\
&\quad  \leq C\mathrm{e}^{-c_{3}t}\Big(\|(\rho,u,\theta,n_{0},n_1)^{S}(0)\|_{B_{2,2}^{3}}^{2}+\|(\rho,u,\theta,n_{0},n_1)^{S}(0)\|_{B_{2,2}^{4}}^{2}\Big)\nonumber\\
&\qquad +C\sigma\int_{0}^{t}\mathrm{e}^{-c_{3}(t-\tau)}\|(\rho,u,\theta,n_{0},n_1)^{L+M}(\tau)\|_{B_{2,2}^{3}}^{2}\mathrm{d}\tau\nonumber\\
&\qquad +C\sigma\int_{0}^{t}\mathrm{e}^{-c_{3}(t-\tau)}\Big(\|(\rho,u,\theta,n_{0},n_1)^{L+M}(\tau)\|_{B_{2,2}^{4}}^{2}+\|(u,\theta)^{L+M}(\tau)\|_{B_{2,2}^{5}}^{2}\Big)\mathrm{d}\tau,    
\end{align}    
where the constant $c_3>0$ is independent of $\sigma$.
\end{prop}
\begin{proof}
Multiplying \eqref{G3.40} by $2^6k$, we have
\begin{align}\label{G4.2}
&\frac{\mathrm{d}}{\mathrm{d}t}2^{6k}\mathcal{H}_{k}(t)+\lambda_{1}2^{6k}\|(\theta_{k},n_{0,k},n_{1,k})(t)\|_{L^{2}}^{2}+\lambda_{1}2^{2k_{1}-1}2^{6k}\|(\rho_{k},u_{k})(t)\|_{L^{2}}^{2}\nonumber\\
& +\lambda_1 2^{8k-1}\|(\rho_{k},u_{k})(t)\|_{L^{2}}^{2}+\lambda_{1}2^{8k}\|(\theta_{k},n_{0,k},n_{1,k})(t)\|_{L^{2}}^{2}+\lambda_{1}2^{10k}\|(u_{k},\theta_{k})(t)\|_{L^{2}}^{2}\nonumber\\
&\quad \leq C 2^{6k}\|(\mathcal{N}_1^k,\mathcal{N}_2^k,\mathcal{N}_3^k,\mathcal{N}_4^k,\nabla\mathcal{N}_{12}^k,\nabla\mathcal{N}_4^k)(t)\|_{L^2}^2+C2^{6k}\|((\nabla u)^{T}\cdot\nabla\rho)_{k}(t)\|_{L^{2}}^{2}\nonumber\\
&\qquad+C2^{6k}\|\nabla u\|_{H^{2}}\|\nabla\rho_{k}\|_{L^{2}}^{2}+C2^{6k}\|\nabla^{2}\rho\|_{H^{2}}^{2}\|u_{k}\|_{L^{2}}^{2}\nonumber\\
&\qquad+C2^{6k}\|\nabla u\|_{L^{\infty}}^{2}\bigg(\sum_{l\geq k-1}2^{k-l}\|\nabla\rho_{l}\|_{L^{2}}\bigg)^{2},    
\end{align}
for any integer $k>k_1>0$.    
For \eqref{G4.2}, thanks to the $l^2$ summation over $k$ for from $k=k_1+1$ to $\infty$, it yields
\begin{align}\label{G4.3}
&\frac{\mathrm{d}}{\mathrm{d}t}\sum_{k>k_{1}}2^{6k}\mathcal{H}_{k}(t)
+\|(\rho,u,\theta,n_{0},n_1)^{S}(t)\|_{\dot{B}_{2,2}^{3}}^{2}
+\|(\rho,u,\theta,n_{0},n_1)^{S}(t)\|_{\dot{B}_{2,2}^{4}}^{2}
+\|(u,\theta)^{S}(t)\|_{\dot{B}_{2,2}^{5}}^{2}\nonumber\\
&\quad  \leq C\|(\mathcal{N}_{1},\mathcal{N}_{2},\mathcal{N}_{3},\mathcal{N}_{4},\nabla \mathcal{N}_{1_2},\nabla\mathcal{N}_4)^{S}(t)\|_{\dot{B}_{2,2}^{3}}^{2}+C\|((\nabla u)^{T}\cdot\nabla\rho)^{S}(t)\|_{\dot{B}_{2,2}^{3}}^{2}\nonumber\\
&\qquad +C\|\nabla u\|_{H^{2}}\|\nabla\rho^{S}\|_{\dot{B}_{2,2}^{3}}^{2}
+C\|\nabla^{2}\rho\|_{H^{2}}^{2}\|u^{S}\|_{\dot{B}_{2,2}^{3}}^{2}\nonumber\\
&\qquad +\|\nabla u\|_{H^{2}}^{2}\sum_{k>k_{1}}2^{6k}\Big(\sum_{l\geq k-1}2^{k-l}\|\nabla\rho_{l}\|_{L^{2}}\Big)^{2}.    
\end{align}
Then, exploiting Lemmas \ref{L2.1}--\ref{L2.2} and Lemma \ref{L2.5}, we get
\begin{align*}
\|(\nabla \mathcal{N}_{1_2})^{S}(t)\|_{\dot{B}_{2,2}^{3}}\leq&\,C\|\nabla^{4}(\rho\mathrm{div}u)\|_{L^{2}}\nonumber\\
\leq&\,C\|\rho\|_{L^{\infty}}\|\nabla^{4}\mathrm{div}u\|_{L^{2}}
+\|\nabla^{4}\rho\|_{L^{2}}\|\mathrm{div}u\|_{L^{\infty}}\nonumber\\
\leq&\,C\sigma\Big(\|\rho\|_{\dot{B}_{2,2}^{4}}+\|u\|_{\dot{B}_{2,2}^{5}}\Big),\\
\|(\mathcal{N}_{1})^{S}(t)\|_{\dot{B}_{2,2}^{3}}\leq&\,C\|\nabla^{3}(u\cdot\nabla\rho)\|_{L^{2}}
+C\|\nabla^{3}\mathcal{N}_{1_2}\|_{L^{2}}\nonumber\\
\leq&\,C\|u\|_{L^{\infty}}\|\nabla^{4}\rho\|_{L^{2}}+C\|\nabla^{3}u\|_{L^{2}}\|\nabla\rho\|_{L^{\infty}}
+C\|\nabla^{3}\mathcal{N}_{1_2}\|_{L^{2}}\nonumber\\
\leq&\,C\sigma\Big(\|(\rho,u)\|_{\dot{B}_{2,2}^{3}}+\|(\rho,u)\|_{\dot{B}_{2,2}^{4}}\Big),\\
\|\nabla(\mathcal{N}_{4})^{S}(t)\|_{\dot{B}_{2,2}^{3}}
\leq&\,C\|\nabla^{4}(\theta^4+\theta^3+\theta^2)\|_{L^{2}}\nonumber\\
\leq&\,C\big(1+\|\theta\|_{L^{\infty}}+\|\theta\|_{L^{\infty}}^2     \big)\|\theta\|_{L^2}\|\nabla^4\theta\|_{L^2}\nonumber\\
\leq&\,C\sigma\|\theta\|_{\dot{B}_{2,2}^{5}}.
\end{align*}
Moreover,
\begin{align}\label{G4.7}
\|(\mathcal{N}_{2})^{S}(t)\|_{\dot{B}_{2,2}^{3}}\leq&\,C\|u\|_{L^{\infty}}\|\nabla^{4}u\|_{L^{2}}
+C\|\nabla^{3}u\|_{L^{2}}\|\nabla u\|_{L^{\infty}}\nonumber\\
&\,+C\|g(\rho)\|_{L^{\infty}}\|\nabla^{4}\rho\|_{L^{2}}
+C\|\nabla^{3}g(\rho)\|_{L^{2}}\|\nabla\rho\|_{L^{\infty}}\nonumber\\
&\,+C\|h(\rho)\|_{L^{\infty}}\big(\|\theta\|_{L^{\infty}}\|\nabla^{4}\rho\|_{L^{2}}
+\|\nabla^{3}\theta\|_{L^{2}}\|\nabla\rho\|_{L^{\infty}}\big)\nonumber\\
&\,+C\|\nabla^{3}h(\rho)\|_{L^{2}}\|\theta\|_{L^{\infty}}\|\nabla\rho\|_{L^{\infty}} +C\|g(\rho)\|_{L^{\infty}}\|\nabla^{5}u\|_{L^{2}}\nonumber\\
&\,+C\|\nabla^{3}g(\rho)\|_{L^{2}}\|\nabla^{2}u\|_{L^{\infty}}
+C\|g(\rho)\|_{L^{\infty}}\|\nabla^{3}n_{1}\|_{L^{2}}\nonumber\\
&\,+C\|\nabla^{3}g(\rho)\|_{L^{2}}\| n_{1}\|_{L^{\infty}}\nonumber\\
\leq&\,C\sigma\Big(\|(\rho,u,,\theta,n_1)\|_{\dot{B}_{2,2}^{3}}
+\|(\rho,u)\|_{\dot{B}_{2,2}^{4}}+\|u\|_{\dot{B}_{2,2}^{5}}\Big).
\end{align}
Similar to \eqref{G4.7}, we have 
\begin{align*}
\|(\mathcal{N}_{3})^{S}(t)\|_{\dot{B}_{2,2}^{3}}
\leq&\, C\sigma\Big(\|(\rho,\theta,n_0)\|_{\dot{B}_{2,2}^{3}}+\|(u,\theta)\|_{\dot{B}_{2,2}^{4}}
+\|\theta\|_{\dot{B}_{2,2}^{5}}\Big),\\
\|(\mathcal{N}_{4})^{S}(t)\|_{\dot{B}_{2,2}^{3}}
\leq
&\, C\sigma\|\theta\|_{\dot{B}_{2,2}^{4}},\\
 \|((\nabla u)^{T}\cdot\nabla\rho)^{S}(t)\|_{\dot{B}_{2,2}^{3}}
 \leq
 &\,C\|\nabla\rho\|_{L^{\infty}}\|\nabla^{4}u\|_{L^{2}}+C\|\nabla^{4}\rho\|_{L^{2}}\|\nabla u\|_{L^{\infty}}\nonumber\\
\leq&\,C\sigma\|(\rho,u)\|_{\dot{B}_{2,2}^{4}},\\
\|\nabla u\|_{H^{2}}\|(\nabla\rho)^{S}(t)\|_{\dot{B}_{2,2}^{3}}^{2}
\leq&\,C\sigma\|\rho\|_{\dot{B}_{2,2}^{4}}^{2}, \\
\|\nabla^{2}\rho\|_{H^{2}}^{2}\|u^{S}\|_{\dot{B}_{2,2}^{3}}^{2}
\leq&\,C\|\nabla^{2}\rho\|_{H^{2}}^{2}\|u\|_{\dot{B}_{2,2}^{3}}^{2}\nonumber\\
\leq&\, C\sigma\|u\|_{\dot{B}_{2,2}^{3}}^{2}.
\end{align*}
For the last term in \eqref{G4.3}, from Young's inequality for series convolution
and Lemma \ref{L2.1} and Lemma \ref{L2.7}, we deduce that
\begin{equation*}
\|\nabla u\|_{H^{2}}^{2}\sum_{k>k_{1}}2^{6k}\Big(\sum_{l\geq k-1}2^{k-l}\|\nabla\rho_{l}\|_{L^{2}}\Big)^{2}\leq C\|\nabla u\|_{H^{2}}^{2}\|\rho\|_{\dot{B}_{2,2}^{4}}^{2}C\sigma\|\rho\|_{\dot{B}_{2,2}^{4}}^{2}.  
\end{equation*}
Substituting all the above estimates into \eqref{G4.3} gives
\begin{align}\label{G4.14}
&\frac{\mathrm{d}}{\mathrm{d}t}\sum_{k>k_{1}}2^{6k}\mathcal{H}_{k}(t)+\|(\rho,u,\theta,n_{0},n_1)^{S}(t)\|_{\dot{B}_{2,2}^{3}}^{2}+\|(\rho,u,\theta,n_{0},n_1)^{S}(t)\|_{\dot{B}_{2,2}^{4}}^{2}+\|(u,\theta)^{S}(t)\|_{\dot{B}_{2,2}^{5}}^{2}\nonumber\\
&\quad \leq C\sigma\Big(\|(\rho,u,\theta,n_{0},n_1)(t)\|_{\dot{B}_{2,2}^{3}}^{2}+\|(\rho,u,\theta,n_{0},n_1)(t)\|_{\dot{B}_{2,2}^{4}}^{2}+\|(u,\theta)(t)\|_{\dot{B}_{2,2}^{5}}^{2}\Big).   
\end{align}
It follows from \eqref{G4.14} and Lemma \ref{L2.6} that
\begin{align*}
&\frac{\mathrm{d}}{\mathrm{d}t}\sum_{k>k_{1}}2^{6k}\mathcal{H}_{k}(t)+\frac{1}{2}\Big(\|(\rho,u,\theta,n_{0},n_1)^{S}(t)\|_{\dot{B}_{2,2}^{3}}^{2}+\|(\rho,u,\theta,n_{0},n_1)^{S}(t)\|_{\dot{B}_{2,2}^{4}}^{2}
+\|(u,\theta)^{S}(t)\|_{\dot{B}_{2,2}^{5}}^{2}\Big)\nonumber\\
&\qquad \leq C\sigma\Big(\|(\rho,u,\theta,n_{0},n_1)^{L+M}(t)\|_{\dot{B}_{2,2}^{3}}^{2}+\|(\rho,u,\theta,n_{0},n_1)^{L+M}(t)\|_{\dot{B}_{2,2}^{4}}^{2}+\|(u,\theta)^{L+M}(t)\|_{\dot{B}_{2,2}^{5}}^{2}\Big).   
\end{align*}
Notice that
\begin{align}\label{G4.16}
\sum_{k>k_{1}}2^{6k}\mathcal{H}_{k}(t)\sim\|(\rho,u,\theta,n_{0},n_1)^{S}(t)\|_{\dot{B}_{2,2}^{3}}^{2}+\|(\rho,u,\theta,n_{0},n_1)^{S}(t)\|_{\dot{B}_{2,2}^{4}}^{2},   
\end{align}
for all $0\leq t\leq T$.
Therefore, there exists a constant $\lambda_4>0$ such that
\begin{align*}
\frac{\mathrm{d}}{\mathrm{d}t}\sum_{k>k_{1}}2^{6k}\mathcal{H}_{k}(t)
+\lambda_4\sum_{k>k_{1}}2^{6k}\mathcal{H}_{k}(t) 
\leq \,& C\sigma\Big(\|(\rho,u,\theta,n_{0},n_1)^{L+M}(t)\|_{\dot{B}_{2,2}^{3}}^{2}\\
&   +\|(\rho,u,\theta,n_{0},n_1)^{L+M}(t)\|_{\dot{B}_{2,2}^{4}}^{2}+\|(u,\theta)^{L+M}(t)\|_{\dot{B}_{2,2}^{5}}^{2}\Big).    
\end{align*}
By Gronwall's inequality, we have
\begin{align}\label{G4.18}
\sum_{k>k_{1}}2^{6k}\mathcal{H}_{k}(t)\leq&\,C\mathrm{e}^{-\lambda_4t}\sum_{k>k_{1}}2^{6k}\mathcal{H}_{k}(0)+C\sigma\int_{0}^{t}\mathrm{e}^{-\lambda_4(t-\tau)}\|(\rho,u,\theta,n_{0},n_1)^{L+M}(\tau)\|_{B_{2,2}^{3}}^{2}\mathrm{d}\tau\nonumber\\
&\,+C\sigma\int_{0}^{t}\mathrm{e}^{-\lambda_4(t-\tau)}\Big(\|(\rho,u,\theta,n_{0},n_1)^{L+M}(\tau)\|_{\dot{B}_{2,2}^{4}}^{2}+\|(u,\theta)^{L+M}(\tau)\|_{B_{2,2}^{3}}^{2}\Big)\mathrm{d}\tau.   
\end{align}
Based on the fact
\begin{align*}\label{G4.19}
\sum _{k> k_{1}}2^{6k}{\mathcal{H} }_{ k}( 0) \sim \| ( \rho _{0}, u_{0}, \theta _{0}, n_{0}^{0},n_{1}^{0}) ^{S}\| _{{\dot{B} _{2, 2}^{3}}}+ \| ( \rho _{0}, u_{0}, \theta _{0}, n_{0}^{0},n_{1}^{0}) ^{S}\| _{{\dot{B} _{2, 2}^{4}}},    
\end{align*}
we can further get the desired \eqref{G4.1} through \eqref{G4.16} and \eqref{G4.18}.
\end{proof}
\subsection{Low-medium frequency part}
In this subsection, we shall establish the $L^2$-norm decay  estimates of classical solutions to 
the problem \eqref{I-3}--\eqref{I--3}. To this end, we apply the
homogeneous frequency localized operator $\dot{\Delta}_k(k\in\mathbb{Z})$ to \eqref{I-3}
to obtain the homogeneous problem as follows:
\begin{equation}\label{G4.20}
\left\{
\begin{aligned}
& \partial_t \overline{\mathbb{U}}_k+\mathbb{A}\overline{\mathbb{U}}_k=0, \\
& \overline{\mathbb{U}}_k|_{t=0}=\overline{\mathbb{U}}_k(0),
\end{aligned}\right.
\end{equation}
with $t>0$.
Here, we denote
\begin{align*}
\overline{\mathbb{U}}_k(t):=\,& (\overline{\rho}_{k}(t),\overline{u}_{k}(t),\overline{\theta}_{k}(t),
\overline{n}_{0,k}(t),\overline{n}_{1,k}(t))^\top,\nonumber \\
{\mathbb{U}}_k(0):=\,&({\rho}_{0}^k,{u}_{0}^k,{\theta}_{0}^k,{n}_0^{0,k},{n}_1^{0,k})^\top,
\end{align*}
and the linear operator $\mathbb{A}$ is defined as 
\begin{align*}
\mathbb{A}=\left(
\begin{matrix}
0 & {\rm div} & 0 & 0 & 0  \\
-\nabla & -\Delta-2\nabla{\rm div} & \nabla & 0 & -1   \\
0 & {\rm div} & -\Delta+4 & -1 & 0  \\
0 & 0 & -4 & 1 & {\rm div}  \\
0 & 0 & 0 & \nabla & 1  \\
\end{matrix}\right).
\end{align*}
By performing Fourier transform on \eqref{G4.20} at $x$, we infer that
\begin{align*}
\overline{\mathbb{U}}_{k}(t)=\mathscr{A}(t)\mathbb{U}_{k}(0),    
\end{align*}
where $\mathscr{A}(t):=e^{-t\mathbb{A}}$($t\geq 0$) is a semigroup
determined by $\mathbb{A}$ and $\mathcal{F}(\mathrm{A}(t)g):=e^{-t{\mathbb{A}}_{\xi}}\widehat{g}(\xi)$
with
\begin{equation*}\label{G4.23}
\mathbb{A}_{\xi}:=\left(
\begin{array}{ccccc}
0 & i\xi^\top & 0 & 0 & 0  \\
i\xi & |\xi|^2\delta_{ij}+2\xi_i\xi_j & i\xi & 0 & -1   \\
0 & i\xi^\top & |\xi|^2+4 & -1 & 0  \\
0 & 0 & -4 & 1 & i\xi^\top  \\
0 & 0 & 0 & i\xi & 1  \\
\end{array}\right).   
\end{equation*}
Next, we give the following estimate of smooth solutions for the low frequency part.
\begin{lem}\label{L4.2}
Let $1\leq p \leq 2$ and $m\geq 0$ be an integer. Then for any integer $k<k_0$,
it holds 
\begin{align}\label{G4.25}
\|\nabla^{m}\big(\mathscr{A}(t)\mathbb{U}_{k}(0))\|_{L^{2}}\leq C(1+t)^{-\frac{3}{2}(\frac{1}{p}-\frac{1}{2})-\frac{m}{2}}\|\mathbb{U}(0)\|_{L^{p}}.   
\end{align}
\end{lem}
\begin{proof}
First, we consider the homogeneous problem of \eqref{I-3}--\eqref{I--3}.
Applying  \eqref{G3.56} to it, we know that  there exists a constant $\lambda_5>0$ such that
\begin{align*}
\frac{\mathrm{d}}{\mathrm{d}t}\mathcal{L}_{l}(t,\xi)+\lambda_{5}|\xi|^{2}\mathcal{L}_{l}(t,\xi)\leq0,   
\end{align*}
for all $|\xi|\leq r_0$.
According to the Gronwall’s inequality, we have
\begin{align*}
 \mathcal{L}_{l}(t,\xi)\leq Ce^{-\lambda_5|\xi|^2t}  \mathcal{L}_{l}(0,\xi).  
\end{align*}
Combining \eqref{G3.55} and Plancherel theorem, we get
\begin{align*}
\|\partial_{x}^{\alpha}(\overline{\rho}_{k},\overline{d}_{k},\overline{\theta}_{k},\overline{{n}}_{0,k},\overline{{M}}_{k})(t)\|=&\,\|(i\xi)^{\alpha}(\widehat{\overline{\rho}}_{k},\widehat{\overline{d}}_{k},\widehat{\overline{\theta}}_{k},
\widehat{\overline{{n}}}_{0,k},\widehat{\overline{{M}}}_{k})\|_{L_{\xi}^{2}}\nonumber\\
\leq&\, C\Big(\int_{|\xi|\leq r_{0}}|\xi|^{2|\alpha|}\big|(\widehat{\overline{\rho}}_{k},\widehat{\overline{d}}_{k},\widehat{\overline{\theta}}_{k},\widehat{\overline{{n}}}_{0,k},\widehat{\overline{{M}}}_{k})(\xi,t)\big|^{2}\mathrm{d}\xi\Big)^{\frac{1}{2}}\nonumber\\
&\leq C\Big(\int_{|\xi|\leq r_{0}}|\xi|^{2|\alpha|}\mathrm{e}^{-\lambda_5|\xi|^{2}t}\big|(\widehat{\rho}_{k},\widehat{d}_{k},\widehat{\theta}_{k},\widehat{n}_{0,k},\widehat{M}_{k})(\xi,0)\big|^{2}\mathrm{d}\xi\Big)^{\frac{1}{2}},
\end{align*}
for any $k< k_0$.
Using  H\"{o}lder’s inequality and Haussdorff-Young's inequality, we consequently arrive at
\begin{align}\label{G4.29}
\|\partial_{x}^{\alpha}(\overline{\rho}_{k},\overline{d}_{k},\overline{\theta}_{k},\overline{{n}}_{0,k},\overline{{n}}_{1,k})(t)\|_{L^{2}}\leq&\,C\|(\rho,u,\theta,n_{0},n_1)(0)\|_{L^{p}}(1+t)^{-\frac{3}{2}(\frac{1}{p}-\frac{1}{2})-\frac{|\alpha|}{2}},  
\end{align}
for $1\leq p\leq 2$.
Similar to \eqref{G4.29}, by using \eqref{GBBB3.65}, we also have
\begin{align}\label{G4.30}
\|\partial_{x}^{\alpha}\big((\overline{{{\mathcal P}u}})_{k},(\overline{{{\mathcal P}n_1}})_{k}\big)(t)\|_{L^{2}}\leq C\|(u,n_1)(0)\|_{L^{p}}(1+t)^{-\frac{3}{2}(\frac{1}{p}-\frac{1}{2})-\frac{|\alpha|}{2}}.    
\end{align}
Finally, \eqref{G4.29} together with \eqref{G4.30} leads to \eqref{G4.25}.
\end{proof}
Based on the estimate in \eqref{G4.25} that we established,
we shall further estimate the nonlinear problem \eqref{I-3}--\eqref{I--3}.
Now, let's rewrite it as  
\begin{equation}\label{G4.31}
\left\{
\begin{aligned}
& \partial_t {\mathbb{U}}_k+\mathbb{A}{\mathbb{U}}_k=\mathcal{N}(\mathbb{U}_k), \\
& {\mathbb{U}}_k|_{t=0}={\mathbb{U}}_k(0),
\end{aligned}\right.
\end{equation}
with $t>0$.
Here, we denote 
\begin{align*}
\mathbb{U}_k(t):=&\,(\rho_k(t),u_k(t),\theta_k(t),n_{0,k}(t),n_{1,k}(t))^\top,\nonumber\\
\mathcal{N}(\mathbb{U}_k):=&\,(\mathcal{N}_1^k,\mathcal{N}_2^k,\mathcal{N}_3^k,\mathcal{N}_4^k)^\top.
\end{align*}
Through Duhamel's principle, \eqref{G4.31} can
be rewritten as follows:
\begin{align*}
\mathbb{U}_k(t)=\mathrm{A}(t)\mathbb{U}_k(0)+\int_0^t\mathrm{A}(t-\tau)\mathcal{N}(\mathbb{U}_k)(\tau)\mathrm{d}\tau,  
\end{align*}
where $\mathscr{A}(t)=e^{-t\mathbb{A}}$.
From Lemma \ref{L3.7} and Lemma \ref{L4.2},
we get the time-decay estimates on the following low-medium frequency part
of classical solutions to   \eqref{G4.31}.
\begin{lem}\label{L4.3}
For any integer $m\geq 0$, it holds that  
\begin{align*}
\|(\rho,u,\theta,n_{0},n_1)^{L+M}(t)\|_{\dot{B}_{2,2}^{m}}\leq&\,C(1+t)^{-\frac{3}{4}-\frac{m}{2}}\|(\rho,u,\theta,n_{0},n_1)(0)\|_{L^{1}}\nonumber\\
&\,+Ce^{-c_3t}\|(\rho,u,\theta,n_{0},n_1)(0)\|_{L^{2}}\nonumber\\
&\,+C\sigma\int_{0}^{\frac{t}{2}}(1+t-\tau)^{-\frac{3}{4}-\frac{m}{2}}\|\nabla(\rho,u,\theta)(\tau)\|_{L^{2}}\mathrm{d}\tau\nonumber\\
&+C\sigma\int_{0}^{\frac{t}{2}}(1+t-\tau)^{-\frac{3}{4}-\frac{m}{2}}\|\nabla^2(u,\theta)(\tau)\|_{L^{2}}\mathrm{d}\tau\nonumber\\
&\,+C\int_{0}^{\frac{t}{2}}(1+t-\tau)^{-\frac{3}{4}-\frac{m}{2}}\|(\rho,\theta,n
_0,n_1)(\tau)\|_{L^{2}}^2\mathrm{d}\tau\nonumber\\
&\,+C\int_{0}^{\frac{t}{2}}(1+t-\tau)^{-\frac{m}{2}}\|\mathcal{N}(\mathbb{U})(\tau)\|_{L^{2}}\mathrm{d}\tau\nonumber\\
&\,+C\int_{0}^{t}e^{-c_3(t-\tau)}\|\mathcal{N}(\mathbb{U})(\tau)\|_{L^{2}}\mathrm{d}\tau,
\end{align*}
with $\mathcal{N}(\mathbb{U}):=(\mathcal{N}_1,\mathcal{N}_2,\mathcal{N}_3,\mathcal{N}_4)$.
\end{lem}
\begin{proof}
It follows from Lemmas \ref{L3.7} and \ref{L4.2} that 
\begin{align}\label{G4.35}
\|\nabla^{m}\mathbb{U}_{k}(t)\|_{L^{2}}\leq&\,C(1+t)^{-\frac{3}{4}-\frac{m}{2}}\|\mathbb{U}(0)\|_{L^{1}}+Ce^{-c_3t}\|\mathbb{U}(0)\|_{L^{2}}\nonumber\\
&\,+C\int_{0}^{\frac{t}{2}}(1+t-\tau)^{-\frac{3}{4}-\frac{m}{2}}\|\mathcal{N}(\mathbb{U})(\tau)\|_{L^{1}}\mathrm{d}\tau\nonumber\\
&\,+C\int_{\frac{t}{2}}^{t}(1+t-\tau)^{-\frac{m}{2}}\|\mathcal{N}(\mathbb{U})(\tau)\|_{L^{2}}\mathrm{d}\tau\nonumber\\
&\,+C\int_{0}^{t}e^{-c_3(t-\tau)}\|\mathcal{N}(\mathbb{U})(\tau)\|_{L^{2}}\mathrm{d}\tau,
\end{align}
for any integer $k\leq k_1$.
From  H\"{o}lder’s inequality, it yields
\begin{align}\label{G4.36}
\|\mathcal{N}(\mathbb{U})(\tau)\|_{L^{1}}\leq&\,C\|(\rho,u,\theta,n_{0},n_1)(\tau)\|_{H^{1}}
\big(\|\nabla(\rho,u,\theta)(\tau)\|_{L^{2}}+\|\nabla^{2}(u,\theta)(\tau)\|_{L^{2}}\big)\nonumber\\
&+\|\theta\|_{L^{2}}^{2}\big(1+\|\theta\|_{L^{\infty}}+\|\theta\|_{L^{\infty}}^2\big)
+\|\rho\|_{L^{2}}\|(n_0,n_1)\|_{L^{2}}\nonumber\\
\leq&\,C\sigma\big(\|\nabla(\rho,u,\theta)(\tau)\|_{L^{2}}+\|\nabla^{2}(u,\theta)(\tau)
\|_{L^{2}}\big)+C\|(\rho,\theta,n_{0},n_1)\|_{L^{2}}^{2}.
\end{align}
Plugging \eqref{G4.36} into \eqref{G4.35} and taking a $l^2$ summation
over $k$ with $k\leq k_1$, we complete the proof of
Lemma \ref{L4.3}.
\end{proof}

\subsection{Optimal time-decay rates of classical solutions}
In this subsection, thanks to Lemmas \ref{L4.2}--\ref{L4.3}, we obtain
the time decay rates of classical solutions to the problem \eqref{I-3}--\eqref{I--3}.
\begin{lem}\label{L4.4}
Under the assumption of Theorem \ref{T1.2}, it holds that 
\begin{align}\label{G4.37}
 \|\nabla^{m}(\rho,u,\theta,n_{0},n_1)(t)\|_{L^{2}}\leq&\, C(1+t)^{-\frac{3}{4}-\frac{m}{2}}, 
 \end{align}
 for $m=0,1,2$.  And
 \begin{align}\label{G4.38}
 \|\nabla^{m}(\rho,u,\theta,n_{0},n_1)(t)\|_{L^{2}}\leq&\, C(1+t)^{-\frac{7}{4}},
\end{align}
for $m=3,4$.
\end{lem}
\begin{proof}
We denote
\begin{align}\label{G4.39}
\mathcal{E}(t):=\sup\limits_{0\leq\tau\leq t}\sum\limits_{m=0}^{2}(1+\tau)^{\frac{3}{4}+\frac{m}{2}}\|\nabla^{m}(\rho,u,\theta,n_{0},n_1)(\tau)\|_{L^{2}}.    
\end{align}
It is obvious that 
\begin{align*}
\|\nabla^{m}(\rho,u,\theta,n_{0},n_1)(\tau)\|_{L^{2}}^{2}\leq C(1+\tau)^{-\frac{3}{4}-\frac{m}{2}}\mathcal{E}(t),  
\end{align*}
for $0\leq\tau\leq t$ and $m=0,1,2$.

By direct calculation, we have
\begin{align*}\label{G4.41}
\|\mathcal{N}(\mathbb{U})\|_{L^{2}}\leq&\,C\|(\rho,u,\theta)\|_{L^{\infty}}\|\nabla(\rho,u,\theta)\|_{L^{2}}+C\|\rho\|_{L^{\infty}}\|\nabla^{2}(u,\theta)\|_{L^{2}}\nonumber\\
&\,+C\|\nabla u\|_{L^{\infty}}\|\nabla u\|_{L^{2}}+C\big(1+\|\rho\|_{L^{\infty}}\big)\|\theta\|_{L^{2}}
\big(\|\theta\|_{L^{\infty}}+\|\theta\|_{L^{\infty}}^2+\|\theta\|_{L^{\infty}}^3\big)\nonumber\\
&\,+C\|\rho\|_{L^{\infty}}\|(\theta,n_{0},n_1)\|_{L^{2}}\nonumber\\
\leq&\,C\|\nabla(\rho,u,\theta)\|_{H^{1}}\|\nabla(\rho,u,\theta)\|_{H^{1}}+C\|\nabla\rho\|_{H^{1}}\|\nabla^{2}(u,\theta)\|_{L^{2}}\nonumber\\
&\,+C\|\nabla^{2}u\|_{H^{1}}\|\nabla u\|_{L^{2}}+C\|\nabla(\rho,\theta)\|_{H^{1}}\|(\theta,n_{0},n_1)\|_{L^{2}}.    
\end{align*}
Combining Lemma \ref{L4.3}, \eqref{G4.31}, and \eqref{G4.39}, we get, for $m\geq 0$,
\begin{equation*}
\begin{split}
\|(\rho,u,\theta,n_{0},n_1)^{L+M}(t)\|_{\dot{B}_{2,2}^{m}}\leq&\,C(1+t)^{-\frac{3}{4}-\frac{m}{2}}\|(\rho,u,\theta,n_{0},n_1)(0)\|_{L^{1}}\nonumber\\
&\,+Ce^{-c_3t}\|(\rho,u,\theta,n_{0},n_1)(0)\|_{L^{2}}\nonumber\\
&\,+C\sigma \mathcal{E}(t)\int_{0}^{\frac{1}{2}}(1+t-\tau)^{-\frac{3}{4}-\frac{m}{2}}(1+\tau)^{-\frac{5}{4}}\mathrm{d}\tau\nonumber\\
&\,+C\sigma \mathcal{E}(t)\int_{0}^{\frac{t}{2}}(1+t-\tau)^{-\frac{3}{4}-\frac{m}{2}}(1+\tau)^{-\frac{7}{4}}\mathrm{d}\tau\nonumber\\
&\,+C\sigma^{\frac{1}{4}}\mathcal{E}^{\frac{7}{4}}(t)\int_{0}^{\frac{t}{2}}(1+t-\tau)^{-\frac{3}{4}-\frac{m}{2}}(1+\tau)^{-\frac{21}{16}}\mathrm{d}\tau\nonumber\\
&+C\sigma^{\frac{1}{4}}\mathcal{E}^{\frac{7}{4}}(t)\int_{\frac{t}{2}}^{t}(1+t-\tau)^{-\frac{m}{2}}(1+\tau)^{-\frac{29}{16}}\mathrm{d}\tau\nonumber\\
&+C\sigma^{\frac{1}{4}}\mathcal{E}^{\frac{7}{4}}(t)\int_{0}^{t}e^{-c_3(t-\tau)}(1+\tau)^{-\frac{29}{16}}\mathrm{d}\tau.
\end{split}
\end{equation*}
From Lemma \ref{L2.3}, we directly get
\begin{align}\label{G4.43}
&\|(\rho,u,\theta,n_{0},n_1)^{L+M}(t)\|_{\dot{B}_{2,2}^{m}}\nonumber\\
& \quad \leq C(1+t)^{-\frac{3}{4}-\frac{m}{2}}\Big(\|(\rho,u,\theta,n_{0},n_1)(0)\|_{L^{1}}+\sigma\mathcal{E}(t)+\sigma^{\frac{1}{4}}\mathcal{E}^{\frac{7}{4}}(t)        \Big),
\end{align}
for $m=0,1,2$,
and
\begin{align}\label{G4.44}
&\|(\rho,u,\theta,n_{0},n_1)^{L+M}(t)\|_{\dot{B}_{2,2}^{m}}\nonumber\\
&\quad \leq C(1+t)^{-\frac{7}{4}}\Big(\|(\rho,u,\theta,n_{0},n_1)(0)\|_{L^{1}}+\sigma\mathcal{E}(t)+\sigma^{\frac{1}{4}}\mathcal{E}^{\frac{7}{4}}(t)        \Big),
\end{align}
for $m=3,4$.
This together with Proposition \ref{P4.1} implies that 
\begin{align}\label{G4.45}
&\|(\rho,u,\theta,n_{0},n_1)^{S}(t)\|_{\dot{B}_{2,2}^{3}}^{2}+\|(\rho,u,\theta,n_{0},n_1)^{S}(t)\|_{\dot{B}_{2,2}^{4}}^{2}\nonumber\\
\leq&\, C\mathrm{e}^{-c_{3}t}\Big(\|(\rho,u,\theta,n_{0},n_1)^{S}(0)\|_{\dot{B}_{2,2}^{3}}^{2}+\|(\rho,u,\theta,n_{0},n_1)^{S}(0)\|_{\dot{B}_{2,2}^{4}}^{2}\Big)\nonumber\\
&\,+C\sigma \Big(\|(\rho,u,\theta,n_{0},n_1)(0)\|_{L^{1}\cap L^{2}}+\sigma \mathcal{E}(t)+\sigma^{\frac{1}{4}}\mathcal{E}^{\frac{7}{4}}(t)\Big)^{2}\int_{0}^{t}\mathrm{e}^{-c_{3}(t-\tau)}(1+\tau)^{-\frac{7}{2}}\mathrm{d}\tau\nonumber\\
\leq&\,C\mathrm{e}^{-c_{3}t}\Big(\|(\rho,u,\theta,n_{0},n_1)^{S}(0)\|_{\dot{B}_{2,2}^{3}}^{2}+\|(\rho,u,\theta,n_{0},n_1)^{S}(0)\|_{\dot{B}_{2,2}^{4}}^{2}\Big)\nonumber\\
&\,+C\sigma (1+t)^{-\frac{7}{2}}\Big(\|(\rho,u,\theta,n_{0},n_1)(0)\|_{L^{1}\cap L^2}+\sigma \mathcal{E}(t)+\sigma^{\frac{1}{4}}\mathcal{E}^{\frac{7}{4}}(t)\Big)^{2}.    
\end{align}
By using Lemma \ref{L2.6}, we have
\begin{align}\label{G4.46}
\|\nabla^{m}(\rho,u,\theta,n_{0},n_1)(t)\|_{L^{2}}\leq\,&C\|(\rho,u,\theta,n_{0},n_1)^{L+M}(t)\|_{\dot{B}_{2,2}^{m}}+\|(\rho,u,\theta,n_{0},n_1)^{S}(t)\|_{\dot{B}_{2,2}^{m}}\nonumber\\
\leq&\,C\|(\rho,u,\theta,n_{0},n_1)^{L+M}(t)\|_{\dot{B}_{2,2}^{m}}+\|(\rho,u,\theta,n_{0},n_1)^{S}(t)\|_{\dot{B}_{2,2}^{3}},  
\end{align}
for $m=0,1,2$.
Combing \eqref{G4.43}--\eqref{G4.46}, Young's inequality and the definition of $\mathcal{E}(t)$,
we arrive at
\begin{align}\label{G4.47}
\mathcal{E}(t)\leq&\,C\Big(\|(\rho,u,\theta,n_{0},n_1)(0)\|_{L^{1}\cap H^{4}}+\sigma^{\frac{1}{4}}\mathcal{E}^{\frac{7}{4}}(t)\Big)\nonumber\\
\leq&\,C\Big(\|(\rho,u,\theta,n_{0},n_1)(0)\|_{L^{1}\cap H^{4}}+1\Big)+\sigma^{\frac{2}{7}}\mathcal{E}^{2}(t)\nonumber\\
=:&\,\tilde{C}_1+\sigma^{\frac{2}{7}}\mathcal{E}^{2}(t).
\end{align}
Now, we claim that $\mathcal{E}(t)\leq  \hat C$ for some constant $\hat C>0$. Assume that there exits a
constant $t_0>0$ such that $\mathcal{E}(t_0)>2\tilde{C}_1$. Under the supposition
of Theorem \ref{T1.2}, $\mathcal{E}(0)=\|(\rho_{0},u_{0},\theta_{0},n_{0}^{0},n_{1}^{0})\|_{H^{2}}$ is small enough.
And notice that the fact $\mathcal{E}(t)\in C[0,+\infty)$, which means that 
there exists $t_1\in(0, t_0)$  such that $\mathcal{E}(t_1)=2\tilde{C}_1$.
It follows from \eqref{G4.47} that
\begin{align*}
\mathcal{E}(t_{1})\leq \tilde{C}_{1}+\sigma^{\frac{2}{7}}\mathcal{E}^{2}(t_{1}),
\end{align*}
and hence
\begin{align}\label{G4.49}
\mathcal{E}(t_{1})\leq \frac{\tilde{C}_{1}}{1-\sigma^{\frac{2}{7}}\mathcal{E}(t_{1})}.
\end{align}
With help of the smallness of $\sigma$, we suitably choose $\sigma^{\frac{2}{7}}\mathcal{E}(t_{1})<\frac{1}{4\tilde{C}_{1}}$.
Then, by utilizing \eqref{G4.49}, we conclude that
\begin{align*}
\mathcal{E}(t_{1})<2\tilde{C}_{1}.  
\end{align*}
This leads to a contradiction with our assumption $\mathcal{E}(t_1)=2\tilde{C}_1$.
Due to the continuity of $\mathcal{E}(t)$, we consequently get
$\mathcal{E}(t)\leq  \hat C$ for any $t\geq 0$. From the definition of
$\mathcal{E}(t)$, we derive \eqref{G4.37}. By the similar argument
for $m=3,4$, we easily get \eqref{G4.38} through \eqref{G4.44} and
Lemma \ref{L2.6}. Thus, we complete the proof of Lemma \ref{L4.4}.
\end{proof}
\begin{proof}[Proof of Theorem \ref{T1.2}]
With Lemma \ref{L4.4} in hand, we continue to prove Theorem \ref{T1.2}.
By a directly calculation, we have
\begin{align*}
\|\partial_{t} \rho (t)\|_{L^{2}}\leq&\, C\|\mathrm{div}u(t)\|_{L^{2}}+\|\mathcal{N}_{1}(t)\|_{L^{2}}\nonumber\\
\leq&\,C\|\nabla u(t)\|_{L^{2}}+\|\rho(t)\|_{L^{\infty}}\|\nabla u(t)\|_{L^{2}}+\|u(t)\|_{L^{\infty}}\|\nabla \rho(t)\|_{L^{2}}\nonumber\\
\leq&\,C(1+t)^{-\frac{5}{4}},    
\end{align*}
and
\begin{align*}
\|\partial_{t}(u,\theta,n_{0},n_1)(t)\|_{L^{2}}\leq&\,C\|\nabla(u,\theta)(t)\|_{L^{2}}+C\|(\theta,n_0,n_1)\|_{L^{2}}+\|\nabla(u,n_0,n_{1})(t)\|_{L^{2}}\nonumber\\
&+C\|(\mathcal{N}_2,\mathcal{N}_3,\mathcal{N}_4)(t)\|_{L^{2}}\nonumber\\
\leq&\,C(1+t)^{-\frac{3}{4}}.   
\end{align*}
Finally,
  we give the
proof of \eqref{G1.10}--\eqref{GB1.10}. 
From the result of \eqref{G1.9} and Lemma \ref{L2.1}, we obtain
\begin{align}\label{G4.74}
&\|(\rho,u,\theta,n_0,n_1)\|_{L^6}\leq C\|\nabla (\rho,u,\theta,n_0,n_1)\|_{L^2}\leq C(1+t)^{-\frac{5}{4}},\\\label{G4.75} 
&\|(\rho,u,\theta,n_0,n_1)\|_{L^2}\leq C(1+t)^{-\frac{3}{4}},\\ \label{GB4.76}
&\|\nabla(\rho,u,\theta,n_0,n_1)\|_{L^6}\leq C\|\Delta (\rho,u,\theta,n_0,n_1)\|_{L^2}\leq C(1+t)^{-\frac{7}{4}}.
\end{align}
For $p\in [2,6]$, using Lemma \ref{L2.4}, it follows from \eqref{G4.74}--\eqref{G4.75} that
\begin{align*}
&\|(\rho,u,\theta,n_0,n_1)\|_{L^p}\leq C\|(\rho,u,\theta,n_0,n_1)\|_{L^2}^{\zeta}\|(\rho,u,\theta,n_0,n_1)\|_{L^6}^{{1-\zeta}}\leq C(1+t)^{-\frac{3}{2}(1-\frac{1}{p})},
\end{align*}
where $\zeta= ({6-p})/ {2p} \in [0,1]$.

By using  Lemma \ref{L2.1}, we obtain
\begin{align}\label{G4.77}
\|(\rho,u,\theta,n_0,n_1)\|_{L^{\infty}}\leq\,& C\|\nabla (\rho,u,\theta,n_0,n_1)\|_{L^{2}}^{\frac{1}{2}}\|\Delta (\rho,u,\theta,n_0,n_1)\|_{L^{2}}^{\frac{1}{2}}\nonumber\\
\leq\,& C(1+t)^{-\frac{3}{2}}.
\end{align}
For $p\in [6,\infty]$, using Lemma \ref{L2.4} again, it follows from \eqref{G4.74} and \eqref{G4.77} that
\begin{align*}
&\|(\rho,u,\theta,n_0,n_1)\|_{L^p}\leq C\|(\rho,u,\theta,n_0,n_1)\|_{L^6}^{\zeta^{\prime}}\|(\rho,u,\theta,n_0,n_1)\|_{L^{\infty}}^{{1-\zeta^{\prime}}}\leq C(1+t)^{-\frac{3}{2}(1-\frac{1}{p})},
\end{align*}
where $\zeta^{\prime}= {6}/{p} \in [0,1]$.
For $p\in [2,6]$, using Lemma \ref{L2.1}, it follows from \eqref{GB4.76} that
\begin{align*}
&\|\nabla(\rho,u,\theta,n_0,n_1)\|_{L^p}\leq C\|\nabla(\rho,u,\theta,n_0,n_1)\|_{L^2}^{\eta}\|\nabla(\rho,u,\theta,n_0,n_1)\|_{L^6}^{{1-\eta}}\leq C(1+t)^{-\frac{3}{2}(\frac{4}{3}-\frac{1}{p})},
\end{align*}
where $\eta= ({6-p})/ {2p} \in [0,1]$.
Therefore, we complete the proof of Theorem \ref{T1.2}.
\end{proof}

\bigskip 
{\bf Acknowledgements:} Jiang  is supported by NSF of Jiangsu Province 
(Grant No. \linebreak BK20191296).  Li and   Ni  are supported by NSFC (Grant Nos. 12331007, 12071212).  
And Li is also supported by the ``333 Project" of Jiangsu Province.

\bibliographystyle{plain}

\end{document}